\documentclass[11pt,reqno]{amsart}

\usepackage{amsfonts,amssymb,amsmath,gensymb,epic,eepic,float,fullpage}
\usepackage{longtable}
\usepackage[usenames,dvipsnames]{color}
\usepackage{siunitx}
\usepackage{empheq}
\usepackage[mathscr]{euscript}
\usepackage[T1]{fontenc}
\usepackage{mathtools}
\usepackage{comment}
\usepackage{bm}
\usepackage{caption}
\usepackage{hyperref}
\usepackage{enumerate,enumitem}
\usepackage{tikz}
\usetikzlibrary{graphs,patterns,decorations.markings,arrows,matrix}
\usetikzlibrary{calc,decorations.pathmorphing,decorations.pathreplacing,shapes}
\allowdisplaybreaks

\newtheorem{thm}{Theorem}[section]
\newtheorem{lem}[thm]{Lemma}
\newtheorem{prop}[thm]{Proposition}

\newtheorem{cor}[thm]{Corollary}

\theoremstyle{definition}
\newtheorem{defn}[thm]{Definition}

\theoremstyle{remark}
\newtheorem{rmk}[thm]{Remark}

\theoremstyle{remark}
\newtheorem{ex}[thm]{Example}

\def\wh{\widehat}

\def\leq{\leqslant}
\def\geq{\geqslant}

\def\l{\lambda}
\def\m{\mu}
\def\n{\nu}

\def\ie{{\it i.e.}\/,}

\colorlet{lgray}{white!85!black}
\colorlet{lred}{white!85!red}
\colorlet{lgreen}{white!80!green}
\colorlet{dgreen}{black!30!green}
\definecolor{green}{rgb}{0.1,0.8,0.1}
\definecolor{yellow}{rgb}{1.0,0.85,0.25}

\newcommand{\bra}[1]{\left\langle #1\right|}
\newcommand{\ket}[1]{\left|#1\right\rangle}

\renewcommand{\tikz}[2]{
\begin{tikzpicture}[scale=#1,baseline=(current bounding box.center),>=stealth]
#2
\end{tikzpicture}}

\renewcommand{\vert}[7]{
W_{#1,#2} \left(#3; 
\tikz{0.3}{
\draw[lgray,line width=1pt,->] (-1,0) -- (1,0);
\draw[lgray,line width=1pt,->] (0,-1) -- (0,1);
\node[left] at (-1,0) {\tiny $#5$};\node[right] at (1,0) {\tiny $#7$};
\node[below] at (0,-1) {\tiny $#4$};\node[above] at (0,1) {\tiny $#6$};
}\right)}

\newcommand{\vertt}[8]{
	W_{#1,#2}^{(#3)} \left(#4; 
	\tikz{0.3}{
		\draw[lgray,line width=1pt,->] (-1,0) -- (1,0);
		\draw[lgray,line width=1pt,->] (0,-1) -- (0,1);
		\node[left] at (-1,0) {\tiny $#6$};\node[right] at (1,0) {\tiny $#8$};
		\node[below] at (0,-1) {\tiny $#5$};\node[above] at (0,1) {\tiny $#7$};
	}\right)}

\def\B{\mathcal{B}}
\def\C{\mathcal{C}}
\def\F{{\sf F}}
\def\G{{\sf G}}

\def\I{\bm{I}}
\def\J{\bm{J}}
\def\K{\bm{K}}

\def\AA{\bm{A}}
\def\BB{\bm{B}}
\def\CC{\bm{C}}
\def\DD{\bm{D}}
\def\l{{\sf L}}
\def\m{{\sf M}}
\def\n{{\sf N}}

\newcommand{\Is}[2]{\I_{[#1,#2]}}

\renewcommand\le{\leq}
\renewcommand\ge{\geq}

\def\stoch{{\rm stoch}}
\def\cG{\mathcal{G}}
\def\P{\bm{P}}
\def\O{\mathcal{O}}
\def\E{\mathbb{E}}
\def\fm{\mathfrak{m}}
\def\fn{\mathfrak{n}}
\def\fx{\mathfrak{x}}
\def\fk{\mathfrak{k}}
\def\fl{\mathfrak{l}}
\def\ff{\mathfrak{f}}
\def\fp{\mathfrak{p}}
\def\fe{\mathfrak{e}}
\def\aa{\bm{\alpha}}
\def\bb{\bm{\beta}}
\def\cc{\bm{\gamma}}
\def\dd{\bm{\delta}}
\def\z{\mathfrak Z}

\numberwithin{equation}{section}

\makeindex

\begin{document}

\title{Observables of coloured stochastic vertex models\\ and their polymer limits}

\author{Alexei Borodin and Michael Wheeler}

\address[Alexei Borodin]{ Department of Mathematics, MIT, Cambridge, USA, and
Institute for Information Transmission Problems, Moscow, Russia. E-mail: borodin@math.mit.edu }

\address[Michael Wheeler]{ School of Mathematics and Statistics, The University of Melbourne, Parkville,
Victoria, Australia. E-mail: wheelerm@unimelb.edu.au}

\begin{abstract} In the context of the coloured stochastic vertex model in a quadrant, we identify a family of observables whose averages are given by explicit contour integrals. The observables are certain linear combinations of $q$-moments of the coloured height functions of the model. In a polymer limit, this yields integral representations for moments of partition functions of strict-weak, semi-discrete Brownian, and continuum Brownian polymers with varying beginning and ending points of the polymers.   
\end{abstract}

\maketitle

\setcounter{tocdepth}{1}
\makeatletter
\def\l@subsection{\@tocline{2}{0pt}{2.5pc}{5pc}{}}
\makeatother
\tableofcontents

\section{Introduction}
\label{sec:intro}

At least since the work of Kardar \cite{Kardar} in 1987, moments of polymer 
partition functions and related quantities have been an indispensable tool in 
analyzing models from the so-called Kardar-Parisi-Zhang (KPZ) universality class 
in (1+1) dimensions \cite{KPZ}, such as various (integrable) models of directed 
polymers in a random environment, exclusion and zero-range processes, and random 
growth models. An advanced (although non-rigorous) replica analysis of moments 
allowed Calabrese-Le Doussal-Rosso \cite{CalabreseDR} and Dotsenko 
\cite{Dotsenko} to derive the long-time asymptotics of the 
so-called narrow wedge solution to the KPZ equation. An alternative approach to 
the asymptotics of this solution based on the pioneering work of Tracy-Widom 
\cite{TracyW1, TracyW2, TracyW3} was developed rigorously by Amir-Corwin-Quastel 
\cite{ACQ} and in the physics literature by Sasamoto-Spohn 
\cite{SS}.\footnote{The two approaches were mostly reconciled by 
Borodin-Corwin-Sasamoto \cite{BorodinCS}.} The moment method became rigorous 
with appearance of explicit integral representations for $q$-moments of 
suitable $q$-deformed models together with asymptotic analysis of their 
generating functions in Borodin-Corwin \cite{BorodinC}.\footnote{A simpler 
approach to the asymptotics through moments was later suggested in 
\cite{Borodin2}, \cite{BorodinO}.}

Dozens of papers with analysis of $q$-moments of various integrable 
probabilistic systems have been written since then. The stochastic six vertex 
model, first introduced by Gwa-Spohn in 1992 \cite{GwaS}, also joined in via the 
work of Borodin-Corwin-Gorin \cite{BorodinCG}, see also 
Borodin-Corwin-Petrov-Sasamoto \cite{BorodinCPS, BorodinCPS2}, Corwin-Petrov 
\cite{CorwinP}, Borodin-Petrov \cite{BorodinP1,BorodinP2} for various 
approaches to the $q$-moments of this model.    

In 2016, Kuniba-Mangazeev-Maruyama-Okado \cite{KunibaMMO} introduced 
Yang-Baxter integrable \emph{coloured} stochastic vertex models, see also 
Bosnjak-Mangazeev \cite{BosnjakM}, Aggarwal-Borodin-Bufetov \cite{AggarwalBB}. 
Colours correspond to the simple roots in the underlying quantum affine algebra 
$U_q(\widehat{\mathfrak{sl}_{n+1}})$, with $n=1$ corresponding to the 
\emph{colourless} or rank-1 case considered previously.

Our recent paper \cite{BW} offered an extensive algebraic analysis of these 
coloured models and uncovered certain distributional correspondences between 
coloured and (much more studied) colourless ones. Such correspondences were 
further extended in Borodin-Bufetov \cite{BorodinBufetov}, 
Borodin-Gorin-Wheeler \cite{BGW}, and they gave access to various unknown 
marginals of the coloured models. However, so far no explicit formulas for 
observables of the coloured models have been found, apart from those that arise 
through matching with colourless models. The primary goal of this paper is to 
remedy this fact.

We obtain explicit integral representations for certain linear combinations of 
$q$-moments of coloured height functions for the coloured stochastic vertex 
model in a quadrant. Further, following a path worked out in \cite{BGW}, we 
degenerate the (fully fused) coloured vertex to directed polymers, thus 
obtaining formulas for  joint moments of polymer partition functions with 
different starting points in the same noise field. The limiting objects include 
the KPZ equation (equivalently, the continuum Brownian polymer), the 
O'Connell-Yor semi-discrete Brownian polymer \cite{OY}, the strict-weak or Gamma 
polymer of Corwin-Sepp\"al\"ainen-Shen \cite{CSS} and O'Connell-Orthmann 
\cite{OCO}, and the Beta polymer of Barraquand-Corwin \cite{BarCor}.

Let us describe our results in more detail.

The coloured stochastic vertex model in a quadrant can be viewed as a Markovian 
recipe of constructing random coloured up-right paths in $\mathbb Z_{\ge 
1}\times\mathbb Z_{\ge 1}$ with the colours labeled by natural numbers. Let us 
also initially assume that no horizontal edge of the lattice can be occupied by 
more than a single path; this restriction will eventually be removed.
The model depends on a quantization parameter $q\in \mathbb C$, spin parameter  
$s\in\mathbb C$, and row \emph{rapidities} denoted by $x_1,x_2,\ldots\in\mathbb 
C$. 

Along the boundary of the quadrant, we demand that no paths enter the 
quadrant from the bottom. On the other hand, a single coloured path enters the 
quadrant from the left in each row. We assume that the colours of the paths 
entering on the left are weakly increasing in the upward direction, and denote 
by $\lambda_1\ge 0$ the number of paths of colour 1, by $\lambda_2\ge 0$ the 
number of paths of colour 2, \emph{etc}. Let us also denote by 
$\ell_k=\lambda_1+\dots+\lambda_k$, $k\ge 1$, the partial sums of this sequence and 
also set $\ell_0=0$.

Once the paths are specified along the boundary, they progress in the up-right 
direction within the quadrant using certain interaction probabilities, also 
known as \emph{vertex weights}.
For each vertex of the lattice, once we know the colours of the entering paths 
along the bottom and left adjacent edges, we decide on the colors of the exiting 
paths along the top and right edges according to those 
probabilities.\footnote{Such decisions at different vertices are independent.} 
They are given by the table \eqref{intro:s-weights}, where it is assumed that 
$x$ is the rapidity of the row to which the vertex belongs, $1\le i<j$, 
$\bm{I}=(I_1,I_2,\dots)$ denotes a vector whose coordinates $I_k$ are equal to 
the number of paths of colour $k$ that enter the vertex from the bottom, and
$\bm{I_i^{\pm}}=\bm{I}\pm \bm{e}_i$ with $\bm{e_i}$ being the standard basis 
vector with 1 as its $i$th coordinate and all other coordinates equal to zero. 
We also use the notations $\bm{I}_{ab}^{+-}=\bm{I}+\bm{e}_a-\bm{e}_b$, and 
$\bm{I}_{\ge a}=I_a+I_{a+1}+\cdots$. The weight of any vertex that does not fall 
into one of the six categories in \eqref{intro:s-weights} is set to zero.
See Section \ref{sec:prelim} for the origin of these weights.

\begin{align}
\label{intro:s-weights}
\begin{tabular}{|c|c|c|}
\hline
\quad
\tikz{0.7}{
	\draw[lgray,line width=1.5pt,->] (-1,0) -- (1,0);
	\draw[lgray,line width=4pt,->] (0,-1) -- (0,1);
	\node[left] at (-1,0) {\tiny $0$};\node[right] at (1,0) {\tiny $0$};
	\node[below] at (0,-1) {\tiny $\I$};\node[above] at (0,1) {\tiny $\I$};
}
\quad
&
\quad
\tikz{0.7}{
	\draw[lgray,line width=1.5pt,->] (-1,0) -- (1,0);
	\draw[lgray,line width=4pt,->] (0,-1) -- (0,1);
	\node[left] at (-1,0) {\tiny $i$};\node[right] at (1,0) {\tiny $i$};
	\node[below] at (0,-1) {\tiny $\I$};\node[above] at (0,1) {\tiny $\I$};
}
\quad
&
\quad
\tikz{0.7}{
	\draw[lgray,line width=1.5pt,->] (-1,0) -- (1,0);
	\draw[lgray,line width=4pt,->] (0,-1) -- (0,1);
	\node[left] at (-1,0) {\tiny $0$};\node[right] at (1,0) {\tiny $i$};
	\node[below] at (0,-1) {\tiny $\I$};\node[above] at (0,1) {\tiny $\I^{-}_i$};
}
\quad
\\[1.3cm]
\quad
$\dfrac{1-s x q^{\bm{I}_{\ge 1}}}{1-sx}$
\quad
& 
\quad
$\dfrac{s(sq^{I_i}-x) q^{\bm{I}_{\ge i+1}}}{1-sx}$
\quad
& 
\quad
$\dfrac{sx(q^{I_i}-1) q^{\bm{I}_{\ge i+1}}}{1-sx}$
\quad
\\[0.7cm]
\hline
\quad
\tikz{0.7}{
	\draw[lgray,line width=1.5pt,->] (-1,0) -- (1,0);
	\draw[lgray,line width=4pt,->] (0,-1) -- (0,1);
	\node[left] at (-1,0) {\tiny $i$};\node[right] at (1,0) {\tiny $0$};
	\node[below] at (0,-1) {\tiny $\I$};\node[above] at (0,1) {\tiny $\I^{+}_i$};
}
\quad
&
\quad
\tikz{0.7}{
	\draw[lgray,line width=1.5pt,->] (-1,0) -- (1,0);
	\draw[lgray,line width=4pt,->] (0,-1) -- (0,1);
	\node[left] at (-1,0) {\tiny $i$};\node[right] at (1,0) {\tiny $j$};
	\node[below] at (0,-1) {\tiny $\I$};\node[above] at (0,1) 
	{\tiny $\I^{+-}_{ij}$};
}
\quad
&
\quad
\tikz{0.7}{
	\draw[lgray,line width=1.5pt,->] (-1,0) -- (1,0);
	\draw[lgray,line width=4pt,->] (0,-1) -- (0,1);
	\node[left] at (-1,0) {\tiny $j$};\node[right] at (1,0) {\tiny $i$};
	\node[below] at (0,-1) {\tiny $\I$};\node[above] at (0,1) {\tiny $\I^{+-}_{ji}$};
}
\quad
\\[1.3cm] 
\quad
$\dfrac{1-s^2 q^{\bm{I}_{\ge 1}}}{1-sx}$
\quad
& 
\quad
$\dfrac{sx(q^{I_j}-1) q^{\bm{I}_{\ge j+1}}}{1-sx}$
\quad
&
\quad
$\dfrac{s^2(q^{I_i}-1)q^{\bm{I}_{\ge i+1}}}{1-sx}$
\quad
\\[0.7cm]
\hline
\end{tabular} 
\end{align}

The \emph{stochasticity} of the weights encodes the fact that these weights add 
up to 1 when summed over all possible states of the outgoing edges with the 
states of incoming edges being fixed. When the parameters of the model are such 
that all the weights are nonnegative, we obtain \emph{bona fide} transition 
probabilities. However, if we agree to deal with complex-valued discrete 
distributions, we will not actually need the positivity assumption for our main 
algebraic result.   

Let us now fix $n\ge 1$ and focus on the state of the model between row $n$ and 
row $n+1$.  
That is, let us record the locations where the paths of colours $1,2,\dots$ 
exit the $n$th row upwards as an $n$-dimensional coloured vector (equivalently, 
a \emph{coloured composition}) $\nu$ with coordinates in $\mathbb{Z}_{\ge 1}$.
By colour conservation that our weights observe, the counts of different 
colours in such a vector are provided by a (finite) sequence 
$\lambda=(\lambda_1,\lambda_2,\dots)$ corresponding to the colours that enter via the left edges of the first $n$ rows; we 
call $\lambda$ the \emph{colouring composition}.\footnote{For convenience of 
notation, we assume that the left-incoming colours in rows $n$ and $n+1$ are 
different, so that we do not need to split the last coordinate in this 
sequence.} We will write the coordinates of $\nu$ as 
\begin{equation}
\label{eq:coloured_blocks}
\nu 
= 
\bigl(\nu_1 \geq \cdots \geq \nu_{\ell_1}\, |\,
\nu_{\ell_1+1} \geq \cdots \geq \nu_{\ell_2} \,|\,\nu_{\ell_2+1} \geq \cdots \geq \nu_{\ell_3} \,|\,\cdots\bigr),
\end{equation}
where the groups separated by vertical bars list the coordinates of the paths between rows $n$ and $n+1$ of colour 1, colour 2, \emph{etc}. See Figure \ref{Fig_coloured} for an example. 

\begin{figure}[t]
	\begin{center}
		{\scalebox{0.9}{\includegraphics{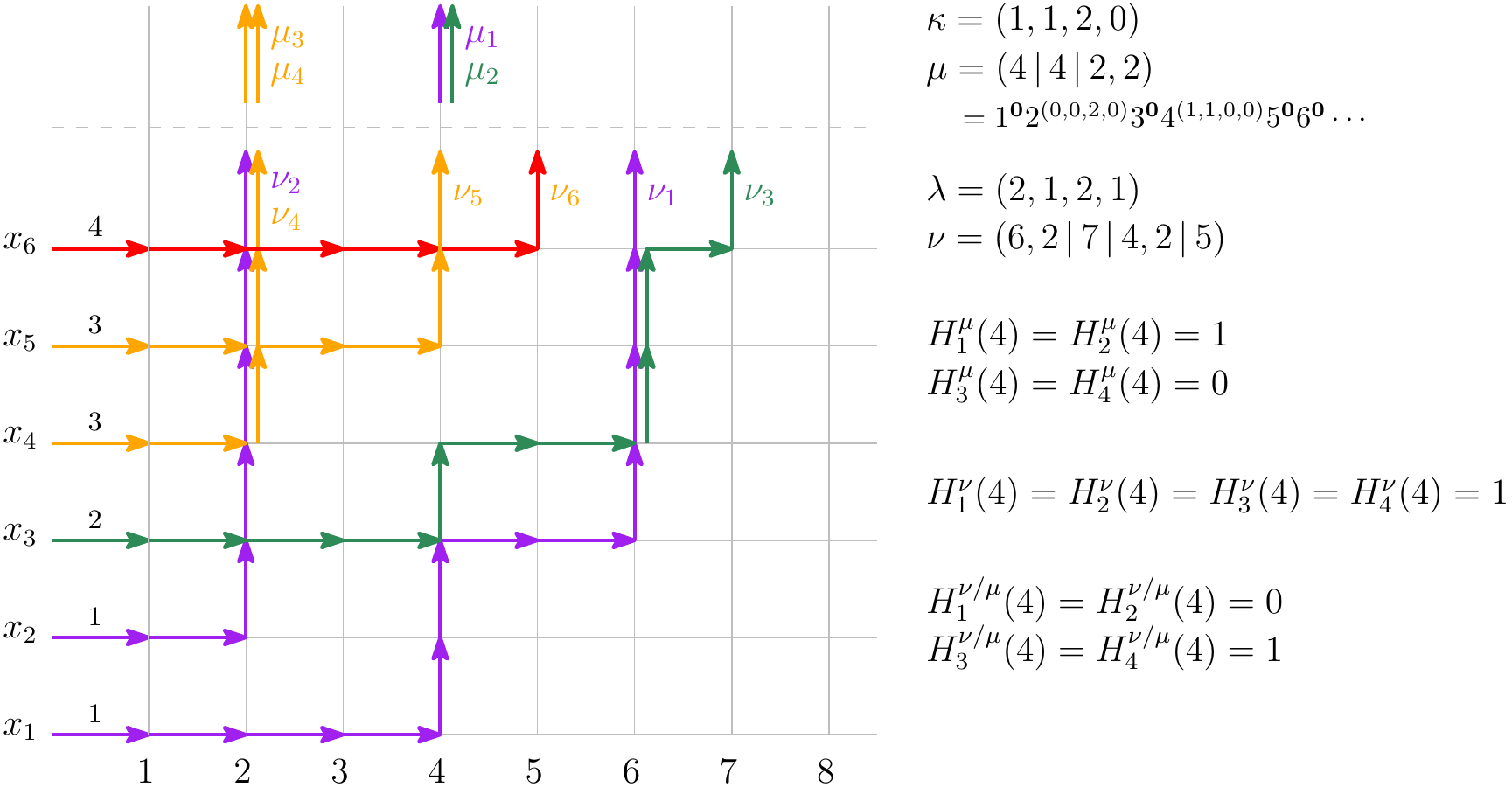}}
			\caption{A possible configuration of the coloured vertex 
				model (best viewed in colour).}
			\label{Fig_coloured}}
	\end{center}
\end{figure}

Our observables on random $\nu$'s are also indexed by coloured compositions. 
Let us fix a  composition $\kappa=(\kappa_1,\kappa_2,\dots)$ satisfying 
$\kappa\le\lambda$ coordinate-wise (otherwise the observable vanishes 
identically), and a coloured composition $\mu$ with coordinates in 
$\mathbb{Z}_{\ge 1}$ whose colour counts are given by $\kappa$. It is also 
convenient to parameterize $\mu$ differently by writing 
$\mu=1^{\bm{m}^{(1)}}2^{\bm{m}^{(2)}}\cdots$, where the vectors 
$\bm{m}^{(j)}=({m}^{(j)}_1,{m}^{(j)}_2,\dots)$ are such that their coordinates 
$m_i^{(j)}$ count the number of parts of $\mu$ of colour $i$ equal to $j$, cf. 
Figure \ref{Fig_coloured}.

Yet another way to describe coloured compositions is through their \emph{coloured height functions}
defined as follows, cf. Figure \ref{Fig_coloured}:
\begin{gather*}
H_i^\rho(x)=\#\{j:\text{colour}(\rho_j)=i,\,\rho_j\ge x\};\\ H_{>i}^\rho(x)=\sum_{k> 
	i}H_k^\rho(x),\ H_{\ge i}^\rho(x)=\sum_{k\ge 
	i}H_k^\rho(x); \quad H^{\nu/\mu}_*\equiv H^\nu_*-H^\mu_*. 
\end{gather*} 
Here $\text{colour}(\rho_j)$ refers to the number of the block in the splitting of the form \eqref{eq:coloured_blocks} for the coloured composition $\rho$ that $\rho_j$ belongs to. 

For $\kappa$-coloured $\mu$ as above, let us now define an observable $\mathcal O_\mu$, whose values on $\lambda$-coloured $\nu$'s are given by 

\begin{multline*}
\O_\mu(\nu)=\displaystyle\prod_{i,j\ge 1} 
q^{m_i^{(j)}H_{>i}^{\nu/\mu}(j+1)}{\displaystyle\binom{H_i^{\nu/\mu}(j+1)}{m_i^{
			(j) } } } _q\\
=\displaystyle\prod_{i,j\ge 1}  
\frac{\bigl(q^{H_{>i}^{\nu/\mu}(j+1)}-q^{H_{\ge 
			i}^{\nu/\mu}(j+1)}\bigr)\bigl(q^{H_{>i}^{\nu/\mu}(j+1)}-q^{H_{\ge 
			i}^{\nu/\mu}(j+1)-1}\bigr)\cdots \bigl(q^{H_{>i}^{\nu/\mu}(j+1)}-q^{H_{\ge 
			i}^{\nu/\mu}(j+1)-m_i^{(j)}+1}\bigr)}{(q;q)_{m_i^{(j)}}}\,.
\end{multline*}

In the \emph{rainbow case} $\lambda=(1,1,\dots,1)$ of all colours being different, $\O_{\mu}$ simplifies to
\begin{equation*}
\O^\text{rainbow}_\mu(\nu)=\displaystyle\prod_{\substack{i,j\,:\,m_i^{(j)}=1}} 
\bm{1}_{H_i^{\nu/\mu}(j+1)=1} \,q^{H_{>i}^{\nu/\mu}(j+1)}=\displaystyle\prod_{\substack{i,j\,:\, m_i^{(j)}=1}}  \frac{q^{H_{> i+1}^{\nu/\mu}(j+1)}-q^{H_{> i}^{\nu/\mu}(j+1)}}{q-1},
\end{equation*}
and in the \emph{colour-blind case} $\lambda=(n)$ it is given by
\begin{equation*}
\O_\mu^{\text{colour-blind}}(\nu)=\frac{(1-q^{H^\nu(\mu_1+1)})(1-q^{
		H^\nu(\mu_2+1)-1})\cdots (1-q^{H^\nu(\mu_m+1)-m+1})}{\prod_{j\ge 1} 
	(q;q)_{\text{mult}_j(\mu)}}\,,
\end{equation*}
where $H^\nu=H^\nu_{\ge 1}$ is the 
colour-blind height function, $\text{mult}_j(\mu)=\#\{i:\mu_i=j\}$, and 
$m=\ell(\mu)$ is the number of parts in $\mu$.  

In order to write down an integral representation for the average of $\O_{\mu}$, 
we need to introduce certain rational functions $f_\mu(\kappa;z_1,\dots,z_m)$ 
with the number of variables $m=\ell(\mu)$ equal to the number of parts of 
$\mu$. In the rainbow case of pairwise distinct colours, coloured compositions 
are the same as uncoloured ones, and for anti-dominant compositions 
$\delta=(\delta_1\le\delta_2\le\dots\le \delta_m)$ these functions are 
completely factorized:
\begin{align*}
f_{\delta}(z_1,\dots,z_m)
=
\frac{\prod_{j \geq 0} (s^2;q)_{\text{mult}_j(\delta)}}{\prod_{i=1}^{m} (1-s z_i)}
\prod_{i=1}^{m} \left( \frac{z_i-s}{1-sz_i} \right)^{\delta_i}.
\end{align*}
Note that we dropped $\kappa$ from the notation for $f_\delta$ as it plays no 
role in the rainbow situation. 
For non-anti-dominant $\mu$'s, one way to define $f_\mu$ is by the following 
recursion that allows to move step by step from anti-dominant compositions toward 
the dominant ones:
If $\mu_i<\mu_{i+1}$ for some $1\le i\le m-1$, then
\begin{equation*}
T_i \cdot f_{\mu}(z_1,\dots,z_m) = f_{(\mu_1,\dots,\mu_{i+1},\mu_i,\dots,\mu_m)}(z_1,\dots,z_m),
\end{equation*}
where 
\begin{align*}
T_i \equiv q - \frac{z_i-q z_{i+1}}{z_i-z_{i+1}} (1-\mathfrak{s}_i),
\quad
1 \leq i \leq m-1,
\end{align*}
with elementary transpositions $\mathfrak{s}_i$ acting by $\mathfrak{s}_i \cdot h(z_1,\dots,z_m):=h(z_1,\dots,z_{i+1},z_i,\dots,z_m)$, are the Demazure--Lusztig operators of the polynomial representation of the Hecke algebra of type $A_{m-1}$. 	 
The rainbow functions $f_\mu$ were thoroughly studied in \cite{BW} under 
the name of \emph{spin nonsymmetric Hall-Littlewood functions}; they also play a 
central role in the present work. 

For generic, not necessarily rainbow $\kappa$ and $\kappa$-coloured $\mu$, $f_\mu$ is defined as a suitable sum of rainbow functions. Concretely, let $\theta:\{1,\dots,m\}\to\{1,2,\dots\}$ be the unique monotone map such that $|\theta^{-1}(j)|=\kappa_j$ for all $j\ge 1$. Then for any composition $\varkappa$ with $m$ parts, we can define a $\kappa$-coloured composition $\theta_*(\varkappa)$ by colouring the coordinate $\varkappa_i$ by colour $j$ if and only if $\theta(i)=j$, for all $i=1,\dots,m$. With this, we set
$$
f_\mu(\kappa,z_1,\dots,z_m)=\sum_{\varkappa\,:\,\theta_*(\varkappa)=\mu} f_\varkappa(z_1,\dots,z_m).
$$

Finally, denote by $c_1<\dots<c_\alpha$ the colours of parts of $\mu$, and denote by 
$\fm_1,\dots,\fm_\alpha\ge 1$ the number of parts of $\mu$ of colours 
$c_1,\dots,c_\alpha$, respectively ($\fm_j$'s are simply re-numbered nonzero 
coordinates of $\kappa$). Set $\fm[a,b]=\fm_a+\fm_{a+1}+\dots+\fm_b$ and recall that $\ell_k=\lambda_1+\dots+\lambda_k$ for $k\ge 1$, $\ell_0=0$. 

We can now state the main result of this paper. 

\begin{thm}
	\label{thm:intro-main}
	 With the above notations, we have
	\begin{multline}
	\label{intro:main}
	\E\, \O_\mu= \frac{q^{\sum_{u\ge 1}\sum_{i>j} m_i^{(u)}m_j ^{(u)}}}{\prod_{j\ge 
			1} (s^2;q)_{|\bm{m}^{(j)}|}}
	\frac{(-s)^{\mu_1+\mu_2+\dots}}{(2\pi\sqrt{-1})^m}\oint\cdots\oint_{\text{around }\{x_j^{-1}\}}
	\prod_{1\le i<j\le m} \frac{y_j-y_i}{y_j-qy_i} \\
	\times
	\prod_{k=1}^{\alpha}\left(
	\sum_{j=0}^{\fm_k} \frac{(-1)^j q^{\binom{\fm_k-j}{2}}}
	{(q;q)_j(q;q)_{\fm_k-j}}
	\prod_{p>\fm[1,{k-1}]}^{j+\fm[1,{k-1}]}\prod_{a>\ell_{c_{k}-1}}^n\frac{1-qx_ay_p}{1-x_ay_p}
	\prod_{r>j+\fm[1,{k-1}]}^{\fm[1,k]}\prod_{b>\ell_{c_{k}}}^n\frac{1-qx_by_r}{1-x_by_r}
	\right)
	\\
	\times  f_{\mu}(\kappa;y_1^{-1},\dots,y_m^{-1})
	\prod_{i=1}^m 
	\frac{(y_i-s)dy_i}{y_i^2}\,,
	\end{multline}
	where (positively oriented) integration contours are chosen to encircle all 
	points $\{x_j^{-1}\}_{j=1}^n$ and no other singularities of the 
	integrand, or as \emph{$q$-nested} closed simple curves with $y_{i}$-contour containing $q^{-1}\cdot (y_j\text{-contour})$ for all $i<j$, and all of the contours encircling $\{x_j^{-1}\}_{j=1}^n$. The contours can also be chosen to either encircle or not encircle the point $0$.   
\end{thm}

Our proof of Theorem \ref{thm:intro-main} is deeply rooted in the formalism of 
spin nonsymmetric Hall-Littlewood functions $f_\mu$ developed in \cite{BW}. The 
key new ingredient is the idea to treat skew-Cauchy identities for these 
functions as averages of certain observables over the measure given by terms of 
the non-skew Cauchy identity. The latter can then be identified, under a certain 
specialization, with the sum over states of the stochastic vertex model along a 
horizontal line.
Accessing observables via deformed Cauchy identities was previously used in 
\cite{BorodinP1, BorodinP2} in the colour-blind case, but the mechanism of 
deformation was different, and the path to integral representations was more 
complex. It should be noted, however, that \cite{BorodinP1} was able to reach an 
integral representation for fully inhomogeneous vertex models (which was later 
exploited analytically in \cite{BP-inhom}). So far we were not able to reach the 
same level of inhomogeneity in the coloured case, although some inhomogeneity 
can be added, and it is actually necessary for the limit to polymers.

Let us briefly mention some algebraic corollaries of Theorem \ref{thm:intro-main}.

First, in the colour-blind case $\lambda=(n)$, \eqref{intro:main} readily leads to a formula for $q$-moments of the height function with a completely factorized integrand that could be viewed as a source of all the major asymptotic advances in the area. This colour-blind reduction is discussed in Section \ref{ssec:main-colour-blind}.

Next, as the dependence on the values of the coordinates of $\mu$ is 
concentrated in the 
$f_\mu$-factor and $f_\mu$'s are eigenfunctions of a transfer-matrix of our 
vertex model, the expectations $\O_{\mu}$ satisfy certain difference equations. 
Those can be seen as evidence that $\O_{\mu}(\nu)$ is actually a \emph{duality 
functional} for our model, see Section \ref{ssec:duality} for details. Duality 
served as a major tool for analyzing ($q$-)moments since \cite{Kardar}. 
It would be very interesting to see if our prospective duality functionals can 
be related to those obtained by Kuan \cite{Kuan}. 

In the rainbow case $\lambda=(1,1,\dots,1)$, together with simplification of 
observables $\O_{\mu}$ to $\O_{\mu}^{\text{rainbow}}$ mentioned above, the 
integrand also simplifies due to the lack of $j$-summations. This case carries 
additional colour-position symmetry both in the left-hand side of 
\eqref{intro:main} and in the right-hand side, cf. Section 
\ref{ssec:main-rainbow}. When applied to $\O_\mu$, this yields another set of 
potential duality functionals. Let us also note that for anti-dominant $\mu$, 
when $f_\mu$ completely factorizes, the result can alternatively be obtained 
from the colour-blind case by applying a (highly nontrivial) shift-invariance 
property of \cite{BGW}. 

Finally, Theorem \ref{thm:intro-main} allows for \emph{stochastic fusion}: A 
cluster of neighboring rows with same left-entering colours and rapidities 
forming a geometric progression can be collapsed to a single ``fat'' row whose 
edges are allowed to carry multiple paths. One can further analytically continue 
in the parameters $q^{\text{number of rows in a cluster}}$, obtaining a 
corresponding result in the fully fused model.\footnote{Our original setup can 
be viewed as a partially fused model because vertical edges are allowed to 
carry multiple paths; it could have been obtained from the \emph{fundamental} 
model with no more than one path on any edge by clustering columns.} Details of 
this procedure can be found in \cite{BW, BGW}, and the resulting version of 
Theorem \ref{thm:intro-main} is Corollary \ref{cor:main-fused} in Section 
\ref{ssec:fusion}. 

The source of analytic corollaries of Theorem \ref{thm:intro-main} is the fact 
the coloured stochastic vertex model and its fully fused version degenerate, in 
various limits, to a variety of other probabilistic systems, see \cite[Chapter 
12]{BW} and \cite{BGW} for some of those degenerations, and the chart in the 
introduction to \cite{BGW} for a ``big picture''. In this text we only consider, 
in Section \ref{sec:polymers} below, the limit into directed random polymers 
that was worked out in \cite{BGW}. We obtain versions of Theorem 
\ref{thm:intro-main} for random Beta-polymers (first considered in 
\cite{BarCor}), strict-weak or Gamma-polymers (first considered in 
\cite{CSS,OCO}), O'Connell-Yor semi-discrete Brownian polymers \cite{OY}, and 
fully continuous\footnote{also known as \emph{continuum}} Brownian polymers 
(equivalently, the stochastic heat equation with multiplicative noise or the KPZ 
equation). 

Two simplifications happen in these limits. First, the presence of colours in 
our vertex models translates into varying starting points of the polymers, 
while the general definitions of the polymer models remain the same as in the 
colourless situation. Second, the observables simplify to pure moments of 
partition functions (no linear combinations necessary) in three of the four 
polymer models we consider. Let us illustrate what happens on the example of 
the continuum Brownian polymer, cf. Section \ref{ssec:cont_Brown}.   

Let $\widetilde{\mathcal{Z}}^{(y)}(t,x)$ be the unique solution of the following 
stochastic partial differential equation with the initial condition:
$$
\widetilde{\mathcal{Z}}^{(y)}_t=\tfrac12 \widetilde{\mathcal{Z}}_{xx}^{(y)}+\eta(t,x)\widetilde{\mathcal{Z}}^{(y)}, \qquad t>0,\quad 
x\in\mathbb{R}; \qquad \widetilde{\mathcal{Z}}^{(y)}(0,x)=\delta(x-y), 
$$
where $\eta=\eta(t,x)$ is the two-dimensional white noise, and set 
$$
\mathcal Z^{(y)}(t,x)=(2\pi t)^{1/2}\, e^{{(x-y)^2}/{2t}}\cdot \widetilde{\mathcal Z}^{(y)}(t,x).
$$ 
See, \emph{e.g.},~\cite{Q} and references therein for an extensive literature on this 
equation and its close relation to continuum Brownian path integrals and the Kardar-Parisi-Zhang equation. 

Let $\varkappa=(\varkappa_1,\dots,\varkappa_m)$ be a coloured vector of length $m$ 
with real coordinates, and let the colours $s_1<\dots<s_\alpha$ of the parts of 
$\varkappa$ also take real values. Denote 
$$
m_i^{(x)}=\#\{j:\varkappa_j=x\ {\rm and\  has\  colour}\ s_i\},\qquad
\fm_i=\sum_x m_i^{(x)}, \qquad 1\le i\le \alpha, \ x\in\mathbb{R}.
$$
Then the limiting version of the fused version of Theorem \ref{thm:intro-main} 
reads, cf. Proposition \ref{prop:main-SHE} below,
\begin{multline}
\label{intro:SHE}
\E\left[ \prod_{(i,x)\,:\,m_i^{(x)}>0}
\frac{\left( 
	\mathcal{Z}^{(s_i)}(t,x)\right)^{m_i^{(x)}}}{{m_i^{(x)}}!}\right]
= 
\frac{1}{(2\pi\sqrt{-1})^m}\int\cdots\int
\prod_{1\le i<j\le m} \frac{w_j-w_i}{w_j-w_i-1} \\
\times
\prod_{k=1}^{\alpha}
\frac{\exp\left(-s_k\cdot\sum_{r>\fm[1,{k-1}]}^{\fm[1,k]}w_r\right)}{
	\fm_k!}
\cdot \fe_{\varkappa}(w_1,\dots,w_m)
\prod_{i=1}^m e^{{tw_i^2}/{2}}{dw_i},
\end{multline}
where the integration is over upwardly oriented lines 
$w_i=a_i+\sqrt{-1}\cdot\mathbb{R}$ with 
$\Re a_j>\Re a_i +1$ for $j>i$. The functions $\fe_\varkappa$ in the integrand 
are the limiting versions of the functions $f_\mu$ in \eqref{intro:main}, and 
they are defined as follows.
In the rainbow case, in the dominant sector $\kappa_1\ge\kappa_2\ge\dots\ge\kappa_m$ one has
\begin{equation*}
\fe_\varkappa(w_1,\dots,w_m)=\exp(\varkappa_1w_1+\dots+\varkappa_m w_m),
\end{equation*} 
and for $\varkappa_i>\varkappa_{i+1}$ for some $1\le i\le m-1$ one uses the exchange relations \begin{equation*}
\mathfrak{T}_i \cdot \fe_{\varkappa} = \fe_{(\varkappa_1,\dots,\varkappa_{i+1},\varkappa_i,\dots,\varkappa_m)},\qquad \mathfrak{T}_i \equiv 1 - \frac{w_i-w_{i+1}+1}{w_i-w_{i+1}} (1-\mathfrak{s}_i),
\qquad
1 \leq i \leq m-1,
\end{equation*}
to extend the definition to all rainbow vectors $\varkappa$. 
For a more general colouring, we use the (unique) colour-identifying monotone 
map $\theta $ from $\{1,\dots,m\}$ to the set of colours $\{s_i\}$, and define 
$ \fe_{\varkappa}=\sum_{{\rm rainbow}\,  \varkappa'\,:\, \theta_*(\varkappa') = \varkappa}
\fe_{\varkappa'}.$

The moment formula \eqref{intro:SHE} has a certain shift-invariance, cf. Remark \ref{rem:shift-invariance} below, that is partially explained by the KPZ-level degeneration of the results of \cite{BGW}. 
This shift-invariance is, furthermore, a corollary of the conjecture in the introduction to \cite{BGW}. It does not imply that conjecture though, because the moments are well-known to not determine the distributions of $\mathcal Z$'s uniquely.  

\subsection*{Acknowledgments}  A.~B.~ was partially supported by NSF grants DMS-1664619 and DMS-1853981. M.~W.~ was partially supported by ARC grant DP19010289.

\section{Preliminaries}\label{sec:prelim}

The goal of this section is to summarize previously proved results that we will need later on. The notation and exposition largely follows \cite{BW}. 

\subsection{The weights}\label{ssec:weights} The vertex models that we consider assign weights to finite collections of finite paths drawn on a square grid. Each vertex for which there exists a path that enters and exits it produces a weight that depends on the configuration of all the paths that go through this vertex. 
The total weight for a collection of paths is the product of weights of the vertices that the paths traverse. We tacitly assume the normalization in which the weight of an empty vertex is always equal to 1.

Our paths are going to be \emph{coloured}, \emph{i.e.}, each path carries a colour that will typically be a (positive) natural number, with colour 0 reserved for the absence of a path. Our vertex weights will actually depend on the ordering of the (nonzero) paths' colours, rather than on their exact values.   
The paths will always travel upward in the vertical direction, and in the horizontal direction a path can travel rightward or leftward, depending on the region of the grid it is in; this choice will always be explicitly specified. 

A basic family of vertex weights that we will use is denoted as 
\begin{align}
\label{eq:W-picture}
	\tikz{0.9}{
		\draw[lgray,line width=4pt,->] (-1,0) -- (1,0);
		\draw[lgray,line width=4pt,->] (0,-1) -- (0,1);
		\node[left] at (-1,0) {\tiny $\BB$};\node[right] at (1,0) {\tiny $\DD$};
		\node[below] at (0,-1) {\tiny $\AA$};\node[above] at (0,1) {\tiny $\CC$};
		\node[left] at (-1.5,0) {$(x,\l) \rightarrow$};
		\node[below] at (0,-1.4) {$\uparrow$};
		\node[below] at (0,-1.9) {$(y,\m)$};
	}
	=
	\vert{\l}{\m}{\frac{x}{y};q}{\AA}{\BB}{\CC}{\DD}
	\equiv
	W_{\l,\m}(x/y;q; \AA,\BB,\CC,\DD),
\end{align}
where $x,y$ are inhomgeneity parameters (or \emph{rapidities}) associated to the row and column, horizontal edges carry no more than $\m$ paths, vertical edges carry no more than $\l$ paths, all paths are of colours $1,\dots,n$ for some $n\ge 1$, and the specific sets of colours on the edges are encoded by compositions\footnote{equivalently, vectors with nonnegative entries} $\AA=(A_1,\dots,A_n)$, $\BB=(B_1,\dots,B_n)$, $\CC=(C_1,\dots,C_n)$, $\DD=(D_1,\dots,D_n)$ subject to the constrains
\begin{align}
	\label{wt-constrain}
	|\AA|, |\CC| \leq \m,
	\qquad
	|\BB|, |\DD| \leq \l,
\end{align}
where we use the notation $|\,\cdot\,|$ to denote the sum of all parts of a composition. 

These remarkable weights come from a family of stochastic $R$-matrices constructed in \cite{KunibaMMO} via symmetric tensor representations of the quantized affine algebra $U_q(\wh{\mathfrak{sl}_{n+1}})$; see also \cite{Kuan}, \cite{BosnjakM}, \cite{AggarwalBB}, \cite{BW}. Let us list some of their properties. 

The \emph{Yang-Baxter equation} for these weights can be written graphically as
\begin{align}
\label{graph-master}
\sum_{\CC_1,\CC_2,\CC_3}
\tikz{0.9}{
	\draw[lgray,line width=4pt,->]
	(-2,1) node[above,scale=0.6] {\color{black} $\AA_1$} -- (-1,0) node[below,scale=0.6] {\color{black} $\CC_1$} -- (1,0) node[right,scale=0.6] {\color{black} $\BB_1$};
	\draw[lgray,line width=4pt,->] 
	(-2,0) node[below,scale=0.6] {\color{black} $\AA_2$} -- (-1,1) node[above,scale=0.6] {\color{black} $\CC_2$} -- (1,1) node[right,scale=0.6] {\color{black} $\BB_2$};
	\draw[lgray,line width=4pt,->] 
	(0,-1) node[below,scale=0.6] {\color{black} $\AA_3$} -- (0,0.5) node[scale=0.6] {\color{black} $\CC_3$} -- (0,2) node[above,scale=0.6] {\color{black} $\BB_3$};
	\node[left] at (-2.2,1) {$(x,\l) \rightarrow$};
	\node[left] at (-2.2,0) {$(y,\m) \rightarrow$};
	\node[below] at (0,-1.4) {$\uparrow$};
	\node[below] at (0,-1.9) {$(z,\n)$};
}
\quad
=
\quad
\sum_{\CC_1,\CC_2,\CC_3}
\tikz{0.9}{
	\draw[lgray,line width=4pt,->] 
	(-1,1) node[left,scale=0.6] {\color{black} $\AA_1$} -- (1,1) node[above,scale=0.6] {\color{black} $\CC_1$} -- (2,0) node[below,scale=0.6] {\color{black} $\BB_1$};
	\draw[lgray,line width=4pt,->] 
	(-1,0) node[left,scale=0.6] {\color{black} $\AA_2$} -- (1,0) node[below,scale=0.6] {\color{black} $\CC_2$} -- (2,1) node[above,scale=0.6] {\color{black} $\BB_2$};
	\draw[lgray,line width=4pt,->] 
	(0,-1) node[below,scale=0.6] {\color{black} $\AA_3$} -- (0,0.5) node[scale=0.6] {\color{black} $\CC_3$} -- (0,2) node[above,scale=0.6] {\color{black} $\BB_3$};
	\node[left] at (-1.5,1) {$(x,\l) \rightarrow$};
	\node[left] at (-1.5,0) {$(y,\m) \rightarrow$};
	\node[below] at (0,-1.4) {$\uparrow$};
	\node[below] at (0,-1.9) {$(z,\n)$};
}
\end{align}
See also \cite[(C.1.2)]{BW} for the corresponding formula. 

These weights enjoy a \emph{transpositional symmetry}
\begin{align}
\label{bm-sym}
\vert{\l}{\m}{x;q}{\AA}{\BB}{\CC}{\DD}
=
\vert{\m}{\l}{\frac{q^{\m-\l}}{x};\frac{1}{q}}{\BB}{\AA}{\DD}{\CC}.
\end{align}

They are \emph{stochastic} in the following sense:
\begin{align}
\label{sum-to-1}
\sum_{\CC, \DD}
\vert{\l}{\m}{x;q}{\AA}{\BB}{\CC}{\DD}
=
1,
\end{align}
where the sum is over all compositions $\CC = (C_1,\dots,C_n)$ and $\DD = (D_1,\dots,D_n)$ with $|\CC|\le \m$, $|\DD|\le \l$.

They can also be given by the following explicit formula originating from \cite{BosnjakM}, see also \cite[Theorem C.1.1]{BW}:
\begin{multline}
\label{eq:W-weights}
\vert{\l}{\m}{x;q}{\AA}{\BB}{\CC}{\DD}
=
\left( \bm{1}_{\AA + \BB = \CC + \DD} \right)
x^{|\DD|-|\BB|} q^{|\AA| \l - |\DD| \m}
\\
\times
\sum_{\bm{P}}
\Phi(\CC-\bm{P},\CC+\DD-\bm{P}; q^{\l-\m} x, q^{-\m} x)
\Phi(\bm{P},\BB; q^{-\l}/x, q^{-\l}),
\end{multline}
where the sum is over compositions $\bm{P} = (P_1,\dots,P_n)$ such that $0 \leq P_i \leq \min(B_i,C_i)$ for all $1 \leq i \leq n$; and for any two compositions $\lambda, \mu \in \mathbb{N}^n$ such that $\lambda_i \leq \mu_i$ for all $1\leq i \leq n$, we used the notation
\begin{align*}
\Phi(\lambda,\mu;x,y)
:=
\frac{(x;q)_{|\lambda|} (y/x;q)_{|\mu-\lambda|}}{(y;q)_{|\mu|}}
(y/x)^{|\lambda|}\,
q^{\sum_{i<j} (\mu_i-\lambda_i) \lambda_j}
\prod_{i=1}^{n} \binom{\mu_i}{\lambda_i}_q.
\end{align*}
The substitution of $\l=\m=1$ into $W_{\l,\m}$ returns the (stochastic version of) the fundamental $R$-matrix for $U_q(\wh{\mathfrak{sl}_{n+1}})$:
\begin{align}
\label{eq:W-R}
\vert{1}{1}{\frac{x}{y};q}{\AA}{\BB}{\CC}{\DD}
=
\left\{
\begin{array}{ll}
R_{y/x}(\AA^*,\BB^*;\CC^*,\DD^*), 
& \quad
|\AA|, |\BB|, |\CC|, |\DD| \leq 1,
\\ \\
0,
& \quad
{\rm otherwise},
\end{array}
\right.
\end{align}
and where we have defined 
\begin{align*}
\I^{*}
=
\left\{
\begin{array}{ll}
0, & \quad \I = \bm{0},
\\
i, & \quad \I = \bm{e}_i\ \  (\text{$i$-th standard basis vector}),
\end{array}
\right.
\end{align*}
for any composition $\I = (I_1,\dots,I_n)$ such that $|\I| \leq 1$, and where $R_{y/x}$ denotes the fundamental $R$-matrix depicted as (with $z=y/x$) 
\begin{align}
\label{R-vert}
R_z(i,j; k,\ell)
=
\tikz{0.7}{
	\draw[lgray,line width=1.5pt,->] (-1,0) -- (1,0);
	\draw[lgray,line width=1.5pt,->] (0,-1) -- (0,1);
	\node[left] at (-1,0) {\tiny $j$};\node[right] at (1,0) {\tiny $\ell$};
	\node[below] at (0,-1) {\tiny $i$};\node[above] at (0,1) {\tiny $k$};
},
\quad
i,j,k,\ell \in \{0,1,\dots,n\},
\end{align}
and with matrix elements summarized by the following table, in which we assume that $0 \leq i < j \leq n$:
\begin{align}
\label{fund-vert}
\begin{tabular}{|c|c|c|}
\hline
\quad
\tikz{0.6}{
	\draw[lgray,line width=1.5pt,->] (-1,0) -- (1,0);
	\draw[lgray,line width=1.5pt,->] (0,-1) -- (0,1);
	\node[left] at (-1,0) {\tiny $i$};\node[right] at (1,0) {\tiny $i$};
	\node[below] at (0,-1) {\tiny $i$};\node[above] at (0,1) {\tiny $i$};
}
\quad
&
\quad
\tikz{0.6}{
	\draw[lgray,line width=1.5pt,->] (-1,0) -- (1,0);
	\draw[lgray,line width=1.5pt,->] (0,-1) -- (0,1);
	\node[left] at (-1,0) {\tiny $i$};\node[right] at (1,0) {\tiny $i$};
	\node[below] at (0,-1) {\tiny $j$};\node[above] at (0,1) {\tiny $j$};
}
\quad
&
\quad
\tikz{0.6}{
	\draw[lgray,line width=1.5pt,->] (-1,0) -- (1,0);
	\draw[lgray,line width=1.5pt,->] (0,-1) -- (0,1);
	\node[left] at (-1,0) {\tiny $i$};\node[right] at (1,0) {\tiny $j$};
	\node[below] at (0,-1) {\tiny $j$};\node[above] at (0,1) {\tiny $i$};
}
\quad
\\[1.3cm]
\quad
$1$
\quad
& 
\quad
$\dfrac{q(1-z)}{1-qz}$
\quad
& 
\quad
$\dfrac{1-q}{1-qz}$
\quad
\\[0.7cm]
\hline
&
\quad
\tikz{0.6}{
	\draw[lgray,line width=1.5pt,->] (-1,0) -- (1,0);
	\draw[lgray,line width=1.5pt,->] (0,-1) -- (0,1);
	\node[left] at (-1,0) {\tiny $j$};\node[right] at (1,0) {\tiny $j$};
	\node[below] at (0,-1) {\tiny $i$};\node[above] at (0,1) {\tiny $i$};
}
\quad
&
\quad
\tikz{0.6}{
	\draw[lgray,line width=1.5pt,->] (-1,0) -- (1,0);
	\draw[lgray,line width=1.5pt,->] (0,-1) -- (0,1);
	\node[left] at (-1,0) {\tiny $j$};\node[right] at (1,0) {\tiny $i$};
	\node[below] at (0,-1) {\tiny $i$};\node[above] at (0,1) {\tiny $j$};
}
\quad
\\[1.3cm]
& 
\quad
$\dfrac{1-z}{1-qz}$
\quad
&
\quad
$\dfrac{(1-q)z}{1-qz}$
\quad 
\\[0.7cm]
\hline
\end{tabular}
\end{align}
Observe that these weights are manifestly stochastic. 

The general weights $W_{\l,\m}$ can be reconstructed from the fundamental ones above via the procedure of \emph{stochastic fusion}. To state how it works we need a bit of notation. 

Let $\n \geq 1$, and consider a vector of nonnegative integers $(i_1,\dots,i_\n) \in \{0,1,\dots,n\}^M$. From this we define another vector,
\begin{align*}
\mathcal{C}(i_1,\dots,i_\n)
:=
(I_1,\dots,I_n),
\qquad
I_a = \#\{k : i_k = a \},
\quad 1 \leq a \leq n,
\end{align*}
which keeps track of the multiplicity of each colour $1 \leq a \leq n$ within $(i_1,\dots,i_\n)$.
Set 
\begin{align*}
{\rm inv}(i_1,\dots,i_\n) = \#\{1\le a<b \le \n: i_a > i_b \},\quad \widetilde{\rm inv}(i_1,\dots,i_\n) = \#\{1\le a<b \le \n: i_a < i_b \}, 
\end{align*}
and, denoting $I_0 := \n - \sum_{a=1}^{n} I_a$,
\begin{align*}
Z_q(\n;\I)
=
\sum_{\mathcal{C}(i_1,\dots,i_\n) = \I}
q^{{\rm inv}(i_1,\dots,i_\n)}
= \sum_{\mathcal{C}(j_1,\dots,j_\n) = \I}
q^{\widetilde{\rm inv}(j_1,\dots,j_\n)}
=
\frac{(q;q)_\n}{(q;q)_{I_0} (q;q)_{I_1} \dots (q;q)_{I_n}}.
\end{align*}
Then, cf. \cite[Appendix]{BGW},
\begin{multline}
\label{eq:fusion}
\vert{\l}{\m}{\frac{x}{y};q}{\AA}{\BB}{\CC}{\DD}
=
\frac{1}{Z_q(\m;\AA)Z_q(\l;\BB)}\\ \times
\sum_{\substack{
		\mathcal{C}(j_1,\dots,j_\l) = \BB
		\\
		\mathcal{C}(\ell_1,\dots,\ell_\l) = \DD
}}
q^{\widetilde{\rm inv}(j_1,\dots,j_\l)}
\sum_{\substack{
		\mathcal{C}(i_1,\dots,i_\m) = \AA
		\\
		\mathcal{C}(k_1,\dots,k_\m) = \CC
}}
q^{{\rm inv}(i_1,\dots,i_\m)}
\tikz{0.8}{
	\foreach\y in {2,...,5}{
		\draw[lgray,line width=1.5pt,->] (1,\y) -- (8,\y);
	}
	\foreach\x in {2,...,7}{
		\draw[lgray,line width=1.5pt,->] (\x,1) -- (\x,6);
	}
	\node[above] at (7,6) {$\tiny k_\m$};
	\node[above] at (5,6) {$\tiny \cdots$};
	\node[above] at (4,6) {$\tiny \cdots$};
	\node[above] at (2,6) {$\tiny k_1$};
	\node[left] at (0,2) {$\tiny x$}; \node[left] at (1,2) {$\tiny j_1$};
	\node[left] at (0.5,3) {$\tiny \vdots$};
	\node[left] at (0.5,4) {$\tiny \vdots$};
	\node[left] at (0,5) {$\tiny q^{\l-1}x$}; \node[left] at (1,5) {$\tiny j_\l$};
	\node[below] at (7,0) {$\tiny y$}; \node[below] at (7,1) {$\tiny i_\m$};
	\node[below] at (5,1) {$\tiny \cdots$};
	\node[below] at (4,1) {$\tiny \cdots$};
	\node[below] at (2,0.2) {$\tiny q^{\m-1}y$}; \node[below] at (2,1) {$\tiny i_1$};
	\node[right] at (8,2) {$\tiny \ell_1$};
	\node[right] at (8,3) {$\tiny \vdots$};
	\node[right] at (8,4) {$\tiny \vdots$};
	\node[right] at (8,5) {$\tiny \ell_\l$};
}
\end{multline}
where the figure on the right denotes the corresponding partition function with $R$-weights 
\eqref{R-vert}-\eqref{fund-vert} (thus, summation over all possible states of interior edges is assumed), and the colours on the boundary edges, as well as row and column rapidities, are explicitly indicated.\footnote{As in \eqref{R-vert}, the spectral parameter of an $R$-vertex is assumed to be equal to the ratio of the column rapidity and the row rapidity.} 

A key feature of the fused $R$-vertices that allows stacking them together is their \emph{$q$-echangeability}: In the above expression, the sum over $(k_1,\dots,k_\m)$ can be omitted at the expense of adding the factor of $q^{-{\rm inv}(k_1,\dots,k_\m)}Z_q(\m;\CC)$ to the summands, and, independently, the sum over $(\ell_1,\dots,\ell_\l)$ can be removed at the expense of adding the factor of $q^{-\widetilde{\rm inv}(\ell_1,\dots,\ell_\l)}Z_q(\l;\DD)$, see \cite[Proposition B.2.2]{BW} for a proof. 

In what follows we will also need partially fused weights defined as follows, cf. \eqref{bm-sym}:
\begin{gather}
\label{eq:L-stoch}
L^\stoch_{x}(\AA,b;\CC,d)
:=
\left.
\vert{1}{\m}{\frac{x}{s};q}{\AA}{\bm{e}_b}{\CC}{\bm{e}_d}
\right|_{q^{\m} \rightarrow s^{-2}},
\\  
\label{eq:M-stoch}
M^\stoch_{x}(\AA,b;\CC,d)
:=
\left.
\vert{1}{\m}{\frac{s}{x};\frac{1}{q}}{\AA}{\bm{e}_b}{\CC}{\bm{e}_d}
\right|_{q^{\m} \rightarrow s^{-2}},
\end{gather}
where the substitution of a generic complex parameter $s^2$ for $q^{-\m}$ is based on the fact that the right-hand sides are rational in $q^\m$. The weight $L_x^\stoch$ are tabulated in \eqref{intro:s-weights}.

In addition, we will use gauge transformed (non-stochastic) versions
\begin{gather}
\label{eq:L-weights}
(-s)^{-\mathbf{1}_{\ell\ge 1}}L^\stoch_x(\I,j;\K,\ell)=:
L_x(\I,j;\K,\ell)=
\tikz{0.7}{
	\draw[lgray,line width=1.5pt,->] (-1,0) -- (1,0);
	\draw[lgray,line width=4pt,->] (0,-1) -- (0,1);
	\node[left] at (-1,0) {$x \rightarrow$ \tiny $j$};\node[right] at (1,0) {\tiny $\ell$};
	\node[below] at (0,-1) {\tiny $\I$};\node[above] at (0,1) {\tiny $\K$};
},\\
\label{eq:M-weights}
(-s)^{\mathbf{1}_{j\ge 1}}M_x^\stoch(\I,j;\K,\ell)
=:
M_x(\I,j;\K,\ell)=\tikz{0.7}{
	\draw[lgray,line width=1.5pt,<-] (-1,0) -- (1,0);
	\draw[lgray,line width=4pt,->] (0,-1) -- (0,1);
	\node[left] at (-1,0)  {$y \leftarrow$ \tiny $\ell$};\node[right] at (1,0) {\tiny $j$};
	\node[below] at (0,-1) {\tiny $\I$};\node[above] at (0,1) {\tiny $\K$};
},
\end{gather}
that, as a consequence of \eqref{graph-master}, satisfy the following version of the Yang-Baxter equation, cf. \cite[(2.3.5)]{BW}:
\begin{align}
\label{graph-RLLb}
\sum_{0 \leq k_1,k_3 \leq n}
\
\sum_{\K \in \mathbb{N}^n}
\tikz{0.8}{
	\draw[lgray,line width=1.5pt,<-] 
	(-1,1) node[left,scale=0.6] {\color{black} $j_3$} -- (1,1) node[above,scale=0.6] {\color{black} $k_3$} -- (2,0) node[below,scale=0.6] {\color{black} $i_3$};
	\draw[lgray,line width=1.5pt,->] 
	(-1,0) node[left,scale=0.6] {\color{black} $i_1$} -- (1,0) node[below,scale=0.6] {\color{black} $k_1$} -- (2,1) node[above,scale=0.6] {\color{black} $j_1$};
	\draw[lgray,line width=4pt,->] 
	(0,-1) node[below,scale=0.6] {\color{black} $\I$} -- (0,0.5) node[scale=0.6] {\color{black} $\K$} -- (0,2) node[above,scale=0.6] {\color{black} $\J$};
	\node[left] at (-1.5,1) {$y \leftarrow$};
	\node[left] at (-1.5,0) {$x \rightarrow$};
}
\quad
=
\quad
\sum_{0 \leq k_1,k_3 \leq n}
\
\sum_{\K \in \mathbb{N}^n}
\tikz{0.8}{
	\draw[lgray,line width=1.5pt,<-] 
	(-2,1) node[above,scale=0.6] {\color{black} $j_3$} -- (-1,0) node[below,scale=0.6] {\color{black} $k_3$} -- (1,0) node[right,scale=0.6] {\color{black} $i_3$};
	\draw[lgray,line width=1.5pt,->] 
	(-2,0) node[below,scale=0.6] {\color{black} $i_1$} -- (-1,1) node[above,scale=0.6] {\color{black} $k_1$} -- (1,1) node[right,scale=0.6] {\color{black} $j_1$};
	\draw[lgray,line width=4pt,->] 
	(0,-1) node[below,scale=0.6] {\color{black} $\I$} -- (0,0.5) node[scale=0.6] {\color{black} $\K$} -- (0,2) node[above,scale=0.6] {\color{black} $\J$};
	\node[left] at (-2.2,1) {$y \leftarrow$};
	\node[left] at (-2.2,0) {$x \rightarrow$};
}
\end{align}
with the spectral parameter of the $R$-vertex equal to $(qxy)^{-1}$.

 The explicit values of the weights \eqref{eq:L-weights} are summarized by the table below:
\begin{align}
\label{s-weights}
\begin{tabular}{|c|c|c|}
\hline
\quad
\tikz{0.7}{
	\draw[lgray,line width=1.5pt,->] (-1,0) -- (1,0);
	\draw[lgray,line width=4pt,->] (0,-1) -- (0,1);
	\node[left] at (-1,0) {\tiny $0$};\node[right] at (1,0) {\tiny $0$};
	\node[below] at (0,-1) {\tiny $\I$};\node[above] at (0,1) {\tiny $\I$};
}
\quad
&
\quad
\tikz{0.7}{
	\draw[lgray,line width=1.5pt,->] (-1,0) -- (1,0);
	\draw[lgray,line width=4pt,->] (0,-1) -- (0,1);
	\node[left] at (-1,0) {\tiny $i$};\node[right] at (1,0) {\tiny $i$};
	\node[below] at (0,-1) {\tiny $\I$};\node[above] at (0,1) {\tiny $\I$};
}
\quad
&
\quad
\tikz{0.7}{
	\draw[lgray,line width=1.5pt,->] (-1,0) -- (1,0);
	\draw[lgray,line width=4pt,->] (0,-1) -- (0,1);
	\node[left] at (-1,0) {\tiny $0$};\node[right] at (1,0) {\tiny $i$};
	\node[below] at (0,-1) {\tiny $\I$};\node[above] at (0,1) {\tiny $\I^{-}_i$};
}
\quad
\\[1.3cm]
\quad
$\dfrac{1-s x q^{\Is{1}{n}}}{1-sx}$
\quad
& 
\quad
$\dfrac{(x-sq^{I_i}) q^{\Is{i+1}{n}}}{1-sx}$
\quad
& 
\quad
$\dfrac{x(1-q^{I_i}) q^{\Is{i+1}{n}}}{1-sx}$
\quad
\\[0.7cm]
\hline
\quad
\tikz{0.7}{
	\draw[lgray,line width=1.5pt,->] (-1,0) -- (1,0);
	\draw[lgray,line width=4pt,->] (0,-1) -- (0,1);
	\node[left] at (-1,0) {\tiny $i$};\node[right] at (1,0) {\tiny $0$};
	\node[below] at (0,-1) {\tiny $\I$};\node[above] at (0,1) {\tiny $\I^{+}_i$};
}
\quad
&
\quad
\tikz{0.7}{
	\draw[lgray,line width=1.5pt,->] (-1,0) -- (1,0);
	\draw[lgray,line width=4pt,->] (0,-1) -- (0,1);
	\node[left] at (-1,0) {\tiny $i$};\node[right] at (1,0) {\tiny $j$};
	\node[below] at (0,-1) {\tiny $\I$};\node[above] at (0,1) 
	{\tiny $\I^{+-}_{ij}$};
}
\quad
&
\quad
\tikz{0.7}{
	\draw[lgray,line width=1.5pt,->] (-1,0) -- (1,0);
	\draw[lgray,line width=4pt,->] (0,-1) -- (0,1);
	\node[left] at (-1,0) {\tiny $j$};\node[right] at (1,0) {\tiny $i$};
	\node[below] at (0,-1) {\tiny $\I$};\node[above] at (0,1) {\tiny $\I^{+-}_{ji}$};
}
\quad
\\[1.3cm] 
\quad
$\dfrac{1-s^2 q^{\Is{1}{n}}}{1-sx}$
\quad
& 
\quad
$\dfrac{x(1-q^{I_j}) q^{\Is{j+1}{n}}}{1-sx}$
\quad
&
\quad
$\dfrac{s(1-q^{I_i})q^{\Is{i+1}{n}}}{1-sx}$
\quad
\\[0.7cm]
\hline
\end{tabular} 
\end{align}
where we assume that $1 \leq i<j \leq n$, and the notation $\Is{k}{n}$ stands for 
$\sum_{a=k}^{n} I_a$. For $n=1$, these weights correspond to the image of the universal $R$-matrix for the quantum affine group $U_q(\wh{\mathfrak{sl}_{2}})$ in the tensor product of its vector representation (horizontal edges) and a Verma module (vertical edges), with parameter $s$ encoding its highest weight; we call $s$ the \emph{spin parameter}. 

\subsection{The $q$-Hahn specialization and a limit relation}\label{ssec:q-Hahn} The complicated expression \eqref{eq:W-weights}
can simplify at special values of parameters; we have already seen this in the case of $L^\stoch$ and $M^\stoch$. Another such specialization that allows $\l$ and $\m$ to remain generic is the following (see \cite[Proposition 7]{KunibaMMO} and also \cite[(7.13)]{BosnjakM}):
Assuming that $\l\le \m$, we have
\begin{multline}
\label{eq:q-Hahn}
\vert{\l}{\m}{1;q}{\AA}{\BB}{\CC}{\DD}
= q^{(\l-\m)|\DD|}\, \frac{(q^{\l-\m};q)_{|\AA|-|\DD|}(q^{-\l};q)_{|\DD|}}{(q^{-\m};q)_{|\AA|}}\,q^{\sum_{i<j}D_i(A_j-D_j)}\prod_{i=1}^n
{\binom{A_i}{A_i-D_i}}_q\\ =\Phi(\DD,\AA;q^{-\l},q^{-\m}).
\end{multline}

The proof follows from the fact that setting $x=1$ restricts the sum in \eqref{eq:W-weights} to a single term with $\bm{P}=\BB$. This specialization is often referred to as the \emph{$q$-Hahn point} because for $n=1$ it reproduces the orthogonality weights for the classical $q$-Hahn orthogonal polynomials. 

Note that replacing $q^{-\m}$ by a generic complex parameter $s^2$ that was done for $W_{1,\m}$ in \eqref{eq:L-stoch} can be also be performed for $W_{\l,\m}$ for $\l>1$ either in \eqref{eq:W-weights} or in \eqref{eq:q-Hahn}, because those are weights are still manifestly rational in $q^\m$. Alternatively, the result of such a replacement could be seen as a stochastic fusion of the weights $L^\stoch$ in the spirit of \eqref{eq:fusion}, where only the outer sum over $j_*$'s and $\ell_*$'s is present. 

We will also need the following limiting relation for the weights \eqref{eq:W-weights}, see \cite[Lemma 6.8]{BGW}:

Assume that $\BB=(0,\dots,0,\l)$. Then 
\begin{equation}
\label{eq:weight-limit}
\lim_{q^{-\m}=s^2\to 0} \vert{\l}{\m}{zq^{\m};q}{\AA}{\BB}{\CC}{\DD}=
\begin{cases}
\dfrac{(zq^\l;q)_\infty}{(z;q)_\infty}\dfrac{(q^{-\l};q)_d}{(q;q)_d}\,(zq^\l)^d  &\text{if}\quad \DD=(0,\dots,0,d), \ d\ge 0,\\
0,&\text{otherwise}. 
\end{cases}
\end{equation}

\subsection{Row operators and rational functions}\label{ssec:row-operators}

Let $V$ be an infinite-dimensional vector space obtained by taking the linear span of all $n$-tuples of nonnegative integers:  
\begin{align*}
V= {\rm Span}\{\ket{\I}\} =
{\rm Span}\{\ket{i_1,\dots,i_n}\}_{i_1,\dots,i_n \in \mathbb{N}}.
\end{align*}
It is convenient to consider an infinite tensor product of such spaces, 
\begin{align*}
\mathbb{V} = V_0 \otimes V_1 \otimes V_2 \otimes \cdots,
\end{align*}
where each $V_i$ denotes a copy of $V$. Let $\bigotimes_{k=0}^{\infty} \ket{\I_k}_k$ be a finite state in $\mathbb{V}$, \ie\ assume that there exists $N \in \mathbb{N}$ such that $\I_k = \bm{0}$ for all $k > N$; in what follows only such states are considered. We define two families of linear operators acting on the finite states:
\begin{align}
\label{C-row}
\C_i(x)
:
\bigotimes_{k=0}^{\infty}
\ket{\I_k}_k
\mapsto
\sum_{\J_0,\J_1,\ldots \in \mathbb{N}^n}
\left(
\tikz{0.6}{
	\draw[lgray,line width=1.5pt,->] (0,0) -- (7,0);
	\foreach\x in {1,...,6}{
		\draw[lgray,line width=4pt,->] (\x,-1) -- (\x,1);
	}
	\node[left] at (-0.5,0) {$x \rightarrow$};
	\node[left] at (0,0) {\tiny $i$};\node[right] at (7,0) {\tiny $0$};
	\node[below] at (6,-1) {\tiny $\cdots$};\node[above] at (6,1) {\tiny $\cdots$};
	\node[below] at (5,-1) {\tiny $\cdots$};\node[above] at (5,1) {\tiny $\cdots$};
	\node[below] at (4,-1) {\tiny $\cdots$};\node[above] at (4,1) {\tiny $\cdots$};
	\node[below] at (3,-1) {\tiny $\J_2$};\node[above] at (3,1) {\tiny $\I_2$};
	\node[below] at (2,-1) {\tiny $\J_1$};\node[above] at (2,1) {\tiny $\I_1$};
	\node[below] at (1,-1) {\tiny $\J_0$};\node[above] at (1,1) {\tiny $\I_0$};
}
\right)
\bigotimes_{k=0}^{\infty}
\ket{\J_k}_k,
\quad
0 \leq i \leq n,
\end{align}

\begin{align}
\label{B-row}
\B_i(x)
:
\bigotimes_{k=0}^{\infty}
\ket{\I_k}_k
\mapsto
\sum_{\J_0,\J_1,\ldots \in \mathbb{N}^n}
\left(
\tikz{0.6}{
	\draw[lgray,line width=1.5pt,<-] (0,0) -- (7,0);
	\foreach\x in {1,...,6}{
		\draw[lgray,line width=4pt,->] (\x,-1) -- (\x,1);
	}
	\node[left] at (-0.5,0) {$x \leftarrow$};
	\node[left] at (0,0) {\tiny $i$};\node[right] at (7,0) {\tiny $0$};
	\node[below] at (6,-1) {\tiny $\cdots$};\node[above] at (6,1) {\tiny $\cdots$};
	\node[below] at (5,-1) {\tiny $\cdots$};\node[above] at (5,1) {\tiny $\cdots$};
	\node[below] at (4,-1) {\tiny $\cdots$};\node[above] at (4,1) {\tiny $\cdots$};
	\node[below] at (3,-1) {\tiny $\J_2$};\node[above] at (3,1) {\tiny $\I_2$};
	\node[below] at (2,-1) {\tiny $\J_1$};\node[above] at (2,1) {\tiny $\I_1$};
	\node[below] at (1,-1) {\tiny $\J_0$};\node[above] at (1,1) {\tiny $\I_0$};
}
\right)
\bigotimes_{k=0}^{\infty}
\ket{\J_k}_k,
\quad
0 \leq i \leq n,
\end{align}
where the vertex weights are those from \eqref{eq:L-weights} and \eqref{eq:M-weights}, respectively. Note that the sums above are always finite due to path conservation, and the infinite number of empty vertices far to the right all have weight 1. 

A direct corollary of the Yang-Baxter equation \eqref{graph-RLLb} is the fact that the row operators $\B_i$ and $\C_i$ satisfy certain explicit quadratic commutation relations, cf. \cite[Section 3.2]{BW}. The simplest ones state that for a fixed $i$, $0\le i\le n$, the operators $\B_i(x)$ commute between themselves for different values of $x$, and, similarly, $\C_i(x)$ also commute for different values of $x$. Slightly more complicated are the following relations between $\B$- and $\C$-operators, cf. \cite[Theorem 3.2.3]{BW}: 

Fix two nonnegative integers $i,j$ such that $0 \leq i,j \leq n$, and complex parameters $x,y$ such that
\begin{align}
\label{eq:weight-condition}
\left|
\frac{x-s}{1-sx}
\cdot
\frac{y-s}{1-sy}
\right|
<
1.
\end{align}
Then the row operators \eqref{C-row} and \eqref{B-row} obey the following commutation relations ($0\le i,j\le n$):
\begin{equation}
\label{eq:CB}
\begin{gathered}
\C_i(x) \B_j(y) = \frac{1-qxy}{1-xy}\, \B_j(y) \C_i(x),
\quad
i<j,
\qquad
q\, \C_i(x) \B_j(y) = \frac{1-qxy}{1-xy}\, \B_j(y) \C_i(x),
\quad
i>j,\\
\C_i(x) \B_i(y) =
\frac{1-q^{-1}}{1-xy} \sum_{k < i} \B_k(y) \C_k(x)
+
\B_i(y) \C_i(x)
-
\frac{(1-q)xy}{1-xy} \sum_{k > i} \B_k(y) \C_k(x).
\end{gathered}
\end{equation}

Note that matrix elements of the left-hand sides in \eqref{eq:CB} are given by infinite sums, and \eqref{eq:weight-condition} ensures that those sums converge. It is thanks to the infinite range of the lattice in \eqref{C-row} and \eqref{B-row} that relations \eqref{eq:CB} are simpler than the usual commutation relations in the Yang-Baxter algebra.

We are now in position to introduce certain rational functions that will play a central role in what follows;
they were called \emph{nonsymmetric spin Hall-Littlewood functions} in \cite{BW}. 

For a composition $\mu$ and $n$ complex parameters $x_1,\dots, x_n$, define a rational function $f_\mu(x_1,\dots,x_n)$ as a partition function depicted below (vertex weights \eqref{eq:L-weights} are being used):
\begin{equation}
\label{eq:f-function}
f_{\mu}(x_1,\dots,x_n)
=
\tikz{0.8}{
	\foreach\y in {1,...,5}{
		\draw[lgray,line width=1.5pt,->] (1,\y) -- (8,\y);
	}
	\foreach\x in {2,...,7}{
		\draw[lgray,line width=4pt,->] (\x,0) -- (\x,6);
	}
	\node[left] at (0.5,1) {$x_1 \rightarrow$};
	\node[left] at (0.5,2) {$x_2 \rightarrow$};
	\node[left] at (0.5,3) {$\vdots$};
	\node[left] at (0.5,4) {$\vdots$};
	\node[left] at (0.5,5) {$x_n \rightarrow$};
	\node[below] at (7,0) {$\cdots$};
	\node[below] at (6,0) {$\cdots$};
	\node[below] at (5,0) {$\cdots$};
	\node[below] at (4,0) {\footnotesize$\bm{0}$};
	\node[below] at (3,0) {\footnotesize$\bm{0}$};
	\node[below] at (2,0) {\footnotesize$\bm{0}$};
	\node[above] at (7,6) {$\cdots$};
	\node[above] at (6,6) {$\cdots$};
	\node[above] at (5,6) {$\cdots$};
	\node[above] at (4,6) {\footnotesize$\bm{A}(2)$};
	\node[above] at (3,6) {\footnotesize$\bm{A}(1)$};
	\node[above] at (2,6) {\footnotesize$\bm{A}(0)$};
	\node[right] at (8,1) {$0$};
	\node[right] at (8,2) {$0$};
	\node[right] at (8,3) {$\vdots$};
	\node[right] at (8,4) {$\vdots$};
	\node[right] at (8,5) {$0$};
	\node[left] at (1,1) {$1$};
	\node[left] at (1,2) {$2$};
	\node[left] at (1,3) {$\vdots$};
	\node[left] at (1,4) {$\vdots$};
	\node[left] at (1,5) {$n$};
}
\end{equation}
with $\bm{A}(k) = \sum_{j=1}^{n} \bm{1}_{\mu_j = k} \bm{e}_j$. These are certain matrix elements of the operator $\C_1(x_1)\cdots \C_n(x_n)$ that can be symbolically written in the form $\bra{\varnothing} \C_1(x_1)\cdots \C_n(x_n) \ket{\mu}$. 
A more general definition of the $f_\mu$'s, as described in \cite[Section 3.4]{BW}, involves some of the operators $\C_i$ being repeated with different arguments (equivalently, some of the paths entering the partition function \eqref{eq:f-function} from the left being of the same colour); we will meet such functions in Section \ref{ssec:coloured-comp} below. 

Similarly to the $f_\mu$'s, one defines dual functions
\begin{equation}
\label{eq:g-function}
g_{\mu}(x_1,\dots,x_n)
=
\tikz{0.8}{
	\foreach\y in {1,...,5}{
		\draw[lgray,line width=1.5pt,<-] (1,\y) -- (8,\y);
	}
	\foreach\x in {2,...,7}{
		\draw[lgray,line width=4pt,->] (\x,0) -- (\x,6);
	}
	\node[left] at (0.5,1) {$x_1 \leftarrow$};
	\node[left] at (0.5,2) {$x_2 \leftarrow$};
	\node[left] at (0.5,3) {$\vdots$};
	\node[left] at (0.5,4) {$\vdots$};
	\node[left] at (0.5,5) {$x_n \leftarrow$};
	\node[above] at (7,6) {$\cdots$};
	\node[above] at (6,6) {$\cdots$};
	\node[above] at (5,6) {$\cdots$};
	\node[above] at (4,6) {\footnotesize$\bm{0}$};
	\node[above] at (3,6) {\footnotesize$\bm{0}$};
	\node[above] at (2,6) {\footnotesize$\bm{0}$};
	\node[below] at (7,0) {$\cdots$};
	\node[below] at (6,0) {$\cdots$};
	\node[below] at (5,0) {$\cdots$};
	\node[below] at (4,0) {\footnotesize$\bm{A}(2)$};
	\node[below] at (3,0) {\footnotesize$\bm{A}(1)$};
	\node[below] at (2,0) {\footnotesize$\bm{A}(0)$};
	\node[right] at (8,1) {$0$};
	\node[right] at (8,2) {$0$};
	\node[right] at (8,3) {$\vdots$};
	\node[right] at (8,4) {$\vdots$};
	\node[right] at (8,5) {$0$};
	\node[left] at (1,1) {$1$};
	\node[left] at (1,2) {$2$};
	\node[left] at (1,3) {$\vdots$};
	\node[left] at (1,4) {$\vdots$};
	\node[left] at (1,5) {$n$};
}
\end{equation}
as matrix elements $\bra{\mu}\B_1(x_1)\cdots \B_n(x_n)\ket{\varnothing}$ of $\B_1(x_1)\cdots \B_n(x_n)$. One proves, see \cite[Proposition 5.6.1]{BW}, that these two families of functions are closely related:
\begin{align}
\label{eq:f-g-sym}
g_{\tilde\mu}(x_n^{-1},\dots,x_1^{-1};q^{-1},s^{-1})
=
c_{\mu}(q,s)
\prod_{i=1}^{n}
x_i
\cdot
f_{\mu}(x_1,\dots,x_n;q,s),\qquad \tilde{\mu} = (\mu_n,\dots,\mu_1),
\end{align}
where the multiplicative constant $c_{\mu}(q,s)$ is given by
\begin{align}
c_{\mu}(q,s)
=
\frac{s^n (q-1)^n q^{\#\{i<j : \mu_i \leq \mu_j\}}}{\prod_{j \geq 0} (s^2;q)_{m_j(\mu)}},\qquad 
m_j(\mu):=\#\{1 \leq k \leq n:\mu_k=j\}, \quad j\ge 0.
\end{align}

It will be convenient for us to use a slightly different normalization of the dual functions:
\begin{align*}
g^{*}_{\mu}(x_1,\dots,x_n)
:=
{q^{n(n+1)/2}}{(q-1)^{-n}}
\cdot
g_{\mu}(x_1,\dots,x_n).
\end{align*}

Finally, let us introduce a third family of symmetric rational functions parameterized by a pair of compositions $\mu,\nu$, or rather by a skew composition $\mu/\nu$, by ($p=1,2,\dots$ is arbitrary)
\begin{align}
\label{eq:G-functions}
G_{\mu/\nu}(x_1,\dots,x_p)
&=
\tikz{0.8}{
	\foreach\y in {1,...,5}{
		\draw[lgray,line width=1.5pt,<-] (1,\y) -- (8,\y);
	}
	\foreach\x in {2,...,7}{
		\draw[lgray,line width=4pt,->] (\x,0) -- (\x,6);
	}
	\node[left] at (0.5,1) {$x_1 \leftarrow$};
	\node[left] at (0.5,2) {$x_2 \leftarrow$};
	\node[left] at (0.5,3) {$\vdots$};
	\node[left] at (0.5,4) {$\vdots$};
	\node[left] at (0.5,5) {$x_p \leftarrow$};
	\node[above] at (7,6) {$\cdots$};
	\node[above] at (6,6) {$\cdots$};
	\node[above] at (5,6) {$\cdots$};
	\node[above] at (4,6) {\footnotesize$\bm{B}(2)$};
	\node[above] at (3,6) {\footnotesize$\bm{B}(1)$};
	\node[above] at (2,6) {\footnotesize$\bm{B}(0)$};
	\node[below] at (7,0) {$\cdots$};
	\node[below] at (6,0) {$\cdots$};
	\node[below] at (5,0) {$\cdots$};
	\node[below] at (4,0) {\footnotesize$\bm{A}(2)$};
	\node[below] at (3,0) {\footnotesize$\bm{A}(1)$};
	\node[below] at (2,0) {\footnotesize$\bm{A}(0)$};
	\node[right] at (8,1) {$0$};
	\node[right] at (8,2) {$0$};
	\node[right] at (8,3) {$\vdots$};
	\node[right] at (8,4) {$\vdots$};
	\node[right] at (8,5) {$0$};
	\node[left] at (1,1) {$0$};
	\node[left] at (1,2) {$0$};
	\node[left] at (1,3) {$\vdots$};
	\node[left] at (1,4) {$\vdots$};
	\node[left] at (1,5) {$0$};
}
\end{align}
with $\bm{A}(k) = \sum_{j=1}^{n} \bm{1}_{\mu_j = k} \bm{e}_j$, $\bm{B}(k) = \sum_{j=1}^{n} \bm{1}_{\nu_j = k} \bm{e}_j$, for all $k\in \mathbb{Z}_{\geq0}$. These are matrix elements $\bra{\mu}\B_0(x_1)\cdots\B_0(x_p)\ket{\nu}$  of the operator $\B_0(x_1)\cdots\B_0(x_p)$, and their symmetry with respect to the $x_i$'s is a direct consequence of the commutativity of $\B_0(x_i)$'s noted above. 

\subsection{Colour-blindness}\label{ssec:colour-blind}

For any integer $k \in \{0,1,\dots,n\}$ define its colour-blind projection
\begin{align*}
\theta(k) 
:=
\bm{1}_{k \geq 1}
=
\left\{ \begin{array}{ll} 0, & \quad k=0, \\ 1, & \quad k \geq 1. \end{array} \right.
\end{align*}
Also, denote the set of compositions of a fixed length\footnote{number of entries} with a given weight\footnote{sum of entries} $k$ as
\begin{align*}
\mathcal{W}(k)
:=
\{ \K \in \mathbb{Z}_{\ge 0}^n : |\K| = k \}.
\end{align*}
Then one proves, see \cite[Proposition 2.4.2]{BW}, that
\begin{align}
\label{eq:colour-blindness}
\sum_{\K \in \mathcal{W}(k)}
L_x(\I,j;\K,0)
=
L^{(1)}_x(|\I|,\theta(j);k,0),
\quad
\sum_{\K \in \mathcal{W}(k)}
\sum_{1 \leq \ell \leq n}
L_x(\I,j;\K,\ell)
=
L^{(1)}_x(|\I|,\theta(j);k,1),
\end{align}
where $L^{(1)}$ refers to the weights \eqref{eq:L-weights} with $n=1$. 
Similarly, one has
\begin{align*}
\sum_{\K \in \mathcal{W}(k)}
M_x(\I,j;\K,0)
=
M^{(1)}_x(|\I|,\theta(j);k,0),
\quad
\sum_{\K \in \mathcal{W}(k)}
\sum_{1 \leq \ell \leq n}
M_x(\I,j;\K,\ell)
=
M^{(1)}_x(|\I|,\theta(j);k,1).
\end{align*}
This has been observed in a number of earlier publications \cite{FodaW,GarbaliGW,Kuan}, and used to different effects within those works. This property allows certain linear combinations of higher-rank partition functions to be computable as rank-1, or \emph{colour-blind} partition functions.

As an example, one derives the symmetrization identities, cf. \cite[Proposition 3.4.4]{BW}, \cite[Proposition 4.4.3]{BW},
\begin{align}
\label{eq:symm}
\sum_{\mu: \mu^{+} = \nu}
f_{\mu}(x_1,\dots,x_n)
=
\F^{\sf c}_{\nu}(x_1,\dots,x_n),\quad
\sum_{\mu: \mu^{+} = \kappa}
G_{\nu/\mu}(x_1,\dots,x_p)
=
q^{-np}
\G_{\nu^+/\kappa}(x_1,\dots,x_p),
\end{align} 
with the sums taken over all compositions with a given dominant reordering denoted by the superscript `$+$', and the symmetric rational functions $\F^{\sf c}$ and $\G$ are colour-blind objects that were considered at length in \cite{Borodin, BorodinP1, BorodinP2}.

\subsection{Recursive relations}\label{ssec:recursive} While explicit formulas representing the functions $f_\mu$ as sums of monomials are rather involved, see Chapters 6 and 7 of \cite{BW} for two different versions, there exist concise recursive relations for them. 

First, 	for \emph{anti-dominant} compositions $\delta = (\delta_1 \leq \cdots \leq \delta_n)$, the functions $f_{\delta}$ are completely factorized (cf. \cite[Propostion 5.1.1]{BW}):
\begin{align}
\label{eq:f-delta}
f_{\delta}(x_1,\dots,x_n)
=
\frac{\prod_{j \geq 0} (s^2;q)_{m_j(\delta)}}{\prod_{i=1}^{n} (1-s x_i)}
\prod_{i=1}^{n} \left( \frac{x_i-s}{1-sx_i} \right)^{\delta_i}, \quad  m_j(\delta)=\#\{1\le k\le n:\delta_k=j\}, \  j\ge 0.
\end{align}
This happens because the partition function \eqref{eq:f-function} has only one configuration of paths that contributes nontrivially. 

Second, the following recursion allows one to move step by step from anti-dominant compositions to the dominant ones, cf. \cite[Theorem 5.3.1]{BW}:

Let $\mu=(\mu_1,\dots,\mu_n)$ be a composition with $\mu_i<\mu_{i+1}$ for some $1\le i\le n-1$. Then 
	\begin{equation}
	\label{eq:exchange}
	T_i \cdot f_{\mu}(x_1,\dots,x_n) = f_{(\mu_1,\dots,\mu_{i+1},\mu_i,\dots,\mu_n)}(x_1,\dots,x_n),
	\end{equation}
	where 
	\begin{align}
	\label{eq:DL-operators}
	T_i \equiv q - \frac{x_i-q x_{i+1}}{x_i-x_{i+1}} (1-\mathfrak{s}_i),
	\quad
	1 \leq i \leq n-1,
	\end{align}
	with elementary transpositions \index{s@$\mathfrak{s}_i$; elementary transposition} $\mathfrak{s}_i \cdot h(x_1,\dots,x_n)
	:=h(x_1,\dots,x_{i+1},x_i,\dots,x_n)$, are the Demazure--Lusztig operators of the polynomial representation of the Hecke algebra of type $A_{n-1}$. 	 

\subsection{Summation identities}\label{ssec:Cauchy_ident}
The functions $f_\mu$, $g_\mu$, and $G_{\mu/\nu}$ satisfy several summation identities that can be found in \cite[Section 4]{BW}; in what follows we will need a couple of them. The first one bears a certain similarity to a summation identity proved by Sahi \cite{Sahi-Jack} and Mimachi--Noumi \cite{MimachiN} for non-symmetric Macdonald polynomials \cite{MimachiN}, it was proved as \cite[Theorem 4.3.1]{BW}.

	Let $(x_1,\dots,x_n)$ and $(y_1,\dots,y_n)$ be two sets of complex parameters such that, cf. \eqref{eq:weight-condition},
	\begin{align}
	\label{eq:weight-conditions}
	\left|
	\frac{x_i-s}{1-sx_i}
	\cdot
	\frac{y_j-s}{1-sy_j}
	\right|
	<
	1,
	\qquad
	\text{for all}\ 1 \leq i,j \leq n.
	\end{align}
	Then
	\begin{align}
	\label{eq:mimachi}
	\sum_{\mu}
	f_{\mu}(x_1,\dots,x_n)
	g^{*}_{\mu}(y_1,\dots,y_n)
	=
	\prod_{i=1}^{n}
	\frac{1}{1-x_i y_i}
	\prod_{n \geq i>j \geq 1}
	\frac{1-q x_i y_j}{1-x_i y_j}\,,
	\end{align}
	where the summation is over all compositions $\mu$ (with nonnegative coordinates).

This identity is proved by evaluating the matrix element $\bra{\varnothing}
\C_1(x_1) \dots \C_n(x_n)
\B_1(y_1) \dots \B_n(y_n)
\ket{\varnothing}$ in two different ways, by inserting the partition of unity $\sum _\mu \ket{\mu}\bra{\mu}$ between the groups of $\B$- and $\C$-operators, and by using the commutation relations \eqref{eq:CB}.  

A similar argument applied to the matrix element $\bra{\varnothing}
\C_1(x_1) \dots \C_n(x_n)
\B_0(y_1) \dots \B_0(y_p)
\ket{\nu}$ leads to the following summation identity that is reminiscent of the skew-Cauchy identities known in the theory of symmetric functions (see \cite[Proposition 4.5.1]{BW} for a detailed proof):

Let $(x_1,\dots,x_n)$ and $(y_1,\dots,y_p)$ be two sets of complex parameters satisfying the constraints \eqref{eq:weight-conditions}, and fix a composition $\nu = (\nu_1,\dots,\nu_n)$. Then one has the identity
\begin{align}
\label{eq:fG-skew}
\sum_{\mu}
f_{\mu}(x_1,\dots,x_n)
G_{\mu / \nu}(y_1,\dots,y_p)
=
\prod_{i=1}^{n}
\prod_{j=1}^{p}
\frac{1-q x_i y_j}{q(1-x_i y_j)}
\cdot
f_{\nu}(x_1,\dots,x_n),
\end{align}
where the summation is taken over all length-$n$ compositions 
$\mu = (\mu_1,\dots,\mu_n)$.

\subsection{Orhtogonality and integral representations}\label{ssec:orthogonality}

	Let $\{C_1,\dots,C_n\}$ be a collection of contours in the complex plane. We say that the set $\{C_1,\dots,C_n\}$ is admissible with respect to a pair of complex parameters $(q,s)$ if the following conditions are met:
	\begin{itemize}
		\item The contours $\{C_1,\dots,C_n\}$ are closed, positively oriented and pairwise non-intersecting.
		\item The contours $C_i$ and $q \cdot C_i$ are both contained within contour $C_{i+1}$ for all $1 \leq i \leq n-1$, where $q \cdot C_i$ denotes the image of $C_i$ under multiplication by $q$.
		\item All contours surround the point $s$.
	\end{itemize}
	An illustration of such admissible contours is given in Figure \ref{fig:contours}.

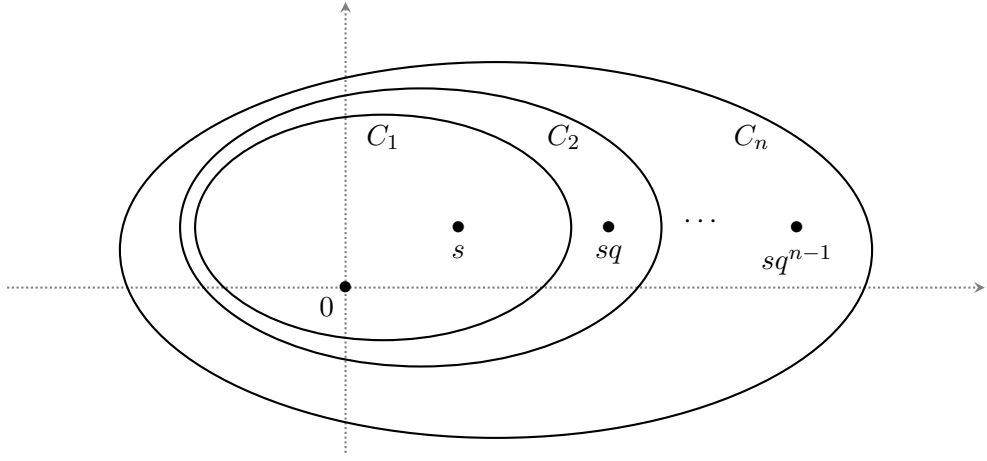
\begin{figure}
	\begin{tikzpicture}[>=stealth]
	\draw[gray,densely dotted,thick,->] (-1,2.2) -- (12,2.2);
	\draw[gray,densely dotted,thick,->] (3.5,0) -- (3.5,6);
	\draw[thick] (5.5,2.7) ellipse [x radius=5cm, y radius=2.5cm];
	\draw[thick] (4.5,3) ellipse [x radius=3.2cm, y radius=1.85cm];
	\draw[thick] (4,3) ellipse [x radius=2.5cm, y radius=1.5cm];
	\node at (3.25,1.95) {$0$};
	\node at (8.9,4.2) {$C_n$};
	\node at (6.4,4.2) {$C_2$};
	\node at (4,4.2) {$C_1$};
	\foreach\x in {5,7,9.5}{
		\node at (\x,3) {$\bullet$};
	}
	\node at (3.5,2.2) {$\bullet$};
	\node[below,text centered] at (5,2.9) {$s$};
	\node[below,text centered] at (7,2.9) {$sq$};
	\node[below,text centered] at (8.25,3.3) {$\cdots$};
	\node[below,text centered] at (9.5,2.9) {$sq^{n-1}$};
	\end{tikzpicture}
	\caption{Admissible contours $\{C_1,\dots,C_n\}$ with respect to $(q,s)$.}
	\label{fig:contours}
\end{figure}
Often when we integrate rational functions over $\{C_1,\dots,C_n\}$, the integrals can also be computed as sums of residues of the integrand inside the contours. Such sums also make sense for values of parameters that prevent admissible contours from existing, and thus the integrals could also be defined via the residue sums. Therefore, we will tacitly assume that we perform such a replacement should the admissible contours not exist, and we will also use a similar convention for other contour integrals below, cf. Remark \ref{rmk:sum-of-residues}.

Let $\mu = (\mu_1,\dots,\mu_n)$ and $\nu = (\nu_1,\dots,\nu_n)$ be two compositions. We have the following orthonormality of non-symmetric spin Hall--Littlewood functions (see \cite[Theorem 8.2.1]{BW}):
\begin{multline}
\label{eq:f-g-orthog}
 \frac{1}{(2\pi\sqrt{-1})^n} 
\oint_{C_1}
\frac{dx_1}{x_1}
\cdots 
\oint_{C_n}
\frac{dx_n}{x_n}
\prod_{1 \leq i<j \leq n}
\frac{x_j-x_i}{x_j-q x_i}
\,
f_{\mu}({x}_1^{-1},\dots,{x}_n^{-1})
g^{*}_{\nu}(x_1,\dots,x_n)
=
{\bm 1}_{\mu = \nu}.
\end{multline}

Coupled with \eqref{eq:mimachi} (which is easily shown to converge uniformly provided that the left-hand sides of \eqref{eq:weight-conditions} are bounded by a uniform constant $<1$), this leads to the integral formula 
\begin{multline}
\label{eq:f-integral-repr}
f_\mu (x_1,\dots,x_n)= \frac{1}{(2\pi\sqrt{-1})^n} 
\oint_{C_1}
\frac{dy_1}{y_1}
\cdots 
\oint_{C_n}
\frac{dy_n}{y_n}
\prod_{1 \leq i<j \leq n}
\frac{y_j-y_i}{y_j-q y_i}\, f_{\mu}({y}_1^{-1},\dots,{y}_n^{-1}) \\ 
\times \prod_{i=1}^{n}
\frac{1}{1-x_i y_i}
\prod_{n \geq i>j \geq 1}
\frac{1-q x_i y_j}{1-x_i y_j}\,.
\end{multline}

A similar integral representation for $G_{\mu/\nu}$ originating from 
\eqref{eq:fG-skew} can be found in \cite[Section 9.5]{BW}. 

\subsection{Coloured compositions}\label{ssec:coloured-comp} This section closely follows \cite[Sections 3.3-3.4]{BW}, where a more detailed exposition can be found. 

Let $\lambda = (\lambda_1,\dots,\lambda_n)$ be a composition of length $n$ and weight $m$: $|\lambda|=\sum_{i=1}^{n} \lambda_i = m$. Denote the partial sums of $\lambda$ by $\sum_{i=1}^{k} \lambda_i = \ell_k$. We introduce the set $\mathcal{S}_{\lambda}$ of $\lambda$-coloured compositions as follows:
\begin{align}
\label{eq:lambda-col}
\mathcal{S}_{\lambda}
=
\Big\{ 
\mu 
= 
(\mu_1 \geq \cdots \geq \mu_{\ell_1} |
\mu_{\ell_1+1} \geq \cdots \geq \mu_{\ell_2} |
\cdots |
\mu_{\ell_{n-1}+1} \geq \cdots \geq \mu_{\ell_n})
\Big\}.
\end{align}
That is, the elements of $\mathcal{S}_{\lambda}$ are length-$m$ compositions $\mu$, which have been subdivided into blocks of length $\lambda_k$, $1 \leq k \leq n$. These blocks demarcate the colouring of $\mu$. Within any given block, the parts of $\mu$ have the same colouring and are weakly decreasing.

Two special cases of $\lambda$-coloured compositions play special roles. The first is when $\lambda = (n,0,\dots,0)$, when compositions $\mu \in \mathcal{S}_{\lambda}$ consist of a single block whose parts are weakly decreasing; \ie\ one simply recovers partitions. As was noted in \eqref{eq:symm}, reducing to this case recovers the symmetric rational functions $\F_{*}$, $\G_{*}$ from \cite{Borodin,BorodinP1,BorodinP2}.

The second one is when $\lambda = (1,1,\dots,1) = (1^n)$. Then compositions $\mu \in \mathcal{S}_{\lambda}$ consist of $n$ blocks, each of a different colour. Thus, the parts of $\mu$ are not bound by any inequalities; accordingly, one recovers the set of all length-$n$ compositions. We will refer to these as {\it rainbow compositions}, or as composition in the \emph{rainbow sector}. The functions $f_\mu$, $g_\mu$ and $G_{\mu/\nu}$ introduced above were all defined under the assumption that the participating compositions were in the rainbow sector. 

Let $\mu \in \mathcal{S}_{\lambda}$ be a $\lambda$-coloured composition, with $\ell_k$ denoting the partial sums of $\lambda$, as above. We associate to each such $\mu$ a vector $\ket{\mu}_{\lambda} \in \mathbb{V}$, defined as follows:
\begin{align}
\label{eq:A(k)-coloured}
\ket{\mu}_{\lambda}
:=
\bigotimes_{k=0}^{\infty}
\ket{\bm{A}(k)}_k,
\qquad
\bm{A}(k) = \sum_{j=1}^{n} A_j(k) \bm{e}_j,
\qquad
A_j(k)
=
\#\{ i : \mu_i = k,\ \ell_{j-1} +1 \leq i \leq \ell_j \},
\end{align}
where by agreement $\ell_{0} = 0$. In other words, the component $A_j(k)$ enumerates the number of parts in the $j$th block of $\mu$ (these are the parts of colour $j$) which are equal to $k$. Further, we define vector subspaces $\mathbb{V}(\lambda)$ of $\mathbb{V}$ which provide a natural grading of $\mathbb{V}$:
\begin{align}
\label{eq:filter}
\mathbb{V} 
= 
\bigoplus_{m=0}^{\infty}
\bigoplus_{|\lambda|=m}
\mathbb{V}(\lambda),
\quad\quad
\mathbb{V}(\lambda) 
:=
{\rm Span}_{\mathbb{C}} 
\Big\{ \ket{\mu}_{\lambda} \Big\}_{\mu \in \mathcal{S}_{\lambda}}.
\end{align}
 The grading \eqref{eq:filter} splits $\mathbb{V}$ into subspaces with fixed particle content: $\mathbb{V}(\lambda)$ is the linear span of all states consisting of $\lambda_i$ particles of colour $i$, for all $i\ge 1$. We refer to these subspaces as {\it sectors} of $\mathbb{V}$, thus generalizing the `rainbow sector' terminology.
 
 The definitions of the rational functions $f_\mu$, $g_\nu$, and $G_{\mu/\nu}$ naturally lift to coloured composition labels. Concretely,
 let $\lambda = (\lambda_1,\dots,\lambda_n)$ be a composition of weight $m$ with partial sums $\ell_k$ as above, and fix a $\lambda$-coloured composition $\mu = (\mu_1,\dots,\mu_m) \in \mathcal{S}_{\lambda}$. In analogy with \eqref{eq:f-function}, set
 \begin{align}
 \label{eq:f-generic}
 f_{\mu}(\lambda; x_1,\dots,x_m)
 :=
 \bra{\varnothing}
 \left(
 \prod_{i=1}^{\ell_1}
 \C_1(x_i)
 \right)
 \left(
 \prod_{i=\ell_1+1}^{\ell_2}
 \C_2(x_i)
 \right)
 \cdots
 \left(
 \prod_{i=\ell_{n-1}+1}^{\ell_n}
 \C_n(x_i)
 \right)
 \ket{\mu}_{\lambda},
 \end{align}
 where $\ket{\mu}_{\lambda} \in \mathbb{V}(\lambda)$ is given by \eqref{eq:A(k)-coloured}, and $\bra{\varnothing} \in \mathbb{V}^{*}$ denotes the (dual) vacuum state
 $\bra{\varnothing}
 =
 \bigotimes_{k=0}^{\infty}
 \bra{\bm{0}}_k,$ %
 which is completely devoid of particles, and $\C_i$'s are the row-operators \eqref{C-row}. Graphically, this is the partition function of the form \eqref{eq:f-function}, with the incoming paths in the bottom $\ell_1$ rows having colour 1, having colour 2 in the $\ell_2$ rows right above those, etc.
 
 One similarly defines, in analogy with \eqref{eq:g-function} and using the row operators \eqref{B-row},
 \begin{align}
 \label{eq:generic-g}
 g_{\mu}(\lambda; x_1,\dots,x_m)
 :=
 \bra{\mu}_{\lambda}
 \left(
 \prod_{i=1}^{\ell_1}
 \B_1(x_i)
 \right)
 \left(
 \prod_{i=\ell_1+1}^{\ell_2}
 \B_2(x_i)
 \right)
 \cdots
 \left(
 \prod_{i=\ell_{n-1}+1}^{\ell_n}
 \B_n(x_i)
 \right)
 \ket{\varnothing},
 \end{align}
 where $\bra{\mu}_{\lambda} \in \mathbb{V}^{*}(\lambda)$ is the dual of the vector \eqref{eq:A(k)-coloured}, and $\ket{\varnothing} \in \mathbb{V}$ denotes the vacuum state
$\ket{\varnothing}
 =
 \bigotimes_{k=0}^{\infty}
 \ket{\bm{0}}_k.$ Again, here the exiting paths in the bottom $\ell_1$ rows have colour 1, in the next $\ell_2$ rows they have colour 2, etc. 
 
 Finally, in analogy with \eqref{eq:G-functions}, for $\mu,\nu\in \mathcal{S}_{\lambda}$ we define
\begin{equation}
\label{eq:G-generic}
G_{\mu/\nu}(\lambda;x_1,\dots,x_p)=\bra{\mu}_\lambda\B_0(x_1)\cdots\B_0(x_p)\ket{\nu}_\lambda,
\end{equation}
with 
\begin{gather*}
\bra{\mu}_{\lambda}
:=
\bigotimes_{k=0}^{\infty}
\bra{\bm{A}(k)}_k,
\quad
\bm{A}(k) = \sum_{j=1}^{n} A_j(k) \bm{e}_j,
\qquad
\ket{\nu}_{\lambda}
:=
\bigotimes_{k=0}^{\infty}
\bra{\bm{B}(k)}_k,
\quad
\bm{B}(k) = \sum_{j=1}^{n} B_j(k) \bm{e}_j,\\
A_j(k)
=
\#\{ i : \mu_i = k,\ \ell_{j-1} +1 \leq i \leq \ell_j \},\qquad
B_j(k)
=
\#\{ i : \nu_i = k,\ \ell_{j-1} +1 \leq i \leq \ell_j \},
\end{gather*}
and the graphical depiction \eqref{eq:G-functions} does not require any modifications. Since the colouring composition $\lambda$ only participates in this definition as a record of the colours in the boundary states $\mu$ and $\nu$, we will often omit it from the notation of $G_{\mu/\nu}$. 

\section{Extensions}\label{sec:extensions}

In this section we provide a few straightforward extensions of the results from Section \ref{sec:prelim}, most of which have very similar proofs. 

\subsection{Column inhomogeneities}\label{ssec:inhom} In \eqref{eq:f-function} 
we defined the functions $f_\mu$ by utilizing the weights $L_{x}$ of 
\eqref{eq:L-weights} with $x=x_i$ and a fixed spin parameter $s$ for all the 
vertices in the $i$th row, $1\le i\le n$. It is meaningful, however, to extend 
the definition when we take $x=x_i\xi_j$ and $s=s_j$ for the vertex in the $i$th 
row and $j$th column, where $1\le i\le n$, $0\le j<+\infty$, and 
$\{\xi_j\}_{j\ge 0}$ and $\{s_j\}_{j\ge 0}$ are two infinite sequences of complex parameters. Correspondingly, in the definitions \eqref{eq:g-function} and \eqref{eq:G-functions} of $g_\mu$ and $G_{\mu/\nu}$ one needs to use the weights $M_x$ of \eqref{eq:M-weights} with $x=x_i\xi_j^{-1}$ and $s=s_j$ for the vertex in the $i$th row and $j$th column, with the same sequences $\{\xi_j\}_{j\ge 0}$ and $\{s_j\}_{j\ge 0}$.   

Many of the results cited in Section \ref{sec:prelim} and their proofs extend to such an inhomogeneous setup almost \emph{verbatim}. In the colour-blind case of $n=1$ this can be seen by comparing \cite{BorodinP1} and \cite{BorodinP2}. 

The basic reason for such an easy extension lies in the fact that the Yang-Baxter equation \eqref{graph-RLLb} has a suitable extension. More exactly, the needed deformed equation has the form
\begin{multline}
\label{eq:inhom-YB}
\sum_{0 \leq k_1,k_3 \leq n}
\
\sum_{\K \in \mathbb{N}^n}
L_{x\xi}(\I,i_1;\K,k_1)
R_{(qxz)^{-1}}(i_3,k_1;k_3,j_1)
M_{z\xi^{-1}}(\K,k_3;\J,j_3)
\\
=\sum_{0 \leq k_1,k_3 \leq n}
\
\sum_{\K \in \mathbb{N}^n}
M_{z\xi^{-1}}(\I,i_3;\K,k_3)
R_{(qxz)^{-1}}(k_3,i_1;j_3,k_1)
L_{x\xi}(\K,k_1;\J,j_1),
\end{multline}
where the important thing to notice is that the parameters of the $R$-weights in the middle remain independent of $\xi$, and they are also independent of $s$ by their definition \eqref{fund-vert}\footnote{This version of the Yang-Baxter equation is also a consequence of the `master' Yang-Baxter equation \eqref{graph-master}.}. 

Let us quickly go through inhomogeneous analogs of the results from Section \ref{sec:prelim} that we will need. 

The commutation relations \eqref{eq:CB} remain unchanged, as they are related to the $R$-matrix in \eqref{eq:inhom-YB}, but the convergence condition \eqref{eq:weight-condition} needs to be modified to 
\begin{equation*}
\lim_{L\to\infty}\prod_{j=0}^L\left| \frac{\xi_jx-s_j}{1-s_j\xi_jx}\cdot \frac{y-s_j\xi_j}{\xi_j-s_jy}\right|=0, 
\end{equation*}
see \cite[Proposition 4.8]{BorodinP1} for an explanation in the $n=1$ case. In what follows, we will only need the situation of finitely many (in fact, exactly one) $\xi_j$'s different from 1 and $s_j$'s different from a certain fixed $s$, in which case, we can clearly continue using \eqref{eq:weight-condition}. 

Correspondingly, the summation identities \eqref{eq:mimachi} and \eqref{eq:fG-skew} also remain unchanged, modulo a similar comment about the convergence condition \eqref{eq:weight-conditions}.

The relation \eqref{eq:f-g-sym} between $f$'s and $g$'s remains valid with the following modifications -- the inversion of $s$ for the $g$ in the left-hand side needs to be applied to all the $s_j$'s, and the multiplicative constant $c_\mu$ needs to be read as 
$$
\frac{\prod_{j\ge 0} s_j\cdot  (q-1)^n q^{\#\{i<j : \mu_i \leq \mu_j\}}}{\prod_{j \geq 0} (s_j^2;q)_{m_j(\mu)}}\,.
$$
The proof remains the same. 

The colour-blindness results \eqref{eq:symm} remain in place with identical proofs and inhomogeneous $\F$ and $\G$ understood as in \cite{BorodinP1}. 

The factorization \eqref{eq:f-delta} for the anti-dominant compositions looks very similar (and the same argument works):
\begin{equation}
\label{eq:f-delta-inhom}
f_{\delta}(x_1,\dots,x_n)
=
\prod_{j \geq 0} (s_j^2;q)_{m_j(\delta)}\cdot 
\prod_{i=1}^{n} \left(\frac1{1-s_{\delta_i}\xi_{\delta_i} x_i} \prod_{j=0}^{\delta_i-1}  \frac{\xi_jx_i-s_j}{1-s_j\xi_jx_i}\right).
\end{equation}

The recurrence relation \eqref{eq:exchange} remains unchanged, with the action of $T_i$'s given by the same Demazure-Lusztig operators \eqref{eq:DL-operators}. Its proof was based on the commutation relations between $\C_i$'s given in \cite[Theorem 3.2.1]{BW} and on \cite[Proposition 5.3.3]{BW}, both of which remain intact in the inhomogeneous setup. 

Finally, the orthogonality relation \eqref{eq:f-g-orthog} remains literally the same, with the contours surrounding all the points $\{\xi_j s_j\}_{j\ge 0}$ instead of just $s$ in the homogeneous case.  It is a bit more difficult to convince oneself that this is so because this proof relies on many ingredients. In addition to the facts already mentioned above, one needs monomial expansions of \cite[Chapter 6]{BW} and a certain explicit contour integral computation. The latter in the inhomogeneous setup is exactly \cite[Lemma 7.1]{BorodinP1}, while the proof of the former carries over to the inhomogeneous case in the same spirit as all the other above-mentioned facts.

\subsection{A simplifying specialization of $G_{\nu/\mu}$}\label{ssec:principal} Our next goal is to prove the following statement. 

 \begin{prop}\label{prop:cG}
Fix generic complex parameters $q$ and $s$, a composition $\lambda$, and two $\lambda$-coloured compositions $\mu,\nu$ of length $n\ge 1$. Then for 
$\l\in\mathbb{N}$ and generic $\epsilon\in\mathbb{C}$, 
$G_{\nu/\mu}(\lambda;\epsilon,q\epsilon,\dots,q^{\l-1}\epsilon)$, with $G$ defined as in 
\eqref{eq:G-generic}, is a rational function in $q^\l$, and there exists a 
limit
\begin{equation}
\label{eq:cG-limit}
\cG_{\nu/\mu}:=\lim_{\epsilon\to 0} \left(q^{n\l}\cdot G_{\nu/\mu}(\lambda;\epsilon,q\epsilon,\dots,q^{\l-1}\epsilon)\right) {\Bigl|}_{q^\l=(s\epsilon)^{-1}},
\end{equation} 
where in the right-hand side we substituted a particular value into a rational function. Explicitly, $\cG_{\nu/\mu}$ has the form
\begin{equation}
\label{eq:cG-formula}
\frac{(-s)^{|\mu|}}{(-s)^{|\nu|}}\,\cG_{\nu/\mu}=\begin{cases}
               s^{-2n} \displaystyle\prod_{x=0}^\infty 
\left((s^2;q)_{|\bm{m}^{(x)}|}q^{-\sum_{i>j}m_i^{(x)}m_j^{(x)}} 
\displaystyle\prod_{i=1}^n 
q^{m_i^{(x)}H_{>i}^{\nu/\mu}(x+1)}{\displaystyle\binom{H_i^{\nu/\mu}(x+1)}{m_i^{
(x) } } } _q \right), &\text {all $\nu_i\ne 0$,} 
               \\
              0,&\text{otherwise,}
              \end{cases}
\end{equation}
	where we defined, for each $x\ge 0$, an $n$-component 
vector $\bm{m}^{(x)}$ whose $i$th component $m^{(x)}_i$, $1\le i\le n$, is equal to 
the number of parts of $\mu$ of colour $i$ that are equal to $x$ (symbolically 
$\mu=0^{\bm{m}^{(0)}}1^{\bm{m}^{(1)}}2^{\bm{m}^{(2)}}\cdots$), 
and also \emph{coloured height functions}
\begin{equation}
 \label{eq:coloured-height}
 H_i^\varkappa(x)=\#\{j:\text{colour}(\varkappa_j)=i,\,\varkappa_j\ge x\},\qquad H_{>i}^\varkappa=\sum_{k> 
i}H_k^\varkappa(x),\qquad H^{\nu/\mu}_*\equiv H^\nu_*-H^\mu_*. 
\end{equation}
The same conclusion holds in the inhomogeneous setting of Section 
\ref{ssec:inhom} under the condition that for each column $x\ge 0$ such that
there exists a part of $\mu$ equal to $x$ (\emph{i.~e.}, $|m^{(x)}|>0$), the 
column rapidity and spin parameter in that column are still equal to 1 and $s$, 
respectively, as in the homogeneous case, and the factor $(-s)^{|\mu|-|\nu|}$ in the left-hand side of \eqref{eq:cG-formula} is replaced by 
$\prod_{i=1}^n\prod_{j=\mu_i}^{\nu_i-1} (-s_j)^{-1}.$ 
\end{prop}
\begin{rmk}
We tacitly follow the convention that the $q$-binomial coefficients 
vanish unless their arguments are nonnegative, and the top one is at least as
large as the bottom one. Thus, the top line of \eqref{eq:cG-formula} can also
produce a zero outcome. 
\end{rmk}

\begin{proof}
Let us first switch from using the weights $M_x$ of \eqref{eq:M-weights} in the partition function of the form \eqref{eq:G-functions} to using the weights $M_x^\stoch$ of \eqref{eq:M-stoch} instead. The correcting factor $(-s)^{{\bf 1}_{j\ge 1}}$ in \eqref{eq:M-weights} taken over all vertices gives $(-s)^{|\nu|-|\mu|}$ (the exponent counts the number of horizontal steps of all the paths, which is exactly $|\nu|-|\mu|$).
Note that this matches the inverse of $(-s)^{|\mu|-|\nu|}$ in the left-hand side of \eqref{eq:cG-formula}, and in the inhomogeneous setup the product of these correcting factors gives $\prod_{i=1}^n\prod_{j=\mu_i}^{\nu_i-1} (-s_j)$ instead.  

With the weights $M^\stoch_x$, it is a bit more convenient to reflect the partition function with respect to a vertical axis so that paths move to the right horizontally. Then, noticing that the left and right boundary conditions of \eqref{eq:G-functions} consist of having no entering or exiting paths, we recognize that using the geometric progression $(\epsilon,q\epsilon,\dots,q^{\l-1}\epsilon)$ for horizontal rapidities is equivalent to performing the stochastic fusion in the vertical direction (corresponding to the outer sum in \eqref{eq:fusion}). Hence, we obtain 
that $(-s)^{|\mu|-|\nu|}G_{\nu/\mu}(\lambda;\epsilon,q\epsilon,\dots,q^{\l-1}\epsilon)$
is a one-row partition function with the incoming paths from the bottom parametrizing $\nu$, outgoing paths on the top parametrizing $\mu$, no paths entering or exiting on either side, and the vertex weights given by 
\begin{equation}
\label{eq:W-modified}
\left.
\vert{\l}{\m}{\frac{s}{\epsilon};\frac{1}{q}}{\AA}{\bm{e}_b}{\CC}{\bm{e}_d}
\right|_{q^{\m} \rightarrow s^{-2}}.
\end{equation}
Here $s$ and $\epsilon$ need to be replaced by $s_j$ and $\xi_j^{-1}\epsilon$, if the vertex is in a column $j$ that carries inhomogeneities $(s_j,\xi_j)$. Pictorially,
\begin{equation}
\label{eq:one-row}
(-s)^{|\mu|-|\nu|}G_{\nu/\mu}(\lambda;\epsilon,q\epsilon,\dots,q^{\l-1}\epsilon)=
\tikz{0.8}{
	\foreach\y in {1}{
		\draw[lgray,line width=4pt,->] (1,\y) -- (8,\y);
	}
	\foreach\x in {2,...,7}{
		\draw[lgray,line width=4pt,->] (\x,0) -- (\x,2);
	}
	\node[below] at (7,0) {\footnotesize$\bm{J}(0)$};
	\node[below] at (6,0) {\footnotesize$\bm{J}(1)$};
	\node[below] at (5,0) {\footnotesize$\bm{J}(2)$};
	\node[below] at (4,0) {$\cdots$};
	\node[below] at (3,0) {$\cdots$};
	\node[below] at (2,0) {$\cdots$};
	\node[above] at (7,2) {\footnotesize$\bm{I}(0)$};
	\node[above] at (6,2) {\footnotesize$\bm{I}(1)$};
	\node[above] at (5,2) {\footnotesize$\bm{I}(2)$};
	\node[above] at (4,2) {$\cdots$};
	\node[above] at (3,2) {$\cdots$};
	\node[above] at (2,2) {$\cdots$};
	\node[right] at (8,1) {$\bm{0}$};
	\node[left] at (1,1) {$\bm{0}$};
}
\end{equation}
with $\bm{I}(k) = \sum_{i=1}^{n} \bm{1}_{\mu_i = k} \bm{e}_i$, $\bm{J}(k) = 
\sum_{j=1}^{n} \bm{1}_{\nu_j = k} \bm{e}_j$, for all $k\in 
\mathbb{Z}_{\geq0}$.
This partition function is a rational function in $q^\l$ because 
every vertex weight \eqref{eq:W-modified} is, cf. \eqref{eq:W-weights}. 

Assume we are in the homogeneous setting first. Observe that if in the right-hand side of \eqref{eq:W-weights} we have $q^{\l-\m}x=1$, then the only nontrivially contributing term comes from $\bm{P}=\CC$, for which the first $\Phi$-factor turns into 1. 
In terms of our weights \eqref{eq:W-modified} this corresponds to 
$q^{\m-\l}s/\epsilon=1$, which is realized 
by setting $q^\l=(s\epsilon)^{-1}$ (recall that $q^\m=s^{-2}$). 

Writing out the term with $\bm{P}=\CC$ we obtain 
\begin{multline}
\label{eq:epsilon-vertex}
\bm{1}_{\AA+\BB=\CC+\DD}\bm\cdot \left(\frac 
s\epsilon\right)^{|\DD|-|\BB|}(s\epsilon)^{|\AA|} 
s^{-2|\DD|}
\\ \times
\frac{(s^{-2};q^{-1})_{|\CC|}(s\epsilon^{-1};q^{-1})_{|\BB|-|\CC|}}	
{((s\epsilon)^{-1};q^{-1})_{|\BB|}} \left(\frac 
s\epsilon\right)^{|\CC|}q^{-\sum_{i<j}(B_i-C_i)C_j}
\prod_{i=1}^n {\binom{B_i}{C_i}}_{q^{-1}}.
\end{multline}
There is an additional factor of $q^{n\l}=(s\epsilon)^{-n}$ in 
\eqref{eq:cG-limit}. Since our compositions have $n$ parts, there is a total
of $n$ paths exiting \eqref{eq:one-row} from the top, and we 
can distribute $(s\epsilon)^{-n}$ as $(s\epsilon)^{-|\CC|}$ over 
all vertices in \eqref{eq:one-row} (with $\CC$ corresponding to the paths
exiting the vertex at the top). Hence, we need to take the limit
as $\epsilon\to 0$ of \eqref{eq:epsilon-vertex} multiplied by 
$(s\epsilon)^{-|\CC|}$. This is a straightforward calculation which yields
\begin{equation}
\label{eq:limit-vertex}
 \bm{1}_{\AA+\BB=\CC+\DD}\cdot s^{-2|\CC|}(s^2;q)_{|\CC|}\cdot 
q^{\sum_{i>j}(B_i-C_j)C_j}\prod_{i=1}^n {\binom{B_i}{C_i}}_q.
\end{equation}

For the inhomogeneous setting, our hypothesis implies that for any vertex in
a column with inhomogeneity parameters $(\tilde s, \xi)$ we must have $\CC=\bm{0}$.
This turns the sum in \eqref{eq:W-weights} into a single term with $\bm{P}=\bm{0}$,
which leads to the replacement of \eqref{eq:epsilon-vertex} by 
\begin{equation}
\label{eq:epsilon-vertex2}
\bm{1}_{\AA+\BB=\CC+\DD}\bm\cdot \left(
\xi^{-1}\frac  s\epsilon\right)^{|\DD|-|\BB|}(s\epsilon)^{|\AA|} 
\tilde s^{-2|\DD|}
\frac{((s\epsilon)^{-1};q^{-1})_{|\DD|}}{(\xi^{-1}\tilde s^{-2}s\epsilon^{-1};q^{-1})_{|\DD|}}
\frac{(\xi^{-1}s\epsilon^{-1};q^{-1})_{|\BB|}}	
{((s\epsilon)^{-1};q^{-1})_{|\BB|}}.
\end{equation} 
This expression has the exact same $\epsilon\to 0$ asymptotics as \eqref{eq:epsilon-vertex} with $\CC=\bm{0}$, thus confirming \eqref{eq:limit-vertex} for the inhomogeneous setup as well.  

It remains to interpret $\BB$ and $\CC$ in \eqref{eq:limit-vertex} in terms
of the height functions. Assume that our vertex is located in the column with 
coordinate $x\ge 0$ (which means, in particular, that 
$\AA=\bm{J}(x),\CC=\bm{I}(x)$ in
the notation of \eqref{eq:one-row}). Then $B_i$ is the number of paths of 
colour $i$ in \eqref{eq:one-row} that enter from the bottom at a location 
strictly to the left of $x$ minus the number of paths of 
colour $i$ that exit through the top at a location 
strictly to the left of $x$, and this equals $H_i^{\nu/\mu}(x+1)$ as 
defined in 
\eqref{eq:coloured-height}. Furthermore, $C_i$ is exactly $m_i^{(x)}$. 
Hence, the part of \eqref{eq:limit-vertex} past the indicator function takes 
the form
\begin{equation}
s^{-2|\bm{m}^{(x)}|}(s^2;q)_{|\bm{m}^{(x)}| } \cdot q^{\sum_{j\ge 1} 
m_j^{(x)}H_{>j}^{\nu/\mu}(x+1)-\sum_{i>j}m_i^{(x)}m_j^{(x)}}\prod_{i=1}^n 
\binom{H_i^{\nu/\mu}(x+1)}{m_i^{(x)}}_q,
\end{equation}
and the product of these expressions over $x\ge 0$ gives the first line of the right-hand side of 
\eqref{eq:cG-formula}. 

On the other hand, the role of the indicator function $\bm{1}_{\AA+\BB=\CC+\DD}$
in \eqref{eq:limit-vertex} is in providing a recipe for uniquely assigning the 
paths along the horizontal edges of \eqref{eq:one-row} inductively: $\DD$ is 
assigned the value of $\AA+\BB-\CC$, starting from the 
far left where no paths are present. This works smoothly until column 0,
in which $\DD$ must be equal to $\bm{0}$, enforcing $\AA+\BB=\CC$. Since the 
$q$-binomial coefficients require $B_i\ge C_i$ for a nonzero outcome, we 
conclude that $\AA$ must be $\bm{0}$ in the 0th column, giving us the second 
line 
of the right-hand side of \eqref{eq:cG-formula}.
\end{proof}

An important special case of $\cG_{\nu/\mu}$ is when $\mu$ has only zero parts. 
A direct inspection of \eqref{eq:cG-formula} leads to the following 

\begin{cor}
For $\mu=0^n=(0,0,\dots,0)$, with notation $\cG_\nu:=\cG_{\nu/0^n}$, the limit \eqref{eq:cG-limit} takes the form
\begin{equation}
\label{eq:cG-nu}
(-s)^{-|\nu|}\,\cG_\nu=\begin{cases}
               s^{-2n} 
(s^2;q)_{n}, &\text {all $\nu_i\ne 0$,} 
               \\
              0,&\text{otherwise.}
              \end{cases}
\end{equation}
In the inhomogeneous setting, assuming that $(s_0,\xi_0)=(s,1)$, a similar formula holds, where one needs to replace $(-s)^{-|\nu|}$ in the left-hand side by $\prod_{i=1}^n\prod_{j=0}^{\nu_i-1} (-s_j)^{-1}$.  
\end{cor}

The second symmetrization relation of \eqref{eq:symm} implies that
$\cG_\nu$ has to equal 
$$
\lim_{\epsilon\to 0} \G_{\nu^+} 
(\epsilon,q\epsilon,\dots,q^{\l-1}\epsilon)|_{q^\l=(s\epsilon)^{-1}}
$$
with $\nu^+$ denoting the dominant reordering of $\nu$, and $\G$ as in 
\cite{BorodinP1}. This limit was evaluated in \cite[Proposition 6.7]{BorodinP1}, and was shown to be equivalent to \eqref{eq:cG-nu}.\footnote{In fact, \cite{BorodinP1} provided two different evaluations for this limit, neither of which coincides with the color-blind version of the proof of Proposition \ref{prop:cG} above.}

Let us also record for the future what happens to \eqref{eq:cG-formula} when $\mu$ and $\nu$ are rainbow compositions (equivalently, $\lambda=(1,1,\dots,1)$). 

\begin{cor}\label{cor:cG-rainbow} Under the assumption that $\mu$ and $\nu$ are rainbow compositions, \eqref{eq:cG-formula} takes the form
\begin{equation}
\label{eq:cG-rainbow}
\frac{(-s)^{|\mu|}}{(-s)^{|\nu|}}\,\cG_{\nu/\mu}=\begin{cases}
s^{-2n} \displaystyle\prod_{x=0}^\infty 
\left((s^2;q)_{|\bm{m}^{(x)}|}q^{-\binom{|\bm{m}^{(x)}|}{2}} 
\displaystyle\prod_{\substack{1\le i\le n\\ m_i^{(x)}=1}} 
\bm{1}_{H_i^{\nu/\mu}(x+1)=1} \,q^{H_{>i}^{\nu/\mu}(x+1)}\right), &\text {all $\nu_i\ne 0$,} 
\\
0,&\text{otherwise.}
\end{cases}
\end{equation}
The same formula holds in the inhomogeneous setting of Section 
\ref{ssec:inhom} under the condition that for each column $x\ge 0$ such that
there exists a part of $\mu$ equal to $x$ (\emph{i.~e.}, $|\bm{m}^{(x)}|>0$), the 
column rapidity and spin parameter in that column are still equal to 1 and $s$, 
respectively, and the factor $(-s)^{|\mu|-|\nu|}$ in the left-hand side is replaced by 
$\prod_{i=1}^n\prod_{j=\mu_i}^{\nu_i-1} (-s_j)^{-1}.$ 
\end{cor}
\begin{proof} Direct inspection of \eqref{eq:cG-formula}. Note that in the rainbow sector, for any given $i$, $1\le i\le n$, the $i$th colour height function $H_i^*$ defined as in \eqref{eq:coloured-height}, can only take values $0$ or $1$.
\end{proof}

\subsection{Colour merging}\label{ssec:colour-merge} Different versions of  
colour merging properties of vertex weights have been previously observed and 
studied in several works including \cite{FodaW,GarbaliGW,Kuan,BW,BGW}. We use 
this section to formulate the statements we need in suitable notation.   

Let $n_1,n_2$ be two positive integers, and let $\theta:\{1,\dots,n_1\}\to\{1,\dots,n_2\}$ be an arbitrary monotone map. 
It induces a map $\theta_*$ that turns a $n_1$-dimensional vector into an $n_2$-dimensional one as follows: 
\begin{equation}
\label{eq:theta-star}
\theta_*:\bm{I}=(I_1,\dots,I_{n_1}) \mapsto \bm{J}=(J_1,\dots,J_{n_2}) \qquad \text{if}
\qquad J_j=\sum_{i\in\theta^{-1}(j)}I_i.
\end{equation}
In other words, we sum the coordinates of $\bm{I}$ that have the same $\theta$-image and turn the result into a coordinate of $\bm{J}$ whose index is that image. Empty sums are interpreted as having value 0.  
\begin{prop}\label{prop:merge} Denote the weights $W_{\l,\m}$ given by 
\eqref{eq:W-weights} with $n$-dimensional vector arguments as $W_{\l,\m}^{(n)}$. 
Then for any $n_1,n_2\ge 1$ and a map $\theta$ as above, we have the following 
colour merging relation: For any $\AA,\BB\in \mathbb{Z}_{\ge 0}^{n_1}$ and 
$\widetilde\CC,\widetilde\DD\in \mathbb{Z}^{n_2}_{\ge 0}$ with 
$|\AA|,|\widetilde\CC|\le \m$, $|\BB|,|\widetilde\DD|\le \l$,   
\begin{equation}
\label{eq:merge}
\sum_{\substack{\CC\in \mathbb{Z}^{n_1}_{\ge 0}\,:\,\theta_*(\CC)=\widetilde\CC
	\\
	 \DD\in \mathbb{Z}^{n_1}_{\ge 0}\,:\,\theta_*(\DD)=\widetilde\DD}}
 \vertt{\l}{\m}{n_1}{x;q}{\AA}{\BB}{\CC}{\DD} =  \vertt{\l}{\m}{n_2}{x;q}{\theta_*(\AA)}{\theta_*(\BB)}{\widetilde\CC}{\widetilde\DD}.
\end{equation} 
\end{prop}
\begin{proof} For $\l=\m=1$, when the $W$-weights turn into matrix elements of 
the $R$-matrix, cf. \eqref{eq:W-R}, the statement coincides with \cite[Proposition 4.3]{BGW} 
(and it is also easy to check directly from the formula \eqref{fund-vert} for 
the weights). For general $\l,\m\ge 1$, \eqref{eq:merge} readily follows 
from the $\l=\m=1$ case and the stochastic fusion \eqref{eq:fusion}. 		
\end{proof} 

\begin{cor}\label{cor:merge} Colour merging statements similar to Proposition  
\ref{prop:merge} hold for the vertex weights $L_x$, $M_x$, $L_x^\stoch$, 
$M_x^\stoch$, $q$-Hahn weights \eqref{eq:q-Hahn}, as well as for the vertex 
weights defined by \eqref{eq:limit-vertex}. 
\end{cor}
\begin{proof} Follows from the fact that all these weights are obtained from $W_{\l,\m}$ by specializations, analytic continuation, multiplication by factors that give the same contribution to the two sides of the merging relation \eqref{eq:merge}, and a limit transition $\epsilon\to 0$ in \eqref{eq:W-modified} in the case of \eqref{eq:limit-vertex}. 
\end{proof}

The colour-blindness statements of Section \ref{ssec:colour-blind} correspond, 
in the notation of Propostion \ref{prop:merge}, to $n_2=1$. In particular, 
applying the colour-blindness relation to either the q-Hanh weights 
\eqref{eq:q-Hahn} or to the weights \eqref{eq:limit-vertex} and removing common 
factors on the two sides, one obtains the following $q$-identity. 

\begin{cor}\label{cor:q-identity} For any $Q\in\mathbb{C}$, $(\alpha_1,\dots,\alpha_n)\in \mathbb{Z}_{\ge 0}^n$, and a fixed $|\beta|\in \mathbb{Z}_{\ge 0}$, $|\beta|\le |\alpha|$, one has
\begin{equation}
\label{eq:q-identity}
\sum_{\substack{0\le \beta_i\le \alpha_i,\, 1\le i\le n,\\
	\beta_1+\dots+\beta_n=|\beta|}} Q^{\sum_{1\le i<j\le n}\beta_i(\alpha_j-\beta_j)}\prod_{i=1}^n {\binom{\alpha_i}{\alpha_i-\beta_i}}_Q=
{\binom{|\alpha|}{|\alpha|-|\beta|}}_Q. 
\end{equation}	
\end{cor}  

Since compositions are vectors, the map $\theta_*$ from \eqref{eq:theta-star} 
sends any composition of length $n_1$ to a composition of length $n_2$. This can 
be naturally extended to coloured compositions as follows. 

Let $\lambda$ be a composition with $\ell(\lambda)=n_1$, weight $|\lambda|=m$, 
and partial sums $\ell_k=\sum_{i=1}^{k} \lambda_i$. Then 
$\rho:=\theta_*(\lambda)$ is a composition with $\ell(\rho)= n_2$, same weight 
$|\rho|=m$, and partial sums that we denote as $r_k=\sum_{i=1}^{k} \rho_i$.
Further, let $\mu$ be a $\lambda$-coloured composition of length $m$, see 
Section \ref{ssec:coloured-comp} for a definition. As in 
\eqref{eq:A(k)-coloured}, we can encode $\mu$ by a sequence of $n_1$-dimensional 
vectors $\{\AA(k)\}_{k\ge 0}$:
\begin{align*}
\bm{A}(k) = \sum_{j=1}^{n_1} A_j(k) \bm{e}_j,
\qquad
A_j(k)
=
\#\{ i : \mu_i = k,\ \ell_{j-1} +1 \leq i \leq \ell_j \}.
\end{align*}
It is not difficult to see that the sequence of $n_2$-dimensional vectors 
$\BB(k):=\theta_*(\AA(k))$, $k\ge 0$,  corresponds, via
\begin{align*}
\theta_*(\AA(k))=\bm{B}(k) = \sum_{j=1}^{n_2} B_j(k) \bm{e}_j,
\qquad
B_j(k)
=
\#\{ i : \nu_i = k,\ r_{j-1} +1 \leq i \leq r_j \},
\end{align*}
to a $\rho$-coloured composition $\nu$ of length $m$ that we will define to be the image of $\mu$ under $\theta_*$: 
\begin{equation}
\label{eq:theta-comp}
\theta^*(\mu)=\nu.
\end{equation}
In less formal terms, if we view a coloured composition $\mu$ as positions of 
finitely many paths coloured by $\{1,\dots,n_1\}$, then $\nu=\theta_*(\mu)$ 
represents positions of the same paths that have been re-coloured according to 
the map $\theta$. Note that we required $\theta$ to be monotone, which 
means that the order of colours is being preserved. 

\begin{prop}\label{prop:fG-merge} Let $\lambda$ be a composition with 
$\ell(\lambda) = n_1$ of weight $|\lambda|=m$, $\rho$ be a composition with 
$\ell(\rho)= n_2$ and same weight $|\rho|=m$ such that $\theta_*(\lambda)=\rho$, 
$\nu'$ by a $\lambda$-coloured composition, and $\nu''$ be a $\rho$-coloured 
composition (both of length $m$). Then for any $p\ge 1$ and complex parameters 
$x_1,x_2,\dots$, we have
\begin{gather}
\label{eq:f-colour-merge}
\sum_{\mu: \theta_*(\mu) = \nu''}
f_{\mu}(\lambda; x_1,\dots,x_m)
= f_{\nu''}(\rho; x_1,\dots,x_m),\\
\label{eq:G-colour-merge}
\sum_{\mu: \theta_*(\mu) = \nu''}
G_{\nu'/\mu}(\lambda; x_1,\dots,x_p)
=
G_{\theta_*(\nu')/\nu''}(\rho;x_1,\dots,x_p).
\end{gather}
\end{prop}
\begin{rmk}
When $\lambda=(1,1,\dots,1)$ and $n_2=1$, one recovers the symmetrization formulas \eqref{eq:symm}.
\end{rmk}
\begin{proof} In complete analogy with proofs of \eqref{eq:symm} in 
\cite[Propositions 3.4.4 and 4.4.3]{BW}, the argument consists in multiple 
applications of Proposition \ref{prop:merge}, or rather its version for the 
weights $L_x$ and $M_x$ (Corollary \ref{cor:merge}) used in the definitions of 
$f$'s and $G$'s. In the case of \eqref{eq:f-colour-merge}, one starts with the 
top-right nontrivial vertex of the partition function \eqref{eq:f-function} 
(with appropriate, not necessarily rainbow colours of entering paths on the 
left), while in the case of \eqref{eq:G-colour-merge} one starts with the 
top-left nontrivial vertex of the partition function \eqref{eq:G-functions}, and 
then moves step by step into the bulk of the partition function. Once the 
colour-merging summation has been performed for all nontrivial vertices, one 
recovers the right-hand sides of \eqref{eq:f-colour-merge} and 
\eqref{eq:G-colour-merge}. 
\end{proof}

\section{Cauchy identities}\label{sec:cauchy}

The goal of this section is to show how the skew-Cauchy identity 
\eqref{eq:fG-skew} leads to formulas for averages of certain observables for 
stochastic vertex models.

For this section let us assume that we are either in the column-homogeneous 
situation, or, slightly more generally, the number of columns in which 
inhomogeneity parameters $(s_j,\xi_j)$ of Section \ref{ssec:inhom} are 
different from $(s,1)$ is finite. 
We start by extending the skew-Cauchy identity \eqref{eq:fG-skew} to limiting versions of the $G$-functions from Section \ref{ssec:principal}. 

\begin{prop}\label{prop:fG-principal}
	 Let $\mu$ be a rainbow composition of length $n\ge 1$, and $x_1,\dots,x_n$ be complex parameters satisfying
\begin{equation}
\label{eq:weight-condition3}
\left|\frac{s(x_i-s)}{1-sx_i}\right|<1,\qquad 1\le i\le n. 
\end{equation}
The we have 
\begin{equation}
\label{eq:fG-principal}
\sum_{\nu}f_{\nu}(x_1,\dots,x_n)\,\cG_{\nu/\mu}=\prod_{i=1}^n \left(1-\frac{x_i}s\right)\cdot f_{\mu}(x_1,\dots,x_n),
\end{equation}	
where $\cG_{\nu/\mu}$ is as in Corollary \ref{cor:cG-rainbow}, and $f_*$'s are 
as \eqref{eq:f-function}.
\end{prop}
\begin{proof} We start with the summation identity \eqref{eq:fG-skew}, multiply 
both sides by $q^{np}$, and substitute $p=\l$,
$(y_1,\dots,y_p)=(\epsilon,q\epsilon,\dots,q^{\l-1}\epsilon)$ to match with 
\eqref{eq:cG-limit}. Then, as in the proof of Proposition \ref{prop:cG}, 
$G_{\nu/\mu}(\epsilon,q\epsilon,\dots,q^{\l-1}\epsilon)$ can be represented as a 
one-row partition function of the form \eqref{eq:one-row} with fused weights 
\eqref{eq:W-modified}, and this partition function is rational in 
$\mathfrak{l}:=q^\l$. The explicit formula for the weights \eqref{eq:W-weights}  
readily shows that $G_{\nu/\mu}(\epsilon,q\epsilon,\dots,q^{\l-1}\epsilon)$ is 
bounded in absolute value by $\text{const}^{|\nu|}$, where the constant can be 
chosen uniformly when $\mathfrak{l}$ varies in a neighborhood of 0. We also 
have 
\begin{equation}
\label{eq:f-estimate}
|f_{\nu}(x_1,\dots,x_n)|\le \text{const} \prod_{i=1}^n \left| \frac{x_i-s}{1-sx_i} \right|^{|\nu|},
\end{equation}
which follows from the fact that in the partition function \eqref{eq:f-function} only vertices of the form 
$\tikz{0.3}{
	\draw[lgray,line width=0.7pt,->] (-1,1) -- (1,1);
	\draw[lgray,line width=2.5pt,->] (0,0) -- (0,2.2);
	\node[left] at (-1,1) {\tiny $i$};\node[right] at (1,1) {\tiny $i$};
	\node[below] at (0,0) {\tiny $\bm{0}$};\node[above] at (0,2) {\tiny $\bm{0}$};
}
$ have the number of appearances that is not \emph{a priori} bounded, and the 
weight of such a vertex in the $j$th row, according to \eqref{eq:L-weights}, is 
$(x_j-s)/(1-sx_j)$ (at least sufficiently far to the right, even if there are 
finitely many column inhomogeneities). Hence, for $x_i$'s in sufficiently small 
neighborhood of $s$, the series in the left-hand side converges uniformly in 
$\mathfrak{l}$ varying in a neighborhood of 0, yielding an analytic function 
of $\mathfrak{l}$. 

On the other hand, the right-hand side of \eqref{eq:fG-skew} multiplied by 
$q^{np}$, upon the substitution of  $p=\l$ and 
$(y_1,\dots,y_p)=(\epsilon,q\epsilon,\dots,q^{\l-1}\epsilon)$, reads
\begin{align*}
\prod_{i=1}^{n}
\frac{1-\mathfrak{l}\, \epsilon x_i }{1- \epsilon x_i }
\cdot
f_{\mu}(x_1,\dots,x_n),
\end{align*}
which is clearly analytic in $\mathfrak{l}$. Since \eqref{eq:fG-skew} gives 
the equality of the two sides for $\mathfrak{l}=q^\l$, $\l=1,2,\dots$, we 
conclude (assuming $|q|<1$ to get a sequence of points accumulating to 0) that 
the two sides are equal for any 
$\mathfrak{l}$ in a neighborhood of 0. 

Let us now substitute $\mathfrak{l}=(s\epsilon)^{-1}$, as \eqref{eq:cG-limit} stipulates. Then the weights for the one-row partition function representation of $G_{\nu/\mu}$ simplify to \eqref{eq:epsilon-vertex}, and using that one checks 
that 
$$
\left|G_{\nu/\mu}(\epsilon,q\epsilon,\dots,q^{\l-1}\epsilon)\bigr|_{q^\l=(s\epsilon)^{-1}}\right| \le \text{const}^{|\nu|},
$$	
where the constant can be chosen uniformly when $\epsilon$ varies in a complex neighborhood of the positive ray $[0,a]$, where $a>0$ is arbitrary. Using \eqref{eq:f-estimate} again and restricting $x_i$'s to a sufficiently small neighborhood of $s$, we can analytically continue both sides of the identity from large $\epsilon$ (or small $\mathfrak{l}$), where it has already been proven, to small $\epsilon$ including 0. The identity at $\epsilon=0$ is exactly \eqref{eq:fG-principal}. Once we have it proven for $x_i$'s sufficiently close to $s$ and $|q|<1$, we can relax these assumptions by further analytic continuation in these parameters (although we will not need $|q|\ge 1$ below). In particular, since the explicit formula \eqref{eq:cG-formula} for $\cG_{\nu/\mu}$ shows $\text{const}\cdot |s|^{|\nu|}$ behavior of $\cG_{\nu/\mu}$ for large $\nu$, using \eqref{eq:f-estimate} we can extend the equality to $x_i$'s satisfying \eqref{eq:weight-condition3}.  
\end{proof}

Taking $\mu=(0,0,\dots,0)$ in Proposition \ref{prop:fG-principal}, we obtain the following
\begin{cor}
\label{cor:fG-nu} 
For $x_1,\dots, x_n\in \mathbb{C}$ satisfying \eqref{eq:weight-condition3}, we have
\begin{equation}
\label{eq:fG-nu}
\prod_{i=1}^n \frac{1-sx_i}{s(s-x_i)}\cdot\sum_{\nu:\,\text{all }\nu_i>0} (-s)^{|\nu|}  f_{\nu}(x_1,\dots,x_n)=1,
\end{equation}
where the summation is over all compositions $\nu$ of length $n$ with no zero parts. In the inhomogeneous setting of Section \ref{ssec:inhom}, assuming that $(s_0,\xi_0)=(s,1)$, a similar formula holds, where one needs to replace $(-s)^{|\nu|}$ in the left-hand side by $\prod_{i=1}^n\prod_{j=0}^{\nu_i-1} (-s_j)$. 
\end{cor}
\begin{proof} Straightforward substitution of $\mu=(0,\dots,0)$ and \eqref{eq:cG-nu} into \eqref{eq:fG-principal}, with the evaluation 
\begin{equation}
\label{eq:f-empty}
f_{(0,\dots,0)}(x_1,\dots,x_n)=\prod_{i=1}^n \frac{1-s^2q^{i-1}}{1-sx_i}\,,
\end{equation}
which follows from the fact that the corresponding partition function \eqref{eq:f-function} is the product of $L$-weights \eqref{eq:L-weights} for the unique path configuration that gives a nonzero contribution. 
\end{proof}

It is not difficult to extend Proposition \ref{prop:fG-principal} and Corollary \ref{cor:fG-nu} to partially merged colours. 

\begin{prop}
\label{prop:fG-merged}
Let $\lambda$ be a composition with $|\lambda|=n$, $\mu$ be a $\lambda$-coloured composition, and $x_1,\dots,x_n\in \mathbb{C}$ satisfy \eqref{eq:weight-condition3}. Then  
\begin{equation}
\label{eq:fG-merged}
\sum_{\nu\text{ is $\lambda$-coloured}}f_{\nu}(\lambda;x_1,\dots,x_n)\,\cG_{\nu/\mu}=\prod_{i=1}^n \left(1-\frac{x_i}s\right)\cdot f_{\mu}(\lambda;x_1,\dots,x_n),
\end{equation}
which in the case of $\mu$ having only zero parts, can be rewritten as 
 \begin{equation}
 \label{eq:fG-nu-merged}
 \prod_{i=1}^n \frac{1-sx_i}{s(s-x_i)}\cdot\sum_{\substack{\nu\text{ is $\lambda$-coloured}\\ \text{all }\nu_i>0}} (-s)^{|\nu|}  f_{\nu}(\lambda; x_1,\dots,x_n)=1.
 \end{equation}
 In the inhomogeneous setting of Section \ref{ssec:inhom}, relation \eqref{eq:fG-nu-merged} also holds under assumption that $(s_0,\xi_0)=(s,1)$ and with $(-s)^{|\nu|}$ in the left-hand side replaced by $\prod_{i=1}^n\prod_{j=0}^{\nu_i-1} (-s_j)$.
\end{prop}
\begin{proof}
Let $\theta:\{1,\dots,n\}\to\{1,\dots,\ell(\lambda)\}$ be the unique colour merging  monotone map such that $\theta_*((1,\dots,1))=\lambda$, cf. Section \ref{ssec:colour-merge}. Then we can sum \eqref{eq:fG-principal} over rainbow $\mu$ with a given image $\theta_*(\mu)$. The right-hand side is immediately computed via \eqref{eq:f-colour-merge}, and in the left-hand side we use \eqref{eq:G-colour-merge} and subsequently perform, using \eqref{eq:f-colour-merge}, a partial summation over $\nu$'s with the same image $\theta_*(\nu)$. The result is \eqref{eq:fG-merged}, with $\theta_*(\mu)$ and $\theta_*(\nu)$ replaced back by $\mu$ and $\nu$. The second relation \eqref{eq:fG-nu-merged} follows from \eqref{eq:fG-merged} in the same way as in the rainbow case of Corollary \ref{cor:fG-nu}. 	
\end{proof}

\begin{defn}\label{def:measure} For any $n\ge 1$, composition $\lambda$ with $|\lambda|=n$, and $q,s,x_1,\dots,x_n\in\mathbb{C}$ satisfying \eqref{eq:weight-condition3}, we define a (generally speaking, complex valued)
probability measure on the subset of $\mathcal{S}_\lambda$ (defined in \eqref{eq:lambda-col}) consisting of $\lambda$-coloured compositions $\nu$  with no zero parts by
\begin{equation}
\label{eq:measure}
\P\left(\{\nu\}\right)= \prod_{i=1}^n \frac{1-sx_i}{s(s-x_i)}\cdot (-s)^{|\nu|}  f_{\nu}(\lambda; x_1,\dots,x_n).
\end{equation}
Graphically, the weight of this measure could be seen as partition functions of the form \eqref{eq:f-function} with incoming colours on the left partially identified according to $\lambda$, with vertex weights $L^\stoch$ given by \eqref{eq:L-weights}, and conditioned to have no exiting paths in the 0th column 
(\ie \  $\AA(0)=\bm{0}$ in terms of \eqref{eq:f-function}).

In what follows we will also use the notation 
\begin{equation}
\label{eq:f-stoch}
 f_{\nu}^\stoch(\lambda; x_1,\dots,x_n)=(-s)^{|\nu|}  f_{\nu}(\lambda; x_1,\dots,x_n).
\end{equation}
In the inhomogeneous setting of Section \ref{ssec:inhom}, we will assume that $(s_0,\xi_0)=(s,1)$, and in the right-hand sides of \eqref{eq:measure} and \eqref{eq:f-stoch} replace $(-s)^{|\nu|}$ by $\prod_{i=1}^n\prod_{j=0}^{\nu_i-1} (-s_j)$.
\end{defn}

The graphical interpretation is based on the observation that the prefactor of the sums in \eqref{eq:fG-nu}, \eqref{eq:fG-nu-merged} is exactly the inverse of the product of $L^\stoch$-weights of vertices $\tikz{0.3}{
	\draw[lgray,line width=0.7pt,->] (-1,1) -- (1,1);
	\draw[lgray,line width=2.5pt,->] (0,0) -- (0,2.2);
	\node[left] at (-1,1) {\tiny $i$};\node[right] at (1,1) {\tiny $i$};
	\node[below] at (0,0) {\tiny $\bm{0}$};\node[above] at (0,2) {\tiny $\bm{0}$};
}
$ in the 0th column of a partition function of the form \eqref{eq:f-function} with no turns in the 0th column, and \eqref{eq:f-stoch} is the result of computing the partition function for $f_\nu$ with $L$-weights replaced by the $L^\stoch$-weights. Together with the stochasticity of the $L^\stoch$-weights, this also implies $\eqref{eq:fG-nu}$ and $\eqref{eq:fG-nu-merged}$. 

\begin{defn}
\label{def:observables}
In the context of Definition \ref{def:measure}, for any $\lambda$-coloured composition $\mu=0^{\bm{m}^{(0)}}1^{\bm{m}^{(1)}}2^{\bm{m}^{(2)}}\cdots$ introduce an observable $\O_\mu$, whose values on $\lambda$-coloured compositions $\nu$ with no zero parts are given by
\begin{equation}
\label{eq:obs}
\O_\mu(\nu)=\prod_{x\ge 1}\displaystyle\prod_{i\ge 1} 
q^{m_i^{(x)}H_{>i}^{\nu/\mu}(x+1)}{\displaystyle\binom{H_i^{\nu/\mu}(x+1)}{m_i^{
			(x) } } } _q,
\end{equation}
where we use the coloured height functions \eqref{eq:coloured-height}. Note that only nonzero parts of $\mu$ play a role in this definition. 

For rainbow compositions, \ie\ when $\lambda=(1,\dots,1)$, the observables take a simpler form, cf. Corollary \ref{cor:cG-rainbow}:
\begin{equation}
\label{eq:obs-rainbow}
\O^\text{rainbow}_\mu(\nu)=\displaystyle\prod_{\substack{x\ge 1,\, i\ge 1\\ m_i^{(x)}=1}} 
\bm{1}_{H_i^{\nu/\mu}(x+1)=1} \,q^{H_{>i}^{\nu/\mu}(x+1)}=\displaystyle\prod_{\substack{x\ge 1,\, i\ge 1\\ m_i^{(x)}=1}}  \frac{q^{H_{> i+1}^{\nu/\mu}(x+1)}-q^{H_{> i}^{\nu/\mu}(x+1)}}{q-1}.
\end{equation}
\end{defn}

\begin{ex}
\label{ex:obs-colour-blind}	
In the colour-blind case $\lambda=(n)$, thinking of the coordinates of $\mu$ as 
being ordered: $\mu_1\ge\mu_2\ge\dots\ge \mu_n$, we can rewrite \eqref{eq:obs} 
as a shifted $q$-moment of the colourless height function: 
\begin{equation}
\label{eq:obs-colour-blind}
\O_\mu^{\text{colour-blind}}(\nu)=\frac{(1-q^{H^\nu(\mu_1+1)})(1-q^{
H^\nu(\mu_2+1)-1})\cdots (1-q^{H^\nu(\mu_k+1)-k+1})}{\prod_{j\ge 1} 
(q;q)_{\text{mult}_j(\mu)}}\,,
\end{equation}
where $\text{mult}_j(\mu)=\#\{i:\mu_i=j\}$ and $k=\max\{j:\mu_j>0\}$. 
\end{ex}

We are now in position to formulate the main result of this section.

\begin{thm}\label{thm:obs-avg} 
With notations of Definitions \ref{def:measure} and \ref{def:observables} above, we have 
\begin{equation}
\label{eq:obs-avg}
\E\, \O_\mu=\E_{\nu}\left[ \O_\mu(\nu)\right]= \frac{q^{\sum_{x\ge 1}\sum_{i>j} m_i^{(x)}m_j ^{(x)}}}{\prod_{j\ge 0} (s^2;q)_{|\bm{m}^{(x)}|}}
\prod_{i=1}^n(1-sx_i)\cdot f_{\mu}^\stoch (\lambda;x_1,\dots,x_n), 
\end{equation}
where the expectation is taken with respect to the weights \eqref{eq:measure}, and the observables are given by \eqref{eq:obs} in the general case, or by \eqref{eq:obs-rainbow} in the rainbow case. 

The formula \eqref{eq:obs-avg} also holds in the inhomogeneous setting of Section \ref{ssec:inhom} under the assumption that $(s_x,\xi_x)=(s,1)$ for any $x$ such that $|\bm{m}^{(x)}|>0$, and for $x=0$. 
\end{thm}
\begin{proof}
Let us take the ratio of two skew-Cauchy identities \eqref{eq:fG-merged} with $\mu$ and with $\mu=(0,\dots,0)$. This yields
\begin{equation*}
\frac{\sum_{\nu}f_{\nu}(\lambda;x_1,\dots,x_n)\,\cG_{\nu/\mu}}{\sum_{\nu}f_{\nu}(\lambda;x_1,\dots,x_n)\,\cG_{\nu}}=\frac{f_{\mu}(\lambda;x_1,\dots,x_n)}{f_{(0,\dots,0)}(\lambda;x_1,\dots,x_n)}\,,
\end{equation*}
which can be rewritten, via \eqref{eq:f-stoch} and \eqref{eq:f-empty} (stated in the rainbow case, but also holding in the non-rainbow one for the same reasons), as
\begin{equation}
\label{eq:modified-Cauchy}
\frac{\sum_{\nu}f_{\nu}(\lambda;x_1,\dots,x_n)\,\cG_{\nu}\cdot \left({\cG_{\nu/\mu}}/{\cG_{\nu}}\right)}{\sum_{\nu}f_{\nu}(\lambda;x_1,\dots,x_n)\,\cG_{\nu}}=\frac{\prod_{i=1}^n(1-sx_i)}{(-s)^{|\mu|}(s^2;q)_n}\, f_{\mu}^\stoch(\lambda;x_1,\dots,x_n).
\end{equation}
The summations are taken over $\lambda$-coloured compositions $\nu$ with no zero parts. 

Observe that the left-hand side of \eqref{eq:modified-Cauchy} is exactly $\E_\nu \left({\cG_{\nu/\mu}}/{\cG_{\nu}}\right)$, cf. Definition \ref{def:measure}. 
The expression \eqref{eq:cG-formula} for $\cG_{\nu/\mu}$ (which is also the source of our assumption in the inhomogeneous setting) implies 	
\begin{multline*}
\frac{\cG_{\nu/\mu}}{\cG_\nu}=\frac{(s^2;q)_{|\bm{m}^{(0)}|}}{(-s)^{|{\mu}|}(s^2;q)_n}\,
q^{\sum_{i=1}^n\left(m_i^{(0)}\sum_{n\ge j>i}(H_j^{\nu/\mu}(1)-m_j^{(0)})\right)}
\prod_{i\ge 1}{\binom{H_i^{\nu/\mu}(1)}{m_i^{
		(0)} } } _q\\
	\times 
 \prod_{x\ge 1}
(s^2;q)_{|\bm{m}^{(x)}|}q^{-\sum_{i>j}m_i^{(x)}m_j^{(x)}} 
\cdot \O_\mu(\nu).
\end{multline*}
The only $\nu$-dependent part of this expression is $\O_\mu(\nu)$, and it stays under the expectation; the other factors can be moved to the right-hand side of \eqref{eq:modified-Cauchy}. It only remains to notice that for $\cG_{\nu/\mu}$ not to vanish, we must have ${H_i^{\nu/\mu}(1)}=m_i^{(0)}$ for all colours $i\ge 1$. 
To see that, it might be easiest to return to the expression \eqref{eq:limit-vertex} for the vertex weights in the one-row representation \eqref{eq:one-row} of $\cG_{\nu/\mu}$, and note that in the 0th column we have 
$\DD=\bm{0}$, which together with $\AA+\BB=\CC+\DD$ and $\BB\ge \CC$ (enforced by the $q$-binomial coefficients) implies $\BB=\CC$, which is exactly the statement we are making, as $B_i=H_i^{\nu/\mu}(1)$ and $C_i=m_i^{(0)}$. Taking these equalities into account and canceling out common factors gives \eqref{eq:obs-avg}. 
\end{proof}

\section{Integral representations for $f_\mu$}
\label{sec:integral}

The result of averaging the observable $\O_\mu$ in Theorem \ref{thm:obs-avg} is the function $f_\mu$, up to simple prefactors. The coloured composition $\mu$ here has the same dimension (\ie\ number of parts) as the coloured compositions $\nu$ over which we are averaging (cf. Definition \ref{def:measure}), and thus does not offer much reduction in complexity. The goal of this section is to show that $f_\mu$ has an integral representation of the dimension equal to the number of \emph{nonzero} parts of $\mu$. Hence, by choosing $\mu$ consisting of mostly zeros, we will be able to obtain tangible formulas for averages over $\nu$'s of a growing dimension. 

\begin{defn}\label{def:mu>1} For a coloured composition $\mu$ (defined as in Section \ref{ssec:coloured-comp}), we will denote by $\mu^{\ge 1}$ the coloured composition obtained from $\mu$ by removing all of its zero parts. The colouring of $\mu^{\ge 1}$ will be the one naturally inherited from $\mu$, and we will denote it by $\lambda^{\ge 1}$ if $\mu$ is $\lambda$-coloured, or by $\text{colour}(\mu^{\ge 1})$ if $\lambda$ is not explicit.  
\end{defn}

Our first goal is to prove the following integral representation of $f_\mu$ for rainbow compositions.

\begin{prop}\label{prop:integral-rainbow}
Let $\mu=0^{\bm{m}^{(0)}}1^{\bm{m}^{(1)}}2^{\bm{m}^{(2)}}\cdots$ be a rainbow 
composition of length $n\ge 1$ not consisting entirely of zeros, let $\mu^{\ge 1}$ 
be as in Definition \ref{def:mu>1} with $m:=n-|\bm{m}^{(0)}|\ge 1$ being the 
length of $\mu^{\ge 1}$, and let $c_1<\dots<c_m$ be the colours of the 
(necessarily nonzero) parts of $\mu^{\ge 1}$. Then 
\begin{multline}
\label{eq:f-integral-rainbow}
f_\mu(x_1,\dots,x_n)=\frac{(s^2;q)_{|\bm{m}^{(0)}|}}{(1-sx_1)\cdots(1-sx_n)}
\frac{(-1)^m}{(2\pi\sqrt{-1})^m}\oint\cdots\oint_{\text{around }\{x_j^{-1}\}}
\prod_{1\le i<j\le m} \frac{y_j-y_i}{y_j-qy_i} \\
\times f_{\mu^{\ge 1}}(y_1^{-1},\dots,y_m^{-1})\prod_{j=1}^m 
\left(\frac{y_j-s}{y_j} \frac{x_{c_j}}{1-x_{c_j}y_j}\prod_{i>c_j}^n 
\frac{1-qx_iy_j}{1-x_iy_j}\, dy_j\right),
\end{multline}
where (positively oriented) integration contours are chosen to encircle all 
points $\{x_j^{-1}\}_{j=1}^n$ and no other singularities of the 
integrand, or as \emph{$q$-nested} closed simple curves with $y_{i}$-contour containing $q^{-1}\cdot (y_j\text{-contour})$ for all $i<j$, and all of the contours encircling $\{x_j^{-1}\}_{j=1}^n$.  

The formula also holds in the column inhomogeneous setting under the assumption that in the 0th column $(s_0,\xi_0)=(s,1)$. 
\end{prop}
\begin{rmk}\label{rmk:sum-of-residues} Contour integration around a specific set of singularities can be viewed formally as the sum of residues at those singularities, and such a sum would make sense if the parameters are such that required contours are not possible to construct. This gives a slightly different way of interpreting \eqref{eq:f-integral-rainbow}, as well as all the other integral representations below.   
\end{rmk}
\begin{rmk}\label{rem:recolouring} According to Definition \ref{def:mu>1}, 
$\mu^{\ge 1}$ is a coloured composition, and thus $f_{\mu^{\ge 
1}}(y_1^{-1},\dots,y_m^{-1})$ should really be written as $f_{\mu^{\ge 
1}}(\lambda;y_1^{-1},\dots,y_m^{-1})$, where $\lambda$ is the colouring 
composition for $\mu^{\ge 1}$. However, since all parts of $\mu^{\ge 1}$ have 
different colours, and our vertex weights always depend on the colours only 
through their ordering, we could replace the colours
$c_1<\dots<c_m$ represented in $\mu^{\ge 1}$ by $1,2,\dots,m$, and denoting the 
resulting rainbow composition by $\tilde\mu^{\ge 1}$, we would have 
$f_{\mu^{\ge 1}}(\lambda,y_1^{-1},\dots,y_m^{-1})=f_{\tilde\mu^{\ge 
1}}(y_1^{-1},\dots,y_m^{-1})$. It is this function that we denoted as 
$f_{\mu^{\ge 1}}(y_1^{-1},\dots,y_m^{-1})$, thus slightly abusing the 
notation. 
\end{rmk}

\begin{rmk}\label{rem:s_0=0} In the case $s_0=0$, there exists a fairly straightforward argument
leading to a formula similar to \eqref{eq:f-integral-rainbow}. Namely, in that case vertices of the form \tikz{0.27}{\draw[lgray,line width=0.7pt,->] (-1,0.2) -- (1,0.2);	\draw[lgray,line width=2.0pt,->] (0,-0.65) -- (0,1.35);	\node[left] at (-1,0.2) {\tiny $j$};\node[right] at (1,0.2) {\tiny $i$};} with $j>i$ have weight zero (in the 0th column) due to vanishing of the bottom right entry in \eqref{eq:L-weights}. This means that the paths of colours $c_1,\dots,c_m$ are not allowed make any vertical steps in column 0, as otherwise they would have no way of exiting this column to the right (because of the ordering of entering colours along the left boundary). This completely determines the configuration of paths in column 0, and the contribution of the remaining columns can be encoded as the matrix element 
$$
\bra{\varnothing}
\C_0(x_1)\cdots \C_0(x_{c_1-1})\C_{c_1}(x_{c_1})\C_0(x_{c_1+1})\cdots
\C_0(x_{c_m-1})\C_{c_m}(x_{c_m})\C_0(x_{c_m+1})\cdots \C_0(x_n)
\ket{\varkappa},
$$
with $\varkappa=\mu^{\ge 1}-1^m$ and $\C$-row operators from Section \ref{ssec:row-operators}.
The summation over $\varkappa$ of such expressions multiplied by $g^*_{\varkappa}(y_1,\dots,y_m)$ is evaluated as an explicit product similar to \eqref{eq:mimachi} via the commutation relations \eqref{eq:CB}, and then the coefficients of $g^*_\varkappa$ are extracted by the orthogonality \eqref{eq:f-g-orthog}. 

It is not clear, however, how to extend this argument to $s_0\ne 0$, and we need to employ a different idea in the proof below. 
	
\end{rmk}

\begin{proof}[Proof of Proposition \ref{prop:integral-rainbow}] Our argument consists of two steps: We will first prove the 
formula for $(c_1,\dots,c_m)=(n-m+1,\dots,n)$, and then show that the formula 
continues to hold when we reduce one of the $c_j$'s by 1. Iterations of such reductions will cover 
all possible choices of $1\le c_1<\dots<c_m\le n$. 

Both steps work identically in the column homogeneous and inhomogeneous settings, and we will give the argument in the homogeneous case. The inhomogeneous analogs of several statements from Section \ref{sec:prelim} used below can be found in Section \ref{ssec:inhom}. 

Assume that $(c_1,\dots,c_m)=(n-m+1,\dots,n)$. This means that paths that enter 
the partition function of the form \eqref{eq:f-function} from the left in rows 
\textsc{}$1,\dots,n-m$ must immediately turn up and exit on top, while paths that enter  
in rows $n-m+1,\dots,n$ must move to the right across the 0th column (recall 
that no horizontal edge can carry more than one path). Hence, the configuration
of paths in the 0th column is completely determined, and the product of the 
$L$-weights \eqref{eq:L-weights} in this column gives 
$$
\frac{(s^2;q)_{|\bm{m}^{(0)}|}}{\prod_{i=1}^{n-m} (1-sx_i)}\,\prod_{j=n-m+1}^n
\frac{x_j-s}{1-sx_j}\,.
$$
On the other hand, the contribution of the remaining columns can be written as
$$
\prod_{j=n-m+1}^n \frac{1-sx_i}{x_i-s} \cdot f_{\mu^{\ge 
1}}(x_{n-m+1},\dots,x_n),
$$
where the prefactor is responsible for the fact that the partition function for 
$f_{\mu^{\ge 1}}(x_{n-m+1},\dots,x_n)$ starts with column 0, and we only had 
weights of vertices in columns $\ge 1$ remaining. We now need to show that the 
product of the last two expressions agrees with the right-hand side of 
\eqref{eq:f-integral-rainbow}. With our choice of $c_i$'s, the integrand is 
independent of $x_1,\dots,x_{n-m}$, and its only poles are at $y_i=x_j^{-1}$
with $n-m+1\le j\le n$. The number of integration variables thus coincides with 
the number of potential pole locations, and no two variables can have 
nonvanishing residues at the same location because of $\prod_{i<j}(y_i-y_j)$ in 
the integrand. Hence, the residue locations cover all points
$x_{n-m+1}^{-1},\dots,x_{n}^{-1}$ exactly once, 
and $\prod_{j=1}^m (y_j-s)x_{c_j}$ in the integrand necessarily evaluates to 
$\prod_{j=n-m+1}^n (1-sx_j)$ in every set of nontrivial residues taken. Moving
this factor out of the integral, we are lead to the following desired equality:
\begin{multline}
\label{eq:desired-integral}
 f_{\mu^{\ge 
1}}(x_{n-m+1},\dots,x_n)=\frac{(-1)^m}{(2\pi\sqrt{-1})^m}\oint\cdots\oint_{\text
{around }\{x_j^{-1}\}}
\prod_{1\le i<j\le m} \frac{y_j-y_i}{y_j-qy_i} \\
\times f_{\mu^{\ge 1}}(y_1^{-1},\dots,y_m^{-1})\prod_{j=n-m+1}^n 
\left(\frac{1}{1-x_{j}y_j}\prod_{i>j}^n 
\frac{1-qx_iy_j}{1-x_iy_j}\,\frac{dy_j}{y_j} \right).
\end{multline}

Comparing to \eqref{eq:f-integral-repr}, we see that the difference is in $(-1)^m$ and the choice of contours, which both come from the same origin. 
Namely, let us take the right-hand side of \eqref{eq:f-integral-repr} (with $n$ replaced by $m$ and $(x_1,\dots,x_n)$ replaced by $(x_{n-m+1},\dots,x_n)$), and deform the outermost $y_m$-contour in the outside direction, moving it through $\infty$ and closing around $\{x_j^{-1}\}$. The recursive construction of (rainbow) $f_\mu$'s from Section \ref{ssec:recursive} readily shows that $(y_1\cdots y_m)^{-1}  f_{\mu^{\ge 1}}(y_1^{-1},\dots,y_m^{-1})$ is a ratio of two polynomials in $y$'s with the denominator consisting only of factors of the form $(y_i-s)$, and $f_{\mu^{\ge 1}}(y_1^{-1},\dots,y_m^{-1})$ viewed as such a ratio has the numerator degree 1 less than the  denominator degree in each of the variables $y_i$, $1\le i\le m$. Hence, our deformation of the $y_m$ collects no residues along the way and yields the factor $(-1)$ for changing the contour direction. 

Next, we do the same deformation with the $y_{m-1}$-contour. Following the same reasoning, there is only one possible pole along the way: $y_{m-1}=q^{-1}y_m$. If we leave this potential singularity inside the $y_{m-1}$-contour, and proceed similarly for subsequent deformations, then we end up with two $q$-nested contours surrounding $\{x_j^{-1}\}$.
It turns out that the $q$-nestedness is not necessary. Indeed, a direct inspection of the integrand shows that the only possible pole of $y_m$ inside its current contour is $x_m^{-1}$, and then the factor $(1-qx_my_{m-1})$ makes the residue at $y_{m-1}=q^{-1}y_m$ vanish. Hence, we can close the $y_{m-1}$-contour around $\{x_j^{-1}\}$, orient it positively, and acquire another $(-1)$. 

Continuing with this procedure for $y_{m-2},y_{m-3},\dots, y_1$-contours (in this order),
we turn \eqref{eq:f-integral-repr} into \eqref{eq:desired-integral}, thus completing the first step of the proof. 

Let us now see why lowering of the $c_j$'s in \eqref{eq:f-integral-rainbow} keeps the formula intact. From the point of view of the (rainbow) composition $\mu$, replacing $c_j\mapsto c_j-1$ (assuming there is no $i\ne j$ such that $c_i=c_j-1$)
is equivalent to swapping $\mu_{c_j-1}=0$ and $\mu_{c_j}>0$. This can be done with the help of the exchange relations \eqref{eq:exchange} by acting on the right-hand side of \eqref{eq:f-integral-rainbow} by
$$
T_{c_j-1} = q - \frac{x_{c_j-1}-q x_{c_j}}{x_{c_{j}-1}-x_{c_j}} (1-\mathfrak{s}_{c_{j}-1}),
$$
cf. \eqref{eq:DL-operators}. The only part of the right-hand side of \eqref{eq:f-integral-rainbow} that is not symmetric in $(x_{c_j-1},x_{c_j})$ is the factor $x_{c_j}/(1-x_{c_j}y_j)$ (note that the contours are symmetric too), and applying $T_{c_j-1}$ to it we read
$$
\frac{qx_{c_j}}{1-x_{c_j}y_j}- \frac{x_{c_j-1}-q x_{c_j}}{x_{c_{j}-1}-x_{c_j}}\left(\frac{x_{c_j}}{1-x_{c_j}y_j}-\frac{x_{c_j-1}}{1-x_{c_j-1}y_j}\right)=\frac{x_{c_j-1}}{1-x_{c_j-1}y_j}\frac{1-qx_{c_j}y_j}{1-x_{c_j}y_j}\,.
$$
This recovers the integrand of \eqref{eq:f-integral-rainbow} with $c_j\mapsto c_j-1$ and completes the proof. 
\end{proof}

\begin{rmk}\label{rmk:through-infinity} While the integral representation \eqref{eq:f-integral-repr} allowed for moving the contours through $\infty$, this is no longer true for \eqref{eq:f-integral-rainbow}, as the added factors $(y_i-s)$ create poles at $\infty$. 
\end{rmk}

The main result of this section is a generalization of Proposition \ref{prop:integral-rainbow} to coloured non-rainbow compositions. 

\begin{thm}\label{thm:f-integral} Let $\lambda$ be a composition of weight $|\lambda|=n$ with partial sums  $\sum_{i=1}^{k} \lambda_i = \ell_k$, $k\ge 1$; $\ell_0:=0$.  
Further, let $\mu=0^{\bm{m}^{(0)}}1^{\bm{m}^{(1)}}2^{\bm{m}^{(2)}}\cdots$ be a $\lambda$-coloured
composition of length $\ell(\mu)=|\lambda|=n$ not 
consisting entirely of zeros, let $\mu^{\ge 1}$ 
be as in Definition \ref{def:mu>1} with inherited colouring $\lambda^{\ge 1}$ 
and $m:=n-|\bm{m}^{(0)}|$ being the length of $\mu^{\ge 1}$, and let 
$c_1<\dots<c_\alpha$ be the colours of the (necessarily nonzero) parts of 
$\mu^{\ge 1}$. Finally, let $\fm_1,\dots,\fm_\alpha\ge 1$ be the number of parts 
of $\mu^{\ge 1}$ of colours $c_1,\dots,c_\alpha$, respectively, and denote 
$\fm[a,b]=\fm_a+\fm_{a+1}+\dots+\fm_b$. Then 
\begin{multline}
\label{eq:f-integral}
f_\mu(\lambda; x_1,\dots,x_n)=\frac{(s^2;q)_{|\bm{m}^{(0)}|}}{(1-sx_1)\cdots(1-sx_n)}
\frac{1}{(2\pi\sqrt{-1})^m}\oint\cdots\oint_{\text{around }\{x_j^{-1}\}}
\prod_{1\le i<j\le m} \frac{y_j-y_i}{y_j-qy_i} \\
\times
\prod_{k=1}^{\alpha}\left(
\sum_{j=0}^{\fm_k} \frac{(-1)^j q^{\binom{\fm_k-j}{2}}}
	{(q;q)_j(q;q)_{\fm_k-j}}
\prod_{p>\fm[1,{k-1}]}^{j+\fm[1,{k-1}]}\prod_{a>\ell_{c_{k}-1}}^n\frac{1-qx_ay_p}{1-x_ay_p}
\prod_{r>j+\fm[1,{k-1}]}^{\fm[1,k]}\prod_{b>\ell_{c_{k}}}^n\frac{1-qx_by_r}{1-x_by_r}
\right)
	\\
	\times  f_{\mu^{\ge 1}}(\lambda^{\ge 1};y_1^{-1},\dots,y_m^{-1})
\prod_{i=1}^m 
\frac{(y_i-s)dy_i}{y_i^2} ,
\end{multline}
where (positively oriented) integration contours are chosen to encircle all 
points $\{x_j^{-1}\}_{j=1}^n$ and no other singularities of the 
integrand, or as \emph{$q$-nested} closed simple curves with $y_{i}$-contour containing $q^{-1}\cdot (y_j\text{-contour})$ for all $i<j$, and all of the contours encircling $\{x_j^{-1}\}_{j=1}^n$. The contours can also be chosen to either encircle or not encircle the point $0$.   

 The formula also holds in the column inhomogeneous setting under the assumption that in the 0th column $(s_0,\xi_0)=(s,1)$. 	
\end{thm}
\begin{rmk} For $\lambda=(1,\dots,1)$ and $\mu$ being a rainbow composition, we have $\ell_k\equiv k$, $\alpha=m$, $\fm_1=\cdots\fm_\alpha=1$, $\fm[1,k]\equiv k$, and the middle line of \eqref{eq:f-integral} evaluates to 
$$
(-1)\left(\frac 1{1-q} 
-\frac1{1-q}\frac{1-qx_{c_k}y_k}{1-x_{c_k}y_k}\right)\prod_{b>c_k}\frac{
1-qx_by_k}{1-x_by_k}=\frac{x_{c_k}y_k}{1-x_{c_k}y_k}\prod_{b>c_k}\frac{1-qx_by_k
}{1-x_by_k}\,,
$$
thus reproducing \eqref{eq:f-integral-rainbow}. 	
\end{rmk}
\begin{proof} The central role in the argument is played by the following 
\begin{lem}\label{lem} For any positive integers $\fk,\fl,\fm$ satisfying 
$\fk+\fm<\fl$, 
and any symmetric function $\Phi(z_1,\dots,z_\fm)$ that is holomorphic in a 
neighborhood of the domain encircled by the integration contours, one has
\begin{multline}\label{eq:lemma}
\sum_{\fl-1\ge i_1>\dots>i_{\fm}\ge \fk}
\oint\cdots\oint_{\text{around }\{\fx_j\}}
\prod_{ i<j} \frac{z_j-z_i}{z_j-qz_i}\, \Phi(z_1,\dots,z_{\fm})
\prod_{p=1}^{\fm}\left(\frac 1{z_p-\fx_{i_p+1}}\prod_{j=1}^{i_p}
\frac{qz_p-\fx_j}{z_p-\fx_j}dz_p\right)	\\
=\oint\cdots\oint_{\text{around }\{\fx_j\}}\prod_{ i<j} 
\frac{z_j-z_i}{z_j-qz_i}\, \Phi(z_1,\dots,z_{\fm})\\ \times
\left(\sum_{j=0}^\fm
\frac{(-1)^j q^{\binom{\fm-j}{2}}}
{(q;q)_j(q;q)_{\fm-j}}
\prod_{p>0}^j \prod_{a=1}^\fl \frac{qz_p-\fx_a}{z_p-\fx_a}
\prod_{r>j}^\fm \prod_{b=1}^\fk \frac{qz_r-\fx_b}{z_r-\fx_b}\right)
\prod_{p=1}^\fm  \frac{dz_p}{z_p}\,,
\end{multline}
where (positively oriented) integration contours are chosen (independently in the two sides of \eqref{eq:lemma}) to encircle all 
points $\{\fx_j\}_{j=1}^\fl\subset\mathbb{C}$ and no other singularities of the 
integrand, or as \emph{$q$-nested} closed simple curves with $z_{i}$-contour containing $q^{-1}\cdot (z_j\text{-contour})$ for all $i<j$, and all of the contours encircling $\{\fx_j\}_{j=1}^\fl$. The contours can also be chosen to either encircle or not encircle the point $0$.   
\end{lem}

Let us postpone the proof of Lemma \ref{lem} and use it for the proof of 
Theorem \ref{thm:f-integral} first. 

Let $\theta:\{1,\dots,n\}\to \{1,\dots,\ell(\lambda)\}$ be the unique monotone
map such that 
\begin{equation}
\label{eq:theta-inverse}
\theta^{-1}(k)=\{\ell_{k-1}+1,\dots,\ell_k\}\quad \text{ for all }\quad  1\le 
k\le \ell(\lambda).
\end{equation}
We can then use the colour merging relation \eqref{eq:f-colour-merge} to obtain 
a formula for 
$\lambda$-coloured $\mu$'s from the formula \eqref{eq:f-integral-rainbow} of 
Proposition \ref{prop:integral-rainbow} for the 
rainbow ones. This means that we need to sum the right-hand sides of 
\eqref{eq:f-integral-rainbow}, written for a rainbow composition $\tilde \mu$, 
over all $\tilde\mu$ with $\theta_*(\tilde\mu)=\mu$. 

Such a summation can be performed in two steps. At the first step we choose, 
for each colour $c_k$ represented in $\mu^{\ge 1}$, $1\le k\le\alpha$, the 
$\fm_k$ colours in $\theta^{-1}(c_k)$ that are represented in $\tilde\mu^{\ge 
1}$. At the second step we choose, for each $k$, $1\le k\le \alpha$, different
assignments of the chosen $\fm_k$ colours in $\theta^{-1}(c_k)$ to the $\fm_k$  
parts of $\mu^{\ge 1}$ that have colour $c_k$. The summation of the second step 
 is exactly the colour merging applied to $f_{\tilde\mu^{\ge 1}}$ (no other 
factors of the 
integrand of \eqref{eq:f-integral-rainbow} depend on the choices in the second
step), and \eqref{eq:f-colour-merge} shows that the second step summation 
results in replacing $f_{\tilde\mu^{\ge 1}}$ in the integrand by $f_{\mu^{\ge 
1}}$. 

Returning to the first step summation, we see that $f_{\mu^{\ge 
1}}$ in the integrand is now independent of the choices involved, and we can 
focus on the rest of the integrand. For each colour $c_k$, $1\le k\le \alpha$, 
we are choosing $c_k^{(1)}<\dots<c_k^{(\fm_k)}$ in $\theta^{-1}(c_k)$, or,
according to \eqref{eq:theta-inverse}, 
$\ell_{c_k-1}<c_k^{(1)}<\dots<c_k^{(\fm_k)}\le \ell_{c_k}$. Now the sum over 
such $\fm_k$-tuple of indices can be computed, for each $1\le k\le \alpha$, 
using Lemma \ref{lem}, where one needs to make substitutions
\begin{gather*}
\fm \mapsto\fm_k,\quad \fk\mapsto (n-\ell_{c_k}),\quad \fl\mapsto 
(n-\ell_{c_k-1}),\quad  
(z_1,\dots,z_\fm)\mapsto \left(y_{\fm[1,k-1]+1},\dots,y_{\fm[1,k]}\right), \\ 
(\fx_1,\fx_2,\dots\fx_\fl)\mapsto 
\left(x_n^{-1},x_{n-1}^{-1},\dots,x_{\ell_{c_k-1}+1}^{-1}\right),\quad 
(\fx_1,\fx_2,\dots\fx_\fk)\mapsto 
\left(x_n^{-1},x_{n-1}^{-1},\dots,x_{\ell_{c_k}+1}^{-1}\right),
\\ 
(\fl> i_1>\dots>i_{\fm}\ge \fk) \mapsto 
\left(n-\ell_{c_k-1}> n-c_k^{(1)}>\dots>n-c_k^{(\fm_k)}\ge n-\ell_{c_k}\right),
\end{gather*}
and collect the factors that do not correspond to the ones in the left-hand
side of \eqref{eq:lemma} into a $\Phi(z_1,\dots,z_\fm)$. The holomorphicity
of $\Phi(z_1,\dots,z_\fm)$ in a neighborhood of $\{\fx_j\}_{j=1}^\fl$ 
is readily visible; and the fact that it is symmetric in $z_1,\dots,z_\fm$ 
follows from the fact that $f_{\mu^{\ge 1}}(y_1^{-1},\dots,y_m^{-1})$, that 
enters $\Phi(z_1,\dots,z_\fm)$ as the only factor that is not manifestly 
symmetric, is symmetric in $\left(y_{\fm[1,k-1]+1},\dots,y_{\fm[1,k]}\right)$
thanks to the commutativity of the $\C_i(x)$ operators (for a fixed $i$ 
and varying $x$) in the definition \eqref{eq:f-generic} of 
$f$-functions for coloured compositions. 

Comparing the right-hand sides of \eqref{eq:f-integral} and 
\eqref{eq:lemma},
we see that this completes the proof of Theorem \ref{thm:f-integral} modulo
the proof of Lemma \ref{lem}.
\end{proof}

\begin{proof}[Proof of Lemma \ref{lem}] While a direct book-keeping of residues of the two sides of \eqref{eq:lemma} might be possible, we will use the theory of Hall-Littlewood processes as a shortcut, with the work \cite[Section 2]{BorodinBW} as our main reference; a more detailed description can be found in \cite[Section 2]{BorodinC}. See also Remark \ref{rem:HL} on the origin of the argument given below. 
	
Consider an ascending Hall-Littlewood process with weights on sequences of partitions $\lambda^{(1)},\dots,\lambda^{(\fn)}$ proportional to 
\begin{equation}
\label{eq:HL-weight}
P_{\lambda{(1)}}(\fx_1)P_{\lambda{(2)}/\lambda{(1)}}(\fx_2)P_{\lambda{(\fn)}/\lambda{(\fn-1)}}(\fx_\fn)
Q_{\lambda{(\fn)}}(\rho), \qquad \ell(\lambda{(k)})\le k, \quad 1\le k\le \fn, 
\end{equation}
with, generally speaking, complex parameters $\{\fx_i\}_{i=1}^\fn$, $\rho$ being the specialization
of the algebra of symmetric functions into a sequence of variables $(b_1,b_2,\cdots)$, and $P_*$ and $Q_*$ being the Hall-Littlewood symmetric functions. 

Our argument is based on \cite[Proposition 2.2]{BorodinBW}, see also \cite[Propositoin 2.2.14]{BorodinC}, which says that for any $\fn\ge m_1\ge \dots\ge m_n\ge 1$, one has
\begin{multline}
\label{eq:from-BBW}
\E_{HL}\left(q^{m_1-\ell(\lambda(m_1))}\cdots q^{m_n-\ell(\lambda(m_n))}\right)
\\ =\frac{q^{\binom{n}{2}}}{(2\pi\sqrt{-1})^n}\oint\cdots\oint \prod_{i<j} \frac{z_j-z_i}{z_j-qz_i}
	\prod_{l=1}^n \left(\prod_{j\ge 1}\frac{1-z_lb_j}{1-qz_lb_j}
	\prod_{i=1}^{m_l} \frac{qz_l-\fx_i}{z_l-\fx_i}
	\frac{dz_l}{z_l}\right),
\end{multline}
where (positively oriented) $z_j$-contours are such that they surround $\{\fx_j\}_{j=1}^{m_1}$, $0$, and they are also $q$-nested in the sense that $z_i$-contour contains $q\cdot(z_j\text{-contour})$ for all $i<j$; no other poles are taken into account.  
 
Let us fix $1\le \fk<\fl\le \fn$, and consider the sum
\begin{equation}
\label{eq:q-ident}
\sum_{\fk\le i_\fm<i_{\fm-1}<\dots<i_1<\fl}\prod_{p=1}^m \frac{q^{i_p-\ell(\lambda(i_p))}-q^{i_p+1-\ell(\lambda(i_p+1))}}{1-q}=q^{\binom{\fm}{2}}
\cdot q^{m(\fk-\ell(\lambda(\fk)))} \cdot {\binom{\fl-\ell(\lambda(\fl))-\fk+\ell(\lambda(\fk))}{\fm}}_q.
\end{equation} 
The equality between the two sides of \eqref{eq:q-ident} is a (non-obvious) special case of Corollary \ref{cor:q-identity}. More exactly, choosing $Q=q$, $\alpha_i\equiv 1$, $|\beta|=\fm$, and reversing the order of indices of $\alpha_i$'s and $\beta_i$'s, turns \eqref{eq:q-identity} into
\begin{equation}
\label{eq:q-identity2}
q^{-\binom{\fm}{2}}\cdot\sum_{\substack{\beta_1,\dots,\beta_n\in\{0,1\}\\ \beta_1+\dots+\beta_n=\fm}}
\prod_{i=1}^n q^{\beta_i\cdot (i-1)}={\binom{n}{\fm}}_q.
\end{equation}
Let $1\le k_\fm<k_{\fm-1}<\dots<k_1\le n$ be the set of index values $k$ for which $\beta_{k}=1$ in \eqref{eq:q-identity2}.
Observe that 
$$
\frac{q^{j-\ell(\lambda(j))}-q^{j+1-\ell(\lambda(j+1))}}{1-q}=
\begin{cases}
q^{j-\ell(\lambda(j))}=q^{\fk-\ell(\lambda(\fk))}\cdot q^{j-\ell(\lambda(j))-\fk+\ell(\lambda(\fk))},&
\ell(\lambda(j))=\ell(\lambda(j+1)),\\
0,&\text{otherwise},
\end{cases}
$$
where `otherwise' refers to the only possible alternative $\ell(\lambda(j))+1=\ell(\lambda(j+1))$.

Set $n=\fl-\ell(\lambda(\fl))-\fk+\ell(\lambda(\fk))$, and let $\fk\le j_1<j_2<\dots<j_n<\fl$ be all the values of $j\in[\fk,\fl)$ such that $\ell(\lambda(j))=\ell(\lambda(j+1))$; equivalently 
$(j+1-\ell(\lambda(j+1)))-(j-\ell(\lambda(j))$ equals 1 rather than 0. 
These are all possible values that summation indices $i_1>\dots>i_\fm$ in the left-hand side of \eqref{eq:q-ident} can take to produce a nonzero term; let $j_{l_1}>\dots>j_{l_\fm}$ be the corresponding choices. Then matching $(k_1,\dots,k_\fm) \equiv (l_1,\dots,l_\fm)$ establishes the equivalence of \eqref{eq:q-ident} and \eqref{eq:q-identity2}, thus proving \eqref{eq:q-ident}. 

Our next step is to compute the averages, with respect to the ascending Hall-Littlewood process, of both sides of \eqref{eq:q-ident} using \eqref{eq:from-BBW}. 

The left-hand side of \eqref{eq:q-ident} is a simple linear combination of those from \eqref{eq:from-BBW}. Moving that linear combination inside the integrand and observing that 
$$
\frac1{1-q}\left(\prod_{i=1}^{i_p} \frac{qz_{p}-\fx_i}{z_{p}-\fx_i }- 	\prod_{i=1}^{i_p+1} \frac{qz_{p}-\fx_i}{z_{p}-\fx_i}\right)=\frac{z_p}{z_{p}-\fx_{i_p+1}}\prod_{i=1}^{i_p} \frac{qz_{p}-\fx_i}{z_{p}-\fx_i },
$$
we obtain
\begin{multline}
\label{eq:exp1}
\E_{HL} \left(\sum_{\fk\le i_\fm<i_{\fm-1}<\dots<i_1<\fl}\prod_{p=1}^m \frac{q^{i_p-\ell(\lambda(i_p))}-q^{i_p+1-\ell(\lambda(i_p+1))}}{1-q}\right)\\=
 \sum_{\fk\le i_\fm<i_{\fm-1}<\dots<i_1<\fl}
 \frac{q^{\binom{\fm}{2}}}{(2\pi\sqrt{-1})^\fm}\oint\cdots\oint \prod_{i<j} \frac{z_j-z_i}{z_j-qz_i}
 \prod_{p=1}^\fm \left(\prod_{j\ge 1}\frac{1-z_pb_j}{1-qz_pb_j}
 \frac1{z_{p}-\fx_{i_p+1}}\prod_{i=1}^{i_p} \frac{qz_p-\fx_i}{z_p-\fx_i}\,
 {dz_p}\right),
\end{multline}
with the same integration contours as in \eqref{eq:from-BBW}. 

Note that 0 is no longer a potential singularity of the integrand; thus, the contour may or may not contain it. Let us also explain why the presence or absence of the potential poles at $z_i=q^{-1}z_j$ does not affect the value of the integral. Generally speaking, the integral with $q$-nested contours is equal to the sum of residues at $z_p=q^{-k_p}\fx_{l_p}$, $1\le p\le\fm$, for certain values of $k_p\ge 0$ and $l_p\ge 1$, that arises by sequential evaluation of residues inside the $z_\fm, z_{\fm-1}, \dots, z_1$-contours in that order. Let $p^*$ be the maximal index such that $k_{p^*}>0$. Since the $z_\fm$-contour encircles only the poles at $\fx_j$'s, we must have $p^*\le \fm-1$. Also, since this pole must have come from a denominator factor $z_{j^*}-qz_{p^*}$ with $j^*>p^*$, due to the maximality of $p^*$, the pole must be at $z_{p^*}=q^{-1} \fx_{l^*}$ with $1\le l^*\le i_{j^*}$. But the integrand contains the factor $(qz_{p^*}-\fx_l^*)$, which will turn the residue into 0 (note that we need the fact that $i_{p^*}\ge i_{j^*}$ to guarantee the presence of this factor). We conclude that the $q$-nestedness is irrelevant for the value of the integral. 

Let us proceed to computing the Hall-Littlewood expectation of the right-hand side of \eqref{eq:q-ident}. 

Using the $q$-binomial theorem
$$
\sum_{j\ge 0} \frac{(a;q)_j}{(q;q)_j}\, z^j =\frac{(az;q)_\infty}{(z;q)_\infty}=\left(1-\frac zq\right)\cdots\left(1-\frac{z}{q^\fm}\right)\qquad \text{for}\quad a=q^{-\fm},
$$
we obtain 
$$
{\binom{A}{\fm}}_q=\frac{\left(1- {q^A}\right)\left(1- {q^{A-1}}\right)\cdots\left(1-q^{A-\fm+1}\right)}{(q;q)_\fm}=\sum_{j=0}^\fm 
	\frac{(q^{-\fm};q)_j}
	{(q;q)_j(q;q)_{\fm}}\,q^{(A+1)j} =\sum_{j=0}^\fm 
	\frac{(-1)^j q^{(A+1)j-\fm j+\binom{j}{2}}}
	{(q;q)_j(q;q)_{\fm-j}}\,.
$$
Hence, the right-hand side of \eqref{eq:q-ident} can be written as
$$
q^{\binom{\fm}{2}}
\cdot q^{m(\fk-\ell(\lambda(\fk)))} \cdot {\binom{\fl-\ell(\lambda(\fl))-\fk+\ell(\lambda(\fk))}{\fm}}_q
=\sum_{j=0}^\fm \frac{q^{\binom{\fm-j}{2}}}
{(q;q)_j(q;q)_{\fm-j}}\,q^{j(\fl-\ell(\lambda(\fl)))+(\fm-j)(\fk-\ell(\fk))},
$$
where for powers of $q$ we used $\binom{\fm}{2}-\fm j+j+\binom j2=\binom{\fm-j}{2}$. Employing \eqref{eq:from-BBW}, we now obtain
\begin{multline}
\label{eq:exp2}
\E_{HL}\left( q^{\binom{\fm}{2}}
\cdot q^{m(\fk-\ell(\lambda(\fk)))} \cdot {\binom{\fl-\ell(\lambda(\fl))-\fk+\ell(\lambda(\fk))}{\fm}}_q \right)
= \sum_{j=0}^\fm \frac{q^{\binom{\fm-j}{2}}}
{(q;q)_j(q;q)_{\fm-j}}\\ \times \frac{q^{\binom{\fm}{2}}}{(2\pi\sqrt{-1})^\fm}
\oint\cdots\oint\prod_{ i<j} 
\frac{z_j-z_i}{z_j-qz_i}
\prod_{p>0}^j \prod_{a=1}^\fl \frac{qz_p-\fx_a}{z_p-\fx_a}
\prod_{r>j}^\fm \prod_{b=1}^\fk \frac{qz_r-\fx_b}{z_r-\fx_b}
\prod_{p=1}^\fm \prod_{j\ge 1}\frac{1-z_pb_j}{1-qz_pb_j}\, \frac{dz_p}{z_p}\,,
\end{multline}
with the integration contours are as for \eqref{eq:from-BBW}. 

Since the two of sides of \eqref{eq:q-ident} are equal, the right-hand side of \eqref{eq:exp1}
and \eqref{eq:exp2} are also equal. This is literally the desired statement of Lemma \ref{lem}, eq. \eqref{eq:lemma}, with a specific choice of 
\begin{equation}
\label{eq:Phi-special}
\Phi(z_1,\dots,z_\fm)=\phi(z_1)\cdots\phi(z_\fm), \qquad \phi(z)=\text{const}\prod_{j\ge 1} \frac{1-z b_j}{1-qzb_j}\,,
\end{equation}
and a specific choice of contours for the right-hand side. The equality (of the right-hand sides of \eqref{eq:exp1} and \eqref{eq:exp2}) has been thus proven for generic $\{\fx_i\}$ and $\{b_j\}$, although certain inequalities on them are needed to make sure that the weights \eqref{eq:HL-weight} are summable (see \cite[Section 2]{BorodinBW} for details). However, the equality itself is an identity of finite sums (of residues), and the restrictions on $\{\fx_i\}$ and $\{b_j\}$ can thus be removed by analytic continuation.

The two sides of \eqref{eq:lemma} computed as (finite) sums of residues are linear combinations of values of the function $\Phi(z_1,\dots,z_\fm)$ at $\fm$-tuples of distinct (because of $\prod_{i<j}(z_j-z_i)$ in the integrand) points from the list of possible singularities, consisting of $\{\fx_i\}$, their $q$-multiples, and $0$. Values of $\Phi(z_1,\dots,z_\fm)$ with permuted $z_j$'s are equal (due to the symmetry of $\Phi$) and can be grouped together; their coefficients are certain explicit rational functions of $\{\fx_i\}$ and $q$. Fix an $\fm$-tuple of distinct possible poles. Using the freedom in the choice of $\{b_j\}$ and the constant prefactor in \eqref{eq:Phi-special}, we can make $\phi(z)$ to be arbitrarily close to 1 at the points of the chosen $\fm$-tuple, and arbitrarily close to 0 at all the other potential singularities (we are assuming that $\{\fx_i\}$ are generic, we only need to prove an identity between rational functions in them).\footnote{Indeed, $\phi(z)=\mathrm{const}\cdot\psi(x)/\psi(qz)$ with $\psi$ being an arbitrary polynomial that can be chosen to approximate any values at any finite set of points; the constant in front is needed to control $\phi(0)$.} This will lead to $\Phi$ being close to 1 at the chosen $\fm$-tuple and its permutations, and close to 0 at all other possible $\fm$-tuples. Since we already proved the equality of the two sides of \eqref{eq:lemma} for all $\Phi$'s of the form \eqref{eq:Phi-special}, this implies the equality of the coefficients of $\Phi$ evaluated at the chosen $\fm$-tuple (together its permutations) in residue expansions of the two sides of \eqref{eq:lemma}. This proves \eqref{eq:lemma} for arbitrary $\Phi$, but still with the particular choice of integration contours in the right-hand side, which are $q$-nested and include 0.

However, we already know that the declared freedom for the choice of contours is valid for the left-hand side, cf. the argument after \eqref{eq:exp1}. This implies that the values of $\Phi$ at $\fm$-tuples that include either 0 or a nontrivial $q$-multiple of a $\fx_i$ do not contribute to the left-hand side. Since we already know that such coefficients are the same on both sides, these values must not contribute to the right-hand side either, which means that we can use the same freedom of contours in the right-hand side without changing the value of the integral.      
\end{proof}

\begin{rmk}\label{rem:HL} Let us comment on the origin of our proof of Lemma \ref{lem} above that might have looked somewhat cryptic. If, in the setting of Theorem \ref{thm:f-integral}, $\mu$ has no zero parts and only one colour, then $f_\mu$ is colour-blind, according to \eqref{eq:symm}. Theorem \ref{thm:obs-avg} then implies that it is given by an average of a $q$-moment type observable \eqref{eq:obs-colour-blind} over a colour-blind stochastic vertex model. The height function of the colour-blind stochastic vertex models can be interpreted, along any down-right path in a quadrant (which includes horizontal lines), as lengths of partitions distributed according to Hall-Littlewood processes; this was the main result of \cite{BorodinBW}. Thus, we get an expression for $f_\mu$ in the form of an average over a Hall-Littlewood process.  
	
On the other hand, the symmetrization of colours in \eqref{eq:symm} can also be taken after the computation of the expectation, with respect to a rainbow coloured model, in the left-hand side of \eqref{eq:obs-avg}. As was shown in \cite[Chapter 10]{BW}, the distribution of coloured height functions at a single observation point can also be described via lengths of partitions in a Hall-Littlewood process, and this is where the computation of that expectation can take place. Once the corresponding average over the Hall-Littlewood process is computed, one can perform its colour  symmetrization. 

The two resulting expressions must be the same - the operations of averaging over our measure and symmetrizing over colours commute. Understanding the reason for that in the language of the Hall-Littlewood processes is not too difficult, this is essentially \eqref{eq:q-ident}. Rewriting what it means in terms of integral representations for averages of Hall-Littlewood observables results in a general identity for symmetric functions, which is exactly our Lemma \ref{lem} above.  
\end{rmk}

\section{Observables of stochastic lattice models}\label{sec:obs}
The purpose of this section is to combine the results of Sections \ref{sec:cauchy}-\ref{sec:integral} in order to obtain integral representations for averages of the observables introduced in Definition \ref{def:observables}, as well as to explore corollaries thereof. 

\subsection{The main result}\label{ssec:main}
Recalling graphical interpretation of Definition \ref{def:measure}, we consider 
a stochastic lattice model in the quadrant $\mathbb{Z}_{\ge 
1}\times\mathbb{Z}_{\ge 1}$, with vertex weights given by $L^\stoch_{x_i}$ in 
\eqref{eq:L-stoch} (see also \eqref{eq:L-weights} and \eqref{s-weights}) for 
the vertices in row $i$ from the bottom.  
The boundary conditions are as follows: No paths (equivalently, only paths of colour 
0) enter the quadrant through its bottom boundary; and along the left boundary 
we have a single path entering at every row, with the bottom $\lambda_1$ paths 
having colour 1, next $\lambda_2$ paths having colour 2, and so on. We focus on 
the state of the model between row $n$ and row $n+1$; that is, we record the 
locations where the paths of colour $1,2,\dots$ exit the $n$th row upwards as a 
$\lambda$-coloured composition $\nu$, where $\lambda$ is the composition with 
parts $\lambda_1,\lambda_2,\dots$, and we truncate this sequence so that 
$|\lambda|=n$. \footnote{Without loss of generality, we assume, for convenience 
of notation, that incoming paths in rows $n$ and $n+1$ have different colours.} 
The partial sums of $\lambda$ are denoted as $\sum_{i=1}^{k} \lambda_i = 
\ell_k$, $k\ge 1$; $\ell_0:=0$. 

The model is also allowed to have column inhomogeneities $\{s_j,\xi_j\}_{j\ge 
1}$, which replace the $(s,x_i)$ parameters in the $L^\stoch$-weight of the 
vertex in row $i$ and column $j$ by $(s_j, \xi_jx_i)$. cf. Section 
\ref{ssec:inhom}. For convenience, we assume that the number of column 
inhomogeneities is finite. (For our results this assumption is not restrictive 
as the state of the model far enough to the right will not play any role, and 
thus, due to stochasticity, the column inhomogeneities there can be chosen 
freely.)

Also recall the observables $\O_\mu$ of Definition \ref{def:observables}, defined for any $\lambda$-coloured composition $\mu=0^{\bm{m}^{(0)}}1^{\bm{m}^{(1)}}2^{\bm{m}^{(2)}}\cdots$, that take values 
\begin{multline}\label{eq:obs2}
\O_\mu(\nu)=\prod_{x\ge 1}\displaystyle\prod_{i\ge 1} 
q^{m_i^{(x)}H_{>i}^{\nu/\mu}(x+1)}{\displaystyle\binom{H_i^{\nu/\mu}(x+1)}{m_i^{
(x) } } } _q\\
=\prod_{x\ge 1}\displaystyle\prod_{i\ge 1}  
\frac{\bigl(q^{H_{>i}^{\nu/\mu}(x+1)}-q^{H_{\ge 
i}^{\nu/\mu}(x+1)}\bigr)\bigl(q^{H_{>i}^{\nu/\mu}(x+1)}-q^{H_{\ge 
i}^{\nu/\mu}(x+1)-1}\bigr)\cdots \bigl(q^{H_{>i}^{\nu/\mu}(x+1)}-q^{H_{\ge 
i}^{\nu/\mu}(x+1)-m_i^{(x)}+1}\bigr)}{(q;q)_{m_i^{(x)}}}
\end{multline}
given in terms of the coloured height functions \eqref{eq:coloured-height}.

For any $\lambda$-coloured $\mu$, we make a new coloured composition $\mu^{\ge 
1}$ that consists of its nonzero parts coloured in the same way with the 
colouring composition of $\mu^{\ge 1}$ denoted by $\lambda^{\ge 1}$, cf. 
Definition \ref{def:mu>1}. Clearly, given the colouring $\lambda$ of $\mu$, 
$\mu$ is uniquely reconstructed from $\mu^{\ge 1}$. We use the notation 
$m:=n-|\bm{m}^{(0)}|$ for the length of $\mu^{\ge 1}$, denote by 
$c_1<\dots<c_\alpha$ the colours of parts of $\mu^{\ge 1}$, and denote by 
$\fm_1,\dots,\fm_\alpha\ge 1$ the number of parts of $\mu^{\ge 1}$ of colours 
$c_1,\dots,c_\alpha$, respectively. Set $\fm[a,b]=\fm_a+\fm_{a+1}+\dots+\fm_b$.

We can now state the main result of this paper. 

\begin{thm}\label{thm:main} With the above notations, we have 
\begin{multline}
\label{eq:main}
\E\, \O_\mu= \frac{q^{\sum_{u\ge 1}\sum_{i>j} m_i^{(u)}m_j ^{(u)}}}{\prod_{j\ge 
1} (s^2;q)_{|\bm{m}^{(j)}|}}
\frac{1}{(2\pi\sqrt{-1})^m}\oint\cdots\oint_{\text{around }\{x_j^{-1}\}}
\prod_{1\le i<j\le m} \frac{y_j-y_i}{y_j-qy_i} \\
\times
\prod_{k=1}^{\alpha}\left(
\sum_{j=0}^{\fm_k} \frac{(-1)^j q^{\binom{\fm_k-j}{2}}}
{(q;q)_j(q;q)_{\fm_k-j}}
\prod_{p>\fm[1,{k-1}]}^{j+\fm[1,{k-1}]}\prod_{a>\ell_{c_{k}-1}}^n\frac{1-qx_ay_p}{1-x_ay_p}
\prod_{r>j+\fm[1,{k-1}]}^{\fm[1,k]}\prod_{b>\ell_{c_{k}}}^n\frac{1-qx_by_r}{1-x_by_r}
\right)
\\
\times  f_{\mu^{\ge 1}}^\stoch(\lambda^{\ge 1};y_1^{-1},\dots,y_m^{-1})
\prod_{i=1}^m 
\frac{(y_i-s)dy_i}{y_i^2}\,,
\end{multline}
where (positively oriented) integration contours are chosen to encircle all 
points $\{x_j^{-1}\}_{j=1}^n$ and no other singularities of the 
integrand, or as \emph{$q$-nested} closed simple curves with $y_{i}$-contour containing $q^{-1}\cdot (y_j\text{-contour})$ for all $i<j$, and all of the contours encircling $\{x_j^{-1}\}_{j=1}^n$. The contours can also be chosen to either encircle or not encircle the point $0$.   

The formula also holds in the column inhomogeneous setting under the assumption that $(s_x,\xi_x)=(s,1)$ for any $x\ge 1$ such that $|\bm{m}^{(x)}|>0$.
\end{thm}

\begin{proof} If the parameters satisfy the inequalities 
\eqref{eq:weight-condition3}, then the statement is a substitution of 
Theorem \ref{thm:f-integral} into Theorem \ref{thm:obs-avg}. Observe that 
$\O_{\mu}$ is independent of the state of the model to the right of the 
maximal coordinate of $\mu$. Thus, both sides of \eqref{eq:main} are actually
rational functions in $x_i$'s, and the extra assumption 
\eqref{eq:weight-condition3} can thus be removed by analytic continuation. 
\end{proof}	

\begin{rmk}
The factor $\prod_{j\ge 1}(s^2;q)^{-1}_{|\bm{m}^{(j)}|}$ in the right-hand side of \eqref{eq:main} does not make the expression singular in the finite spin situation $s^2=q^{-J}$, $J=1,2,\dots$, because the same factor appears in the numerator when one writes $f_{\mu^{\ge 1}}^\stoch=(-s)^{|\mu|}f_{\mu^{\ge 1}}$ explicitly by taking the factorization \eqref{eq:f-delta}, where this factor is manifest, and acting on it by difference operators \eqref{eq:exchange}-\eqref{eq:DL-operators} and colour merging \eqref{eq:f-colour-merge} as needed. 
\end{rmk}

\subsection{The colour-blind case}\label{ssec:main-colour-blind}
Let us see how Theorem \ref{thm:main} works in the colour-blind situation. 
The observables then simplify to \eqref{eq:obs-colour-blind}, and we obtain 

\begin{cor}\label{cor:main-colour-blind} 
In the colour-blind case $\lambda=(n)$, with ordered coordinates
$\theta_1\ge\dots\ge\theta_m\ge 1$ of $\mu^{\ge 1}$ and no column 
inhomogeneities, we have
\begin{multline}
\label{eq:main-colour-blind}
\E \left[(1-q^{H^\nu(\theta_1+1)})(1-q^{H^\nu(\theta_2+1)-1})\cdots 
(1-q^{H^\nu(\theta_m+1)-m+1})\right]\\
=\frac{(-1)^m (-s)^{|\theta|}}{(2\pi\sqrt{-1})^m}\oint\cdots\oint_{\text{around 
}\{x_j^{-1}\}}
\prod_{1\le i<j\le m} \frac{y_i-y_j}{y_i-qy_j}\prod_{p=1}^m 
\left(\frac{1-sy_p}{y_p-s}\right)^{\theta_i}\prod_{a=1}^n 
\frac{1-qx_ay_p}{1-x_ay_p}\,\frac{dy_p}{y_p}\,,
\end{multline}
where (positively oriented) integration contours are chosen to encircle all 
points $\{x_j^{-1}\}_{j=1}^n$ and no other singularities of the 
integrand, and $H^\nu(x)=\#\{i\in\{1,\dots,n\}:\nu_i\ge x\}$. 
\end{cor}
\begin{rmk}
\label{rem:main-colour-blind}
Eq. \eqref{eq:main-colour-blind} is readily seen to coincide with crucial 
\cite[Lemma 9.10]{BorodinP2}, which is essentially equivalent to the integral 
representation of the most general multi-point moments for the colourless 
stochastic vertex model along a single line. 

It is not hard to extend this
result, with very similar proof as the one given below, to the column 
inhomogeneous case under the condition that $(s_x,\xi_x)=(s,1)$ for any $x\ge 1$ 
such that no parts of $\mu$ are equal to $x$. However, \cite[Lemma 
9.11]{BorodinP1} gives such an extension to the fully column inhomogeneous 
situation. It remains unclear how to achieve this level of generality in the 
coloured case.
\end{rmk}
\begin{rmk}\label{rem:contours}
While the freedom in choosing the contours as $q$-nested is still there (it can 
be checked directly as in the proof of Lemma \ref{lem} or carried over from 
Theorem \ref{thm:main}), the contours cannot include 
the point 0 anymore, and its inclusion (together with $q$-nestedness) would 
actually change the $q$-shifted moments in the left-hand side into unshifted 
moments of \cite[Theorem 9.8]{BorodinP2}; see \cite[Section 9]{BorodinP2} 
for a detailed explanation of that transition.  
\end{rmk}

\begin{proof}[Proof of Corollary \ref{cor:main-colour-blind}] In the 
colour-blind case $\lambda=(n)$, using \eqref{eq:symm} and 
\eqref{eq:obs-colour-blind}, we write \eqref{eq:main} as 
\begin{multline}
\E \left[\frac{(1-q^{H^\nu(\theta_1+1)})(1-q^{H^\nu(\theta_2+1)-1})\cdots 
(1-q^{H^\nu(\theta_m+1)-m+1})}{\prod_{j\ge 
1}(q;q)_{\text{mult}_j(\theta)}}\right]= \frac{1}{\prod_{j\ge 
1} (s^2;q)_{\text{mult}_j(\theta)}}\\ \times
\frac{1}{(2\pi\sqrt{-1})^m}\oint\cdots\oint_{\text{around }\{x_j^{-1}\}}
\prod_{1\le i<j\le m} \frac{y_j-y_i}{y_j-qy_i}
\sum_{j=0}^{m} \left(\frac{(-1)^j q^{\binom{m-j}{2}}}
{(q;q)_j(q;q)_{m-j}}
\prod_{p=1}^j\prod_{a=1}^n\frac{1-qx_ay_p}{1-x_ay_p}
\right)
\\
\times  \F^{\sf c}_{\theta}(y_1^{-1},\dots,y_m^{-1})
\prod_{i=1}^m 
\frac{(y_i-s)dy_i}{y_i^2}\,,
\end{multline}
with contours around $\{x_j^{-1}\}$ and no other singularities. Observe that 
if the summation index $j$ in the integrand above takes any value $j<m$, then 
the integrand, viewed as a function in $y_m$, has no singularities at $\{x_*^{-1}\}$, and the integral 
vanishes. Hence, we can set $j=m$. \footnote{It is at this moment that we loose
the freedom to have 0 inside the contours, as the terms that we remove may have 
nontrivial residues at 0, cf. Remark \ref{rem:contours} above.} 

Further we write
$$
\prod_{i<j} \frac{y_j-y_i}{y_j-qy_i}=\prod_{ i\ne j} 
\frac{y_j-y_i}{y_j-qy_i}\cdot \prod_{i>j} \frac{y_j-qy_i}{y_j-y_i}\,,
$$
note that the first factor in the right-hand side is symmetric in $y_i$'s, and 
the same is true about all other parts of the integrand. Hence, we can 
sum over the second factor over all permutations of the $y_i$'s, which yields $(q;q)_m/(1-q)^m$ in the integrand and multiplies the value 
of the integral by $m!$. 

Finally, from \cite[Theorem 4.12]{BorodinP2} we read
$$
\F^{\sf c}_{\theta}(y_1^{-1},\dots,y_m^{-1})=\frac{\prod_{j\ge 
1} (s^2;q)_{\text{mult}_j(\theta)}}{\prod_{j\ge 
1}(q;q)_{\text{mult}_j(\theta)}}\,\frac{(1-q)^m}{\prod_{i=1}^m (1-sy_i^{-1})}
\sum_{\sigma\in\mathfrak{S}_m} 
\sigma\left(\prod_{i<j}\frac{y_i^{-1}-qy_j^{-1}}{y_i^{-1}-y_j^{-1}}
\prod_{i=1}^m \left(\frac{y_i^{-1}-s}{1-sy_i^{-1}}\right)^{\theta_i}\right),
$$
where the sum is over all permutations $\sigma$ in the symmetric group 
$\mathfrak{S}_m$ on $m$ symbols, and $\sigma$'s permute the variables 
$y_1,\dots,y_m$ in the expression that they are applied to.
Since the rest of the integrand is symmetric in the $y_i$'s, we can remove
the sum over permutation leaving only the term with $\sigma=\text{id}$, and 
divide the integral by $m!$.

Implementing the above transformation and canceling common factors yields 
\eqref{eq:main-colour-blind}.
\end{proof}

\subsection{Duality}\label{ssec:duality} One convenient feature of the integral representation \eqref{eq:main} for $\E\O_\mu$ is that the dependence on values of coordinates of $\mu$ is concentrated in the factor $f_{\mu^{\ge 1}}(\lambda^{\ge 1};y_1^{-1},\dots,y_m^{-1})$ of the integrand. This allows to easily derive certain difference equations that $\E\O_\mu$ must satisfy. Let us see how this works. 

The skew-Cauchy identity \eqref{eq:fG-skew} was stated for the rainbow case but 
is readily extended to the colour-merged case with the help of Proposition 
\ref{prop:fG-merge} (in fact, this colour merging argument was already used in 
the proof of Proposition \ref{prop:fG-merged} for $G_*$'s replaced by their 
limits $\cG_*$'s). In particular, it implies
\begin{equation*}
\sum_{\mu^{\ge 1}\text{ is $\lambda^{\ge 1}$ coloured}} f_{\mu^{\ge 1}}(\lambda^{\ge 1};y_1^{-1},\dots,y_m^{-1})G_{\mu^{\ge 1}/\varkappa}((qX)^{-1})=
\prod_{i=1}^m \frac{1-y_iX}{1-qy_{i}X}\cdot 
 f_{\varkappa}(\lambda^{\ge 1};y_1^{-1},\dots,y_m^{-1}),
\end{equation*}
where $\varkappa$ is an arbitrary $\lambda^{\ge 1}$-coloured composition, $X$ is a complex parameter.
Multiplying both sides by $(-s)^{|\varkappa|}=(-s)^{|\varkappa|-|\mu|}(-s)^{|\mu|}$, we can also replace $f$ by $f^\stoch$ and $G$ by $G^\stoch$ in this identity, where $G^\stoch_{\mu^{\ge 1}/\varkappa}=(-s)^{|\varkappa|-|\mu|}G_{\mu^{\ge 1}/\varkappa}$ is a stochastic kernel with matrix elements built from stochastic $M^\stoch$-weights \eqref{eq:M-stoch}. 

Observe that the prefactor is of the form of the exact inverse of the factors in the integrand of \eqref{eq:main} that depend on $x_i$'s. Hence, if we compute the expectation $\E_{n+1}\O_\mu$ along the $(n+1)$st row of the stochastic vertex model, with left entering colour $(n+1)$ and rapidity $X$ used in that row, and subsequently sum against the stochastic kernel $G^\stoch_{\mu^{\ge 1}/\varkappa}((qX)^{-1})$, the result will coincide\footnote{subject to certain convergence conditions that we are ignoring here} with the expectation of $\E_n\O_\mu$ computed along the $n$th row:
\begin{equation}\label{eq:duality}
\sum_{\mu^{\ge 1}}\E_{n+1}\O_\mu \cdot  G^\stoch_{\mu^{\ge 1}/\varkappa}((qX)^{-1}) = \E_n\O_\varkappa.
\end{equation}

As the transition from row $n$ to row $n+1$ is also realized by a stochastic kernel, we observe a \emph{duality} of the action of two stochastic operators on the \emph{duality functional} $\O_\mu$. 

Of course, the above arguments only verify the duality relation \eqref{eq:duality} for a specific class of distributions on row $n$. However, this class is sufficiently general ($n\ge m$ and parameters $x_1,\dots,x_n$ are arbitrary), and it is quite plausible that the duality relation will hold for generic distributions on coloured compositions $\mu^{\ge 1}$ on the $n$th row, and also for the $n$th row being the whole lattice $\mathbb{Z}$, rather than $\mathbb{Z}_{\ge 1}$ with a path entering from the left. 

It would be very interesting to see an independent verification of \eqref{eq:duality}, possibly by a reduction to duality functional constructed in \cite{Kuan}. Given that spectral decomposition of stochastic kernels $G^\stoch$ are known, cf. \cite[Section 9.5]{BW}, this  would likely lead to an alternative proof of Theorem \eqref{eq:main}, apart from other possible applications.

For a colour-blind version of the above discussion see \cite[Section 8.5]{BorodinP2}, \cite[Section 8.5]{BorodinP1}, and references therein. 

\subsection{The rainbow case}\label{ssec:main-rainbow} Let us now focus on the rainbow sector, with the colouring composition $\lambda$ being $(1,1,\dots, 1)$. The simplification of the observables $\O_\mu$ in this case was given in \eqref{eq:obs-rainbow}, and this leads us to 

\begin{cor}
\label{cor:main-rainbow}
In the notations of Section \ref{ssec:main}, assume that $\lambda=(1,1,\dots,1)$. Then 
\begin{multline}
\label{eq:main-rainbow}
\E \displaystyle\prod_{\substack{i,j\ge 1\\ m_i^{(j)}=1}}  \frac{q^{H_{\ge i}^{\nu/\mu}(j+1)}-q^{H_{> i}^{\nu/\mu}(j+1)}}{q-1}
 = \frac{q^{\sum_{j\ge 1} \binom{|\bm{m}^{(j)}|}{2}}}{\prod_{j\ge 1} (s^2;q)_{|\bm{m}^{(j)}|}}
 \frac{(-1)^m}{(2\pi\sqrt{-1})^m}\oint\cdots\oint_{\text{incl.}\{x_j^{-1}\}}
\prod_{1\le i<j\le m} \frac{y_j-y_i}{y_j-qy_i} \\
\times f^\stoch_{\mu^{\ge 1}}(y_1^{-1},\dots,y_m^{-1})\prod_{j=1}^m 
\left(\frac{y_j-s}{y_j} \frac{x_{c_j}}{1-x_{c_j}y_j}\prod_{i>c_j}^n 
\frac{1-qx_iy_j}{1-x_iy_j}\, dy_j\right),
\end{multline}
where (positively oriented) integration contours are chosen to encircle all 
points $\{x_j^{-1}\}_{j=1}^n$ and no other singularities of the 
integrand, or as \emph{$q$-nested} closed simple curves with $y_{i}$-contour containing $q^{-1}\cdot (y_j\text{-contour})$ for all $i<j$, and all of the contours encircling $\{x_j^{-1}\}_{j=1}^n$.

The formula also holds in the column inhomogeneous setting under the assumption that $(s_x,\xi_x)=(s,1)$ for any $x\ge 1$ such that $|\bm{m}^{(x)}|>0$.
\end{cor}
\begin{proof} This can be obtained by either a direct substitution $\lambda=(1,\dots,1)$ into Theorem \ref{thm:main}, or a substitution of Proposition \ref{prop:integral-rainbow} into the rainbow case of Theorem \ref{thm:obs-avg}. 
\end{proof}

\begin{rmk}\label{rem:factorization} Denote the parts of $\mu^{\ge 1}$ by $(\theta_1,\dots,\theta_m)$, where $\theta_1$ carries the smallest colour $c_1$ of those represented in $\mu^{\ge 1}$, $\theta_2$ carries the next smallest colour $c_2$, \emph{etc.} Then in the anti-dominant case, when 
$\theta_1\le\theta_2\le\dots\le\theta_m$, $f_\mu^\stoch(y_1^{-1},\dots,y_m^{-1})$ in the integrand of \eqref{eq:main-rainbow} completely factorizes, cf. \eqref{eq:f-delta}, \eqref{eq:f-delta-inhom}, and Remark \ref{rem:recolouring} (we write the expression for the column homogeneous case below):
\begin{equation}
\label{eq:f-theta}
f_{\mu^{\ge 1}}^\stoch(y_1^{-1},\dots,y_m^{-1})
=
{\prod_{j\ge 1} (s^2;q)_{|\bm{m}^{(j)}|}}\prod_{i=1}^{n} \frac{y_i}{y_i-s}
\prod_{i=1}^{n} \left( \frac{1-s y_i}{y_i-s} \right)^{\theta_i},
\end{equation}
leading to a completely factorized integrand in \eqref{eq:main-rainbow}. In this case there is also another path to \eqref{eq:main-rainbow}. Namely, using the shift invariance property established in \cite{BGW}, one can rewrite the left-hand side of \eqref{eq:main-rainbow} in terms of an average for a combination of $q$-moments of the height function for a colour-blind vertex on a quadrant taken at points along a down-right path in the quadrant. In their turn, such moments possess explicit integral representations, see \cite{BorodinBW} for details. See also Remark \ref{rem:shift-invariance} below for a related observation. 
\end{rmk}

\begin{rmk}\label{rem:colour-position} 
The identity \eqref{eq:main-rainbow} admits two colour-position symmetries, one for each side. 

For the left-hand side, \cite[Theorem 7.3]{BorodinBufetov} shows that the joint distributions of the coloured height functions of the (rainbow) coloured stochastic vertex model along boundaries of a certain class of down-right domains are symmetric with respect to rotations of the domains by 180 degrees that also swap the roles of colour and position of the entering/exiting paths. In our case, the domain is the rectangle, and the application of this symmetry will swap colours and positions in the observables, as well as the roles of rows and columns. This will give integral formulas for averages of the new observables, as well as indicate that those observables are also likely to be duality functionals for the same reasons as those in Section \ref{ssec:duality}.

For the right-hand side, if we represent $f_\mu^\stoch$ as a result of a 
sequence of applications of the difference operators $T_i$ given by 
\eqref{eq:exchange}-\eqref{eq:DL-operators} to a factorized expression of the 
type \eqref{eq:f-theta}, then one can use self-adjointness of the $T_i$'s 
with respect to the integral scalar product with weight $\prod_{i<j} 
(y_j-y_i)/(y_j-qy_j)$, cf. \cite[Proposition 8.1.3]{BW}, to move the application 
of the $T_i$'s to the fully factorized part of the integrand. Treating that 
part as an analog of \eqref{eq:f-delta}, or rather of \eqref{eq:f-delta-inhom}, 
at the specific value of $s=q^{-1/2}$, we will see a new $f$-like
function appearing in the integrand, that will utilize suitably permuted colours $c_1,\dots,c_m$ for its coordinates, and horizontal rapidities $x_1,\dots,x_n$ as its inhomogeneities. Thus, we will see a similar formula with positions and colours, as well as rows and columns, swapped.   

It would be interesting to see if applying both symmetries, at least in the case $s=q^{-1/2}$, returns one to the original formula, but we will not pursue that here.   
\end{rmk}

\subsection{Fusion}\label{ssec:fusion} The goal of this section is to fuse rows of the stochastic vertex model from Section \ref{ssec:main} and to see what Theorem \ref{thm:main} turns into in that situation. 

First, one can fuse finitely many rows, which in terms of individual vertices corresponds to the outer sum in \eqref{eq:fusion} (one can think of the inner sum in that relation as already performed, with our spin parameter $s$ being $q^{-\m/2}$). This starts by replacing a single row of colour $c$ with rapidity $x$ by $\l\ge 1$ rows of the same colour $c$ and rapidities forming a finite geometric progression $x,\dots,q^{\l-1}x$. Since our left boundary condition in these $\l$ rows consists of all incoming edges occupied by paths of the same colour, no summation along that boundary is necessary.
Hence, the $\l$ rows of vertices with $L^\stoch$-weights can be collapsed into a single row with the weight of a vertex in that row and column $j$ being $W_{\l,\m}(s_j^{-1}\xi_j x;q; *)|_{q^{-\m}=s_j^2}$, where $(s_j,\xi_j)$ are the column inhomogeneities as in Section \ref{ssec:inhom}, and the appearance of $s_j^{-1}$ in front of $\xi_jx_j$ is due to the argument $x/s$ in the expression \eqref{eq:L-stoch} of $L^\stoch$-weights in terms of $W$-weights. 

The right-hand side of \eqref{eq:main} also behaves well 
with respect to such fusion. More exactly, it leads to a simple replacement of 
all factors of the form $(1-qxy_k)/(1-xy_k)$, for various $k$ and $x=x_*$ being 
the rapidity of the fused row, by
\begin{equation}
\label{eq:chain}
\frac{1-qxy_k}{1-xy_k}\cdot \frac{1-q^2xy_k}{1-qxy_k}\cdots 
\frac{1-q^\l xy_k}{1-q^{\l-1}xy_k}=
\frac{1-q^\l xy_k}{1-xy_k}\,.
\end{equation}
Thus, the right-hand side can be immediately analytically continued in $q^\l$. 
It is not clear, however, what that would mean on the side of the stochastic 
vertex model as the 
left boundary condition in the fused row consists of $\l$ paths. 

To remedy this situation, we will perform a special limit transition with 
vertical inhomogeneity in column 1; this is parallel to what was done in 
\cite[Section 6]{BGW}. More exactly, we will rely on the limiting relation 
\eqref{eq:weight-limit}. According to the left-hand side of that relation, we 
will take $\xi_1=\zeta/s_1$ to turn the weight in the first column of the 
fused row into
$W_{\l,\m}(s_1^{-2} \zeta x;q; *)|_{q^{-\m}=s_1^2}$, and then take the limit $s_1\to 
0$. Here $z$ is an additional parameter that remains finite in the limit 
$s_1\to 0$; it regulates the strength of the left boundary. As we are about to 
make the rest of the model homogeneous, let us also set $\zeta$ to $s$, as this 
value will make the first column `blend in' with the other ones.  Thus, our 
limit results, by virtue of \eqref{eq:weight-limit}, a random number of 
paths of colour $c$ passing horizontally from column 1 to column 2 in the fused 
row, with the distribution of this random number given by 
\begin{equation*}
\text{Prob}\{k\}=
\dfrac{(s q^\l x;q)_\infty}{(s x ;q)_\infty}\dfrac{(q^{-\l};q)_k}{(q;q)_k}\,(s q^\l x)^k,\qquad k\ge 0.
\end{equation*}
Note that for $\l\in\mathbb{Z}_{\ge 1}$, this distribution is supported by $\{0,1,\dots,\l\}$,
as it should be, as an $\l$-fused row cannot carry more than $\l$ paths. But this distribution is also suitable for analytic continuation in $q^\l$, which we will use momentarily.  

Let us now discuss vertices in the fused row and other columns. Their weights $W_{\l,\m}(s_j^{-1}\xi_j x;q; *)$ are given by the right-hand side of \eqref{eq:W-weights}, which is explicit but rather complicated. We will consider a simpler situation instead, governed by the $q$-Hahn specialization of Section \ref{ssec:q-Hahn}. Hence, we will specialize $(x,s_j,\xi_j)\mapsto (s,s,1)$ to turn the weights into more tangible expressions as in the right-hand side of \eqref{eq:q-Hahn}. 

Finally, rather than performing fusion for a particular row of colour $c$ as we have done above (replacing it by $\l$ columns first and then fusing them together), let us do it for every row of the stochastic vertex model of Section \ref{ssec:main}. 

Gathering all the pieces together, we obtain a stochastic vertex model in the quadrant $\mathbb{Z}_{\ge 2}\times \mathbb{Z}_{\ge 1}$, depending on three parameters $q$, $s(=q^{-\m/2})$, and $z(=q^{-\l/2})$ satisfying $|q|,|s|,|z|,|s/z|<1$, defined as follows.

\begin{itemize}
	
	\item Along the boundary of the quadrant, no paths enter the quadrant 
through its bottom boundary; and along the left boundary we have a random number 
of paths entering at every row, with the bottom $\lambda_1$ rows hosting 
entering paths of colour 1, next $\lambda_2$ rows hosting paths of colour 2, and 
so on; the sequence $\{\lambda_j\}_{j\ge 1}$ of nonnegative integers is given.
As before, the partial sums of $\lambda$ are denoted as $\sum_{i=1}^{k} 
\lambda_i = 
\ell_k$, $k\ge 1$; $\ell_0:=0$. 
The distribution of the number of paths that enter in any row is given by 
	\begin{equation}
	\label{eq:first-column}
	\text{Prob}\{k\}=
	\dfrac{(s^2/z^2;q)_\infty}{(s^2 ;q)_\infty}\dfrac{(z^2;q)_k}{(q;q)_k}\,\left(\frac{s^2}{z^2}\right)^k,\qquad k\ge 0.
	\end{equation}

    \item The vertex weights for the model are given by 	
\begin{equation}
\label{eq:fused-bulk}
\text{weight}_{s,z} \left( 
\tikz{0.3}{
	\draw[lgray,line width=1pt,->] (-1,0) -- (1,0);
	\draw[lgray,line width=1pt,->] (0,-1) -- (0,1);
	\node[left] at (-1,0) {\tiny $\BB$};\node[right] at (1,0) {\tiny $\DD$};
	\node[below] at (0,-1) {\tiny $\AA$};\node[above] at (0,1) {\tiny $\CC$};
}\right)= \left(\frac{s^2}{z^2}\right)^{|\DD|}\, \frac{(s^2/z^2;q)_{|\AA|-|\DD|}(z^2;q)_{|\DD|}}{(s^2;q)_{|\AA|}}\,q^{\sum_{i<j}D_i(A_j-D_j)}\prod_{i\ge 1}
{\binom{A_i}{A_i-D_i}}_q.
\end{equation}	
\end{itemize}   

This model is more general than the one in Section \ref{ssec:main} in the sense 
that it can carry any number of paths along any edges, not just the vertical 
ones. However, it is less general in that there are no remaining row and column 
inhomogeneities.\footnote{In fact, we could have left $s$ and $z$ parameters 
column and row dependent, respectively, but chose not to do it for the sake of 
simplicity. On the other hand, we could not have left row and column rapidities 
$x_*$, $\xi_*$ generic and still had the same factorized form of the weights.}
 
As before, we focus on the paths that cross upwards from row $n$ to row $n+1$ 
for some $n\ge 1$.\footnote{Since rows above $n$ do not matter to us, we 
assume, as in Section \ref{ssec:main} above, that colours of left-entering arrows in rows $n$ 
and $n+1$ are different.} Encoding colours and horizontal positions of 
these crossings by a coloured composition $\nu$, we can define observables
$\O_\mu$ on the set of possible $\nu$'s by the same formula \eqref{eq:obs2}, 
where $\mu=2^{\bm{m}^{(2)}}3^{\bm{m}^{(3)}}\cdots$ is an arbitrary coloured
composition with no parts smaller than 2; this is what used to be $\mu^{\ge 1}$.
Let us use the familiar notations $m=\ell(\mu)$ for the length of $\mu$,  
$c_1<\dots<c_\alpha$ for the colours represented in $\mu$, and  
$\fm_1,\dots,\fm_\alpha\ge 1$ for the number of parts of $\mu$ of 
colours $c_1,\dots,c_\alpha$, respectively. Finally, let $\text{colour}(\mu)$ 
be the composition that encodes the colouring of $\mu$, and let $\mu-1^m$
denote the composition obtained from $\mu$ by subtracting 1 from each part.

\begin{cor} \label{cor:main-fused}
With the above notations for the fused stochastic vertex model, we 
have
 \begin{multline}
\label{eq:main-fused}
\E\, \O_\mu= \frac{q^{\sum_{u\ge 2}\sum_{i>j} m_i^{(u)}m_j ^{(u)}}}{\prod_{j\ge 
2} (s^2;q)_{|\bm{m}^{(j)}|}}
\frac{(-s)^m}{(2\pi\sqrt{-1})^m}\oint\cdots\oint
\prod_{1\le i<j\le m} \frac{y_j-y_i}{y_j-qy_i} \\
\times
\prod_{k=1}^{\alpha}\left(
\sum_{j=0}^{\fm_k} \frac{(-1)^j q^{\binom{\fm_k-j}{2}}}
{(q;q)_j(q;q)_{\fm_k-j}}
\prod_{p>\fm[1,{k-1}]}^{j+\fm[1,{k-1}]}\left(\frac{1-z^{-2}sy_p}{
1-sy_p}\right)^{n-\ell_{c_{k}-1}}
\prod_{r>j+\fm[1,{k-1}]}^{\fm[1,k]}\left(\frac{1-z^{-2}sy_r}{
1-sy_r}\right)^{n-\ell_{c_{k}}}
\right)
\\
\times  f_{\mu-1^m}^\stoch({\rm colour}(\mu);y_1^{-1},\dots,y_m^{-1})
\prod_{i=1}^m 
\frac{dy_i}{y_i^2}\,,
\end{multline}
where the integration contours are either positively oriented and $q$-nested around 
$s^{-1}$ with $y_{i}$-contour containing $q^{-1}\cdot 
(y_j\text{-contour})$ for all $i<j$, or negatively oriented and $q$-nested 
around $s$, with $y_j$-contour containing $q\cdot(y_i-\text{contour})$ for all 
$i<j$. The point 0 can be either inside or outside the contours in either 
case.
\end{cor}

\begin{proof} The starting point is Theorem 5.1 with $\mu$ from 
\eqref{eq:main-fused} being $\mu^{\ge 1}$ there. 

The first step is to turn each row of the quadrant into $\l$ rows of the same 
colour, and fuse them as was described above. This does not 
affect the integral representation much, except for the change described around 
\eqref{eq:chain}. 

The next step is the limit transition in column $1$ described 
above. To see how it affects the factor $f_{\mu^{\ge 
1}}^\stoch(y_1^{-1},\cdots,y_m^{-1})$, recall the partition function definiton 
\eqref{eq:f-function} (and its coloured
modification \eqref{eq:f-generic}) of the $f$'s. Since all parts of $\mu^{\ge 1}$
are assumed to be at least 2, the first column contains only vertices of the 
form 
$\tikz{0.3}{
	\draw[lgray,line width=0.7pt,->] (-1,1) -- (1,1);
	\draw[lgray,line width=2.5pt,->] (0,0) -- (0,2.2);
	\node[left] at (-1,1) {\tiny $i$};\node[right] at (1,1) {\tiny $i$};
	\node[below] at (0,0) {\tiny $\bm{0}$};\node[above] at (0,2) {\tiny 
$\bm{0}$};
}
$ 
that have $L^\stoch$-weights 
$$
\frac{(-s_1)(\xi_1 
y^{-1}_j-s_1)}{1-s_1\xi_1 y^{-1}_j}=\frac{(-s_1)(\xi_1-s_1y_j)}{y_j-s_1\xi_1 
}=\frac{(-s_1)(s/s_1-s_1y_j)}{y_j-s 
}=\frac{s_1^2 y_j-s}{y_j-s 
},\qquad 1\le j\le m, 
$$
where we used the value $\xi_1=s/s_1$ as above. The limit $s_1\to 0$ of this 
expression is $-s/(y_j-s)$, and the product over all $1\le j\le m$ gives
$(-s)^m \prod_{j=1}^m (y_j-s)^{-1}$. Adding that to the integrand of 
\eqref{eq:main} yields the integrand of \eqref{eq:main-fused}, together with
the replacement $\mu^{\ge 1}\mapsto \mu-1^m$ in the index of $f^\stoch$, where 
the subtraction of $1^m$ is responsible 
for removing a column from the partition function representation of 
$f^\stoch$ that we just performed. 

Let us now look at the contours. For the application of Theorem 5.1 we could 
choose them to $q$-nest around $\{x_i,qx_i,\dots,q^{\l-1}x_i\}_{i=1}^n=
\{s,qs,\dots,q^{\l-1}s\}$ and either contain 0 or not. (We could not choose
them to encircle $\{x_i,qx_i,\dots,q^{\l-1}x_i\}_{i=1}^n$ and no other 
singularities as this choice of horizontal rapidities forces to include 
at least some singularities of $\prod_{i<j} (y_j-qy_i)^{-1}$ into the contours.)
Since the integrand is manifestly nonsingular at $qs,\dots,q^{\l-1}s$, we can 
remove the condition of encircling those points. As the integrand is readily 
seen to not have poles at $y_*=\infty$, we can also move the contours through 
$\infty$ and 
have them $q$-nest around the only other singularity, which is at $s$ (again, 
$0$ can be either inside or outside). 

Thus, we have now proved \eqref{eq:main-fused} for $\l=1,2,\dots$, and 
the final step consists in analytic continuation in $q^\l=z^{-2}$ from 
the set of points $\{q,q^2,q^3,\dots\}$ accumulating at 0. This, 
however, is straightforward: The only dependence on $z$ of the right-hand side 
is through factors $(1-z^{-2}y_*)$, and the left-hand side is readily seen
to be given by uniformly convergent series with rational terms at least as long 
as $|q|,|s|,|sz^{-1}|<const<1$.
\end{proof}
\begin{rmk}
	The principal reason for our carrying the (incomplete) column inhomogeneity of the model
throughout Sections \ref{sec:extensions}-\ref{sec:obs} was to be able to perform the limit transition in column 1 that was just described, which gave us access to averages of observables in the fused models and, as a consequence discussed in the next section, in integrable models 
of directed random polymers. 
\end{rmk}

\section{Limit to polymers}\label{sec:polymers}
The goal of this section is to explore the consequences of Corollary \ref{cor:main-fused} for a few models of directed polymers in (1+1) dimensions. The exposition of the limit transitions from fused coloured stochastic models to directed polymers follows \cite[Section 7]{BGW}. 

\subsection{Continuum stochatic vertex model}\label{ssec:cont-vertex} Let us start by introducing a vertex model that will serve as a limiting object for the fused vertex model of Section \ref{ssec:fusion} described by weights \eqref{eq:first-column}-\eqref{eq:fused-bulk}. 

As in the fused case, the vertices of the continuum model will be parameterized by points of a quadrant, and to keep the notation parallel to that of Section \ref{ssec:fusion}, we will use the quadrant $\mathbb{Z}_{\ge 2}\times \mathbb{Z}_{\ge 1}$. 

Each vertex will have a certain \emph{mass} of each colour $\ge 1$ entering from the bottom and from the left, and exiting through the top and to the right. The mass is a real number in $[0,\infty)$, and for each vertex the total number of colours that have nonzero mass entering the vertex will always be finite. The mass of each colour passing through a vertex will always be preserved -- the sum of incoming mass from the bottom and from the left must be equal to the sum of exiting mass to the right and through the top. 

Let us denote the masses of colours $1,2,\dots$ entering through the bottom of a vertex by $\aa=(\alpha_1,\alpha_2,\dots)$, entering from the left by $\bb=(\beta_1,\beta_2,\dots)$, exiting through the top by $\cc=(\gamma_1,\gamma_2,\dots)$, and exiting to the right by $\dd=(\delta_1,\delta_2,\dots)$, respectively. The mass preservation means $|\aa|+|\bb|=|\cc|+|\dd|$. The notation is chosen to be in parallel with $(\AA,\BB;\CC,\DD)$ notation for the vertex models, as in \eqref{eq:W-picture}.  

Recall that a random variable with values in $(0,1)$ is said to be \emph{Beta-distributed} with parameters $a,b>0$ if it has a density, with respect to the Lebesgue measure, given by 
\begin{equation}
\label{eq:beta-density}
\frac{\Gamma(a+b)}{\Gamma(a)\Gamma(b)}\,x^{a-1}(1-x)^{b-1}, \qquad 0<x<1. 
\end{equation}

Given the coloured masses $\aa,\bb$ entering a vertex of our continuum vertex model, the coloured masses $\cc,\dd$ exiting the vertex are random and determined as follows. The procedure has two parameters $a,b>0$, same as in the Beta-distribution \eqref{eq:beta-density}.  
If all coordinates $\aa$ are zero, \ie\ $\aa=\bm{0}$, then we set $\cc=\bb$ and $\dd=\bm{0}$. If $\aa\ne \bm{0}$, let $n\ge 1$ be the maximal natural number such that $\alpha_n\ne 0$, and let $\zeta$ be an $(a,b)$-Beta distributed random variable. Then we set $\delta_{n+1},\delta_{n+2},\dots$ to 0 and define $\delta_n,\delta_{n-1},\dots,\delta_1$ recursively via
\begin{equation}
\label{eq:vertex-continuum}
\begin{aligned}
\delta_n&=-\log\bigl(e^{-\alpha_n}+\zeta(1-e^{-\alpha_n})\bigr),\\
\delta_{n-1}+\delta_n&=-\log\bigl(e^{-\alpha_{n-1}-\alpha_n}+\zeta(1-e^{-\alpha_{n-1}+\alpha_n})\bigr),\\
\dots &\qquad  \dots \\
\delta_1+\ldots+\delta_{n-1}+\delta_n&=-\log\bigl(e^{-\alpha_1-\ldots-\alpha_{n-1}-\alpha_n}+\zeta(1-e^{-\alpha_1-\ldots-\alpha_{n-1}+\alpha_n})\bigr).
\end{aligned}  
\end{equation}
One can show that this implies $0<\delta_j<\alpha_j$ for $1\le j\le n$. 
Finally, we set $\cc=\bb+(\aa-\dd)$, thus enforcing mass conservation. 

In addition to defining what happens at the vertices of the quadrant, we need to specify boundary conditions. As before, we will assume that no mass enters the quadrant from the bottom, \ie\ $\aa=\bm{0}$ for all vertices in the bottom row. On the other hand, along the left boundary we will assume that for the left most vertex in row $i$, the left-entering coloured mass $\bb$ has all but one coordinates equal to 0, with the exception of the $i$th one, which is $(a,b)$-Beta distributed.

As usual, we think of the randomness as have no space dependency, which means that the Beta-distributed random variables at different vertices, as well as those used to define the left boundary condition, are independent. 

The following statement was proved in \cite[Corollary 6.22]{BGW}. 

\begin{prop}\label{prop:limit-to-continuum} Consider the fused coloured vertex 
model defined around \eqref{eq:first-column}-\eqref{eq:fused-bulk} and set
\begin{equation}
\label{eq:limit-parameters}
q=\exp(-\epsilon),\qquad s^2=q^\sigma, \qquad z^2=q^\rho,
\end{equation}
for some $\sigma>\rho>0$ and $\epsilon>0$. Then as $\epsilon\to 0$, the fused coloured vertex models scaled by $\epsilon$ converges to the continuum vertex model defined above with parameters $(a,b)=(\sigma-\rho,\rho)$, in the sense that any finite collection of numbers of paths of arbitrary fixed colours entering/exiting any fixed set of vertices in fixed directions, when multiplied by $\epsilon$, weakly converges to the collection of corresponding  colour masses entering/exiting the corresponding vertices of the continuum model. 
\end{prop}

This immediately implies the convergence of the averages of the observables from Corollary \ref{cor:main-fused} as well, but we will postpone the limiting statement until we reformulate the continuum vertex model as a directed random polymer in the next section. 
 
\subsection{Random Beta-polymer}\label{ssec:beta} The Beta-polymer was first introduced in \cite{BarCor}. In order to define it, let $\{\eta_{t,m}\}_{t,m\ge 1}$ be a family of independent identically Beta-distributed random variables with parameters $a,b>0$, cf. \eqref{eq:beta-density}. The partition function $\z_{t,m}$ of the Beta-polymer, with $(t,m)\in\mathbb{Z}_{\ge 0}\times\mathbb{Z}_{\ge 1}$ and $t\ge m-1$, is determined by the recurrence relation 
$$
\z_{t,m}=\eta_{t,m}\z_{t-1,m}+(1-\eta_{t,m})\z_{t-1,m-1}
$$
and boundary conditions 
$$
\z_{t,t+1}\equiv 1,\qquad \z_{t,1}=\eta_{1,1}\eta_{2,1}\cdots\eta_{t,1}.
$$
Pictorially, $\z_{t,m}$ is a sum over all directed lattice paths with $(0,1)$ and $(1,1)$ steps that joint $(0,1)$ and $(t,m)$, of products of edge weights that all have the form $\eta_*$ or $1-\eta_*$, cf. \cite{BarCor, BGW}. 

Let us also define \emph{delayed} partition functions $\z^{(k)}_{t,m}$, $k\ge 1$, $t\ge m+k-1$, by the same recurrence relation 
$$
\z_{t,m}^{(k)}=\eta_{t,m}\z_{t-1,m}^{(k)}+(1-\eta_{t,m})\z_{t-1,m-1}^{(k)},
$$
where we are using the same family of random variable $\{\eta_{t,m}\}_{t,m\ge 1}$ to evaluate the coefficients, and shifted boundary conditions 
$$
\z_{t,t-k+2}^{(k)}\equiv 1,\qquad \z_{t,1}^{(k)}=\eta_{k,1}\eta_{k+1,1}\cdots\eta_{t,1}. 
$$
Clearly, $\z^{(1)}_{t,m}\equiv \z^{(k)}_{t,m}$, and for any $k\ge 1$, $\z_{t,m}^{(k)}$ can be interpreted graphically in a similar way to $\z_{t,m}$, but with paths jointing $(k-1,1)$ and $(t,m)$. 

As was shown in \cite[Section 7.1]{BGW}, there is a way to identify the 
continuum vertex model of Section \ref{ssec:cont-vertex} and the family of 
Beta-polymer partition functions $\{\z^{(k)}_{t,m}\}$. In order to see the 
equivalence, let us introduce the coloured height functions $\{h^{(\ge 
k)}(x,y)\mid {x\ge 2;y,k\ge 1}\}$ that count the total mass of colours $\ge k$ 
that exit vertices $(x,1),(x,2),\dots,(x,y)$, either upward or rightward, in the 
continuum model. Then one has the identification
\begin{equation}
\label{eq:beta-to-vertex}
-\log\z_{t,m}^{(k)} = h^{(\ge k)} (m+1,t),\qquad  t,m,k\ge 1,\  t\ge m+k-1.
\end{equation}
For $t,m,k\ge 1$ that do not satisfy the inequality $t\ge m+k-1$, the right-hand
side of \eqref{eq:beta-to-vertex} is readily seen to vanish, and we set 
$\z_{t,m}^{(k)}$ to 1 for these values as well; then \eqref{eq:beta-to-vertex}
holds for any $t,m,k\ge 1$.

Together with Proposition \ref{prop:limit-to-continuum}, this allows us to obtain a limiting version of Corollary \ref{cor:main-fused} for the Beta-polymer. But in order to take the corresponding limit of the integral representation, we need to introduce limiting versions of the rational functions $f_\mu$. 

\begin{lem}\label{lemma:f-lim} Take $q=\exp(-\epsilon)$, $s^2=q^\sigma$ (as in \eqref{eq:limit-parameters}), and fix a $\lambda$-coloured composition $\mu=0^{\bm{m}^{(0)}} 1^{\bm{m}^{(1)}} 2^{\bm{m}^{(2)}} \cdots$ of length $m\ge 1$, where $\lambda$ is a composition of weight $|\lambda|=m$. Then there exists a limit 
\begin{equation}
\label{eq:f-limit}
\mathfrak{f}_\mu(\lambda;u_1,\dots,u_m)=\lim_{\epsilon\to 0} \epsilon^m\cdot
\frac{f_\mu(\lambda;1+\epsilon u_1,\dots,1+\epsilon u_m)}{\prod_{j\ge 0}(s^2;q)_{|\bm{m}^{(j)}|}},
\end{equation} 
where the convergence is uniform for complex $u_1,\dots,u_m$ varying in compact 
sets that do not include $\sigma/2$, and the rational function 
$\ff_\mu(\lambda;u_1,\dots,u_m)$ can be characterized as follows. 

{\rm (i)} For a rainbow $\mu$, \ie\ for $\lambda=(1,\dots,1)$, in the anti-dominant sector $\mu_1\le\mu_2\le\dots\le\mu_m$ one has (omitting $\lambda$ from the notation)
\begin{equation}
\label{eq:f-lim-antidominant}
\ff_\mu(u_1,\dots,u_m)=\prod_{i=1}^m 
\frac1{\sigma/2-u_i}\left(
\frac{\sigma/2+u_i}{\sigma/2-u_i}
\right)^{\mu_i},
\end{equation} 
and in case $\mu_i<\mu_{i+1}$ for some $1\le i\le m-1$, one has
\begin{equation}
\begin{aligned}
\label{eq:exchange-limit}
&\mathfrak{T}_i \cdot \ff_{\mu}(u_1,\dots,u_m) = \ff_{(\mu_1,\dots,\mu_{i+1},\mu_i,\dots,\mu_m)}(u_1,\dots,u_m),\\
&\mathfrak{T}_i \equiv 1 - \frac{u_i-u_{i+1}+1}{u_i-u_{i+1}} (1-\mathfrak{s}_i),
\qquad
1 \leq i \leq m-1,
\end{aligned}
\end{equation}
with elementary transpositions $\mathfrak{s}_i \cdot h(u_1,\dots,u_m)
:=h(u_1,\dots,u_{i+1},u_i,\dots,u_m)$, cf. \eqref{eq:exchange}-\eqref{eq:DL-operators}.

{\rm (ii)} For a general colouring composition $\lambda$, let $\theta:\{1,\dots,m\}\to\{1,\dots,n\}$ be a monotone map (as in Section \ref{ssec:colour-merge}) with $\theta_*((1,\dots,1))=\lambda$. Then 
\begin{equation}
\label{eq:f-lim-merge}
\sum_{{\rm rainbow}\,  \nu\,:\, \theta_*(\nu) = \mu}
\ff_{\nu}(u_1,\dots,u_m)
= \ff_{\mu}(\lambda; u_1,\dots,u_m).
\end{equation}

\end{lem}
\begin{proof} This limit is an immediate consequence of the recursive definition
\eqref{eq:f-delta}-\eqref{eq:DL-operators} of Section \ref{ssec:recursive} and 
\eqref{eq:f-colour-merge} of Proposition \ref{prop:fG-merge}.
\end{proof}

We are now ready to take a $q\to 1$ limit in Corollary \ref{cor:main-fused}. 

\begin{prop}
\label{prop:main-beta}
Fix $\sigma>\rho>0$, and consider the partition functions $\z_{t,m}^{(k)}$ of 
the Beta-polymer as defined above with the parameters of \eqref{eq:beta-density} 
given by $(a,b)=(\sigma-\rho,\rho)$. Let 
$\mu=2^{\bm{m}^{(2)}}3^{\bm{m}^{(3)}}\cdots$ be a coloured
composition with no parts smaller than 2, $m=\ell(\mu)$ be the length of 
$\mu$, $c_1<\dots<c_\alpha$ be the colours represented in $\mu$, and  
$\fm_1,\dots,\fm_\alpha\ge 1$ be  the number of parts of $\mu$ of 
colours $c_1,\dots,c_\alpha$, respectively. Also, let $\text{colour}(\mu)$ 
be the composition that encodes the colouring of $\mu$. Then for any $t\ge \max\{i:m_i^{(x)}>0\}$ 
we have
 \begin{multline}
\label{eq:main-beta}
\E\, \prod_{x\ge 1}\displaystyle\prod_{i \ge 1}  
\frac{\left(\z^{(i+1)}_{t,x}- 
\z^{(i)}_{t,x}\right)^{m_i^{(x)}}}{{m_i^{(x)}}!}
= 
\frac{(-1)^{|\mu|}}{(2\pi\sqrt{-1})^m}\oint\cdots\oint
\prod_{1\le i<j\le m} \frac{u_j-u_i}{u_j-u_i+1} \\
\times
\prod_{k=1}^{\alpha}\left(
\sum_{j=0}^{\fm_k} \frac{(-1)^j}
{j!(\fm_k-j)!}
\prod_{p>\fm[1,{k-1}]}^{j+\fm[1,{k-1}]}\left(\frac{\sigma/2-u_p-\rho}{
\sigma/2-u_p}\right)^{t-c_k+1}
\prod_{r>j+\fm[1,{k-1}]}^{\fm[1,k]}\left(\frac{\sigma/2-u_r-\rho}{
\sigma/2-u_r}\right)^{t-c_k}
\right)
\\
\times  \ff_{\mu-1^m}({\rm colour}(\mu);-u_1,\dots,-u_m)
\prod_{i=1}^m 
{du_i},
\end{multline}
with the integration contours either positively oriented and nested around 
$\sigma/2$ with $u_{i}$-contour containing $(u_j\text{-contour})+1$ for all 
$i<j$, or negatively oriented and nested 
around $-\sigma/2$, with $u_j$-contour containing $(u_i\text{-contour})-1$ 
for all 
$i<j$. 
\end{prop}	
\begin{proof} We start with \eqref{eq:main-fused} and 
$(\lambda_1,\lambda_2,\dots)=(1,1,\dots)$; equivalently, $\ell_j=j$. 
Let us make the substitution \eqref{eq:limit-parameters} and look at the 
asymptotics of both sides. 

On the left-hand side we have averages of the observables $\O_\mu$ given 
by \eqref{eq:obs2}. The denominators are deterministic and asymptotically give
$(q;q)_{m_i^{(x)}}^{-1}\sim \epsilon^{-m_i^{(x)}}m_i^{(x)}!$ as 
$\epsilon\to 0$. For the numerators, according to Proposition \ref{prop:limit-to-continuum},
we obtain, along row $t=n$, $\epsilon^{-1} H_{\ge i}^\nu (x+1)\to h^{(\ge i)}(x+1,t)$, and, changing $\nu$ to $\nu/\mu$ with $\epsilon$-independent $\mu$, $\epsilon^{-1} H_{\ge k}^{\nu/\mu} (x+1)\to h^{(\ge k)}(x+1,t)$ weakly as $\epsilon\to 0$, where $h$ denotes the coloured height functions of the Beta-polymer. Using \eqref{eq:beta-to-vertex} and the fact that all the observables are bounded, we see that $\E\O_\mu$ is asymptotically equivalent to $\epsilon^{-m}$ times the left-hand side of \eqref{eq:main-beta}. 

For the right-hand side of \eqref{eq:main-fused}, we change the variables $y_i=1+\epsilon u_i$, use Lemma \ref{lemma:f-lim}, and also $(q;q)_{j}^{-1}(q;q)_{\fm_k-j}^{-1}\sim \epsilon^{-\fm_k} j!^{-1}(\fm_k-j)!^{-1}$. 

The powers of $\epsilon$ from the $q$-symbols on both sides cancel out, the powers of $\epsilon$ from the changes of variables and $f\to\ff$ limit also cancel out, and the prefactor $(-s)^m$ in the right-hand side of \eqref{eq:main-fused} together with $(-s)^{|\mu|-m}$ required to convert $f^\stoch_{\mu-1^m}$ to $f_{\mu-1^m}$ give $(-1)^{|{\mu}|}$  in the limit. This concludes the proof. 
\end{proof}

\subsection{Strict-weak polymer}\label{ssec:strict-weak} The strict-weak or gamma polymer was first introduced in \cite{CSS} and \cite{OCO}. Its partition functions are determined by a very similar recurrence as those for the Beta-polymer. Namely, let us define $Z_{t,m}^{(k)}$ for $(t,m,k)\in \mathbb{Z}_{\ge 0}\times\mathbb{Z}_{\ge1}\times\mathbb{Z}_{\ge 1}$,  $t\ge m+k-1$,  by 
$$
Z_{t,m}^{(k)}=\eta_{t,m}Z_{t-1,m}^{(k)}+Z_{t-1,m-1}^{(k)}
$$
with boundary conditions 
$$
Z^{(k)}_{t,t-k+2}\equiv 1, \qquad Z_{t,1}^{(k)}=\eta_{k,1}\eta_{k+1,1}\cdots \eta_{t,1},
$$
where $\{\eta_{t,m}\}_{t,m\ge 1}$ is a family of independent identically distributed random variables with a Gamma distribution that has density
$$
\frac{1}{\Gamma(\kappa)}\,x^{\kappa-1}\exp(-x),\qquad x>0, 
$$ 	
with respect to the Lebesgue measure. Here $\kappa>0$ is a parameter.

The strict-weak polymer is a limiting instance of the Beta-polymer of Section \ref{ssec:beta}, because a Beta-distributed random variable with density \eqref{eq:beta-density} and parameters $(a,b)=(\kappa,\epsilon^{-1})$, when multiplied by $\epsilon^{-1}$, converges to a Gamma-distributed random variable with parameter $\kappa$ as $\epsilon\to 0$, both in distribution and with all moments. 

In order to argue the convergence of the partition functions and their moments, we will appeal to the following 

\begin{lem}
\label{lem:moments-conv}
Let $\{X_n\}_{n\ge 1}$ and $\{Y_n\}_{n\ge 1}$ be two sequences of nonnegative random variables such that $\{(X_n,Y_n)\}_{n\ge 1}$ weakly converges to a two-dimensional random variable $(X,Y)$ with finite moments and a jointly continuous distribution function.\footnote{This requirement of continuity can be easily removed by an extra step in the proof.} Furthermore, assume that coordinate moments converge:
$$
\lim_{n\to\infty} \E X_n^k = \E X^k, \qquad \lim_{n\to\infty} \E Y_n^k = \E Y^k,\qquad k\ge 1. 
$$
Then the joint moments also converge:
$$
\lim_{n\to\infty} \E \bigl( X_n^k Y_n^l\bigr) = \E \bigl(X^k Y^l\bigr), \qquad k,l\ge 1.
$$ 
\end{lem}
\begin{proof}\footnote{We are very grateful to Vadim Gorin for providing the argument below.}
Fix $C>0$ and write
$$
\E \bigl( X_n^k Y_n^l\bigr)=\E \bigl( X_n^k Y_n^l \cdot \bm{1}\{X_n<C, Y_n<C\}\bigr)+\E \bigl( X_n^k Y_n^l \cdot \bm{1}\{X_n\ge C\text{ or }Y_n\ge C\}\bigr),
$$
where we use $\bm{1}\{A\}$ to denote the indicator function of an event $A$. In the first term, we have a bounded functional under the expectation, which converges to $\E \bigl( X^k Y^l \cdot \bm{1}\{X<C, Y<C\}\bigr)$ by the distributional convergence of $(X_n,Y_n)$. (If the distribution of $(X,Y)$ were not continuous, we would have needed to choose $C$ as its continuity point.)

Further, let us show that the second term converges to 0 as $C\to\infty$ uniformly in $n$. We have
\begin{multline*}
\E \bigl( X_n^k Y_n^l \cdot \bm{1}\{X_n\ge C\text{ or }Y_n\ge C\}\bigr)\le \E \bigl( X_n^k Y_n^l \cdot (X_n/C+Y_n/C)\bigr)=C^{-1} \E \bigl( X_n^k Y_n^l \cdot(X_n+Y_n)\bigr)\\ \le 2C^{-1}\bigl(\E X_n^{k+l+1}+\E Y_n^{k+l+1}\bigr),
\end{multline*}
where we used the inequality $x^ay^b\le x^{a+b}+y^{a+b}$ that holds for $x,y,a,b>0$. Since the moments of $X_n$ and $Y_n$ are bounded (because they converge by the hypothesis), the final expression tends to 0 as $C\to\infty$ uniformly in $n$.

As $\lim_{C\to \infty}\E \bigl( X^k Y^l \cdot \bm{1}\{X<C, Y<C\}\bigr) =\E \bigl( X^k Y^l \bigr)$, the proof is complete. 
\end{proof}

Lemma \ref{lem:moments-conv} implies, in particular, that by choosing $(\sigma-\rho,\rho)=(\kappa,\epsilon^{-1})$, we can ensure the convergence
$$
\lim_{\epsilon \to 0} \epsilon^{m+k-t-2} \cdot \z^{(k)}_{t,m}=Z^{(k)}_{t,m}, \qquad t\ge m+k-2,
$$
together will all the joint moments. This will allow us to take such a limit in Proposition \ref{prop:main-beta} momentarily, after the following analog of Lemma \ref{lemma:f-lim}. 

\begin{lem}\label{lemma:f-lim2} Fix a $\lambda$-coloured composition $\mu=0^{\bm{m}^{(0)}} 1^{\bm{m}^{(1)}} 2^{\bm{m}^{(2)}} \cdots$ of length $m\ge 1$, where $\lambda$ is a composition of weight $|\lambda|=m$. Then there exists a limit 
	\begin{equation}
	\label{eq:f-limit2}
	\fp_\mu(\lambda;v_1,\dots,v_m)=\lim_{\sigma\to \infty} (-1)^{|\mu|+m}\sigma^{-|\mu|}\cdot
	{\ff_\mu(\lambda;\sigma/2+v_1,\dots,\sigma/2+v_m)}
	\end{equation} 
	where the convergence is uniform for complex $v_1,\dots,v_m$ varying in compact 
	sets that do not include $0$, and the function 
	$\fp_\mu(\lambda;u_1,\dots,u_m)$ is a polynomial in $v_1^{-1},\ldots,v_m^{-1}$ that can be characterized as follows.
	
	 {\rm (i)} For a rainbow $\mu$, in the anti-dominant sector $\mu_1\le\mu_2\le\dots\le\mu_m$ one has (omitting $\lambda$ from the notation)
	\begin{equation}
	\label{eq:f-lim-antidominant2}
	\fp_\mu(v_1,\dots,v_m)=\prod_{i=1}^m 
    v_i^{-\mu_i-1},
	\end{equation} 
	and in case $\mu_i<\mu_{i+1}$ for some $1\le i\le m-1$ one uses the exchange relations \eqref{eq:exchange-limit} with $(\fp, v_*)$ instead of $(\ff, u_*)$.

	{\rm (ii)} For a general colouring composition $\lambda$, let $\theta:\{1,\dots,m\}\to\{1,\dots,n\}$ be a monotone map with $\theta_*((1,\dots,1))=\lambda$. Then 
	\begin{equation}
	\label{eq:f-lim-merge2}
	\sum_{{\rm rainbow}\,  \nu\,:\, \theta_*(\nu) = \mu}
	\fp_{\nu}(u_1,\dots,u_m)
	= \fp_{\mu}(\lambda; u_1,\dots,u_m).
	\end{equation}
\end{lem}
The proof of this lemma is straightforward.

\begin{prop}
\label{prop:main-strict-weak}
Fix $\kappa>0$, and consider the partition functions $Z_{t,m}^{(k)}$ of 
the strict-weak polymer as defined above with $\kappa$ being the parameter of the Gamma distribution. Let $\mu$ and the associate notation be as in Proposition \ref{prop:main-beta}. Then for any $t\ge m+\max\{i:m_i^{(x)}>0\}-1$ we have
\begin{multline}
\label{eq:main-strict-weak}
\E\, \prod_{x\ge 1}\displaystyle\prod_{i=1}^n  
\frac{\left(Z^{(i+1)}_{t,x}\right)^{m_i^{(x)}}}{{m_i^{(x)}}!}
= 
\frac{1}{(2\pi\sqrt{-1})^m}\oint\cdots\oint
\prod_{1\le i<j\le m} \frac{v_j-v_i}{v_j-v_i-1} \\
\times
\prod_{k=1}^{\alpha}
\prod_{r>\fm[1,{k-1}]}^{\fm[1,k]}\frac{(\kappa+v_r)^{t-c_k}}{
	\fm_k!}
\cdot \fp_{\mu-1^m}({\rm colour}(\mu);v_1,\dots,v_m)
\prod_{i=1}^m 
{dv_i},
\end{multline}
where the integration contours are positively oriented and nested around $0$, with $v_j$-contour containing $(v_i\text{-contour})+1$ 
for all $i<j$. 
\end{prop}
\begin{proof} We start with \eqref{eq:main-beta}, set $\sigma-\rho$ to $\kappa$, change the integration variables via $u_i=-v_i-\sigma/2$, and take the limit. Since $\z^{(i)}_{t,x}\sim \sigma^{x+i-t-2}Z^{(i)}_{t,x}$ and 
$\z^{(i+1)}_{t,x}\sim \sigma^{x+i-t-1}Z^{(i)}_{t,x}$, the second term dominates, and employing Lemma \ref{lem:moments-conv}, we obtain the convergence of the left-hand side of \eqref{eq:main-beta} to that of \eqref{eq:main-strict-weak}, with an additional power of $\sigma$ that has exponent 
$$
\sum_{i,x\ge 1} m_i^{(x)}(x+i-t-1)=|\mu|+\sum_{k=1}^\alpha c_k\fm_k-(t+1)m.
$$
On the other hand, in the integrand
we have factors of the form $\sigma/2-u_*=\sigma+v_*\sim\sigma$ in the denominator,
that make the terms with $j=0$ dominate and produce the power of $\sigma$ with the exponent
$$
\sum_{k=1}^\alpha \fm_k (c_k-t)=\sum_{k=1}^\alpha c_k\fm_k-tm. 
$$
Finally, the limit relation \eqref{eq:f-limit2} yields the power of $\sigma$ with the exponent $|\mu|-m$, thus matching the powers of $\sigma$ on both sides. All the signs cancel out, where one needs to note that we used the second choice of the contours in Proposition \ref{prop:main-beta} and changed the negative orientation to the positive one for \eqref{eq:main-strict-weak}.  
\end{proof}

\subsection{O'Connell-Yor semi-discrete Brownian polymer}\label{ssec:OY-polymer} This polymer model was first introduced in \cite{OY}. It is defined using a family $\{B_n(t)\}_{n\ge 1, t\ge 0}$ of independent standard Brownian motions. For each $n\ge 1$ and $t \ge s\ge 0$ we define its point-to-point partition function (with one of the points situated on level 1) as 
\begin{equation}
\label{eq:OY-polymer}
Z^{OY}_{(1,s)\to (n,t)}=
\int_{s=\tau_0<\tau_2<\dots<\tau_{n}=t}
\exp\left(\sum_{i=1}^{n}(B_{i}(\tau_i)-B_{i}(\tau_{i-1}))\right)d\tau_1\cdots d\tau_{n-1}.
\end{equation}

The classical functional central limit theorem and the fact that a 
Gamma-distributed random variable with large parameter $L$ divided by $L$ is 
approximately equal to 1 plus a standard normal variable divided by $\sqrt{L}$, 
yield the convergence, cf. \cite[Section 7.3]{BGW},
\begin{equation}
\label{eq:strict-weak-to-OY}
\lim_{L\to\infty} 
\frac{Z^{(Ls+1)}_{Lt,n}(\kappa=L)}{L^{L(t-s)}}=\exp\bigl({\tfrac{s-t}2}
\bigr)\cdot Z^ { OY } _ { (1 , s)\to (n,t)},\qquad t\ge s\ge 0,\quad n\ge 1, 
\end{equation}
where on the left we take the strict-weak polymer partition functions from the 
previous section with the parameter $\kappa=L$, and the convergence is in 
finite-dimensional distributions.  

\begin{prop}
\label{prop:main-OY}
Let $\mu$ and the associate notation be as in Proposition \ref{prop:main-beta}, and fix $0\le s_{1}<\dots<s_{\alpha}$. Then for any $t\ge s_{\alpha}$ we have
\begin{multline}
\label{eq:main-OY}
\E\, \prod_{x\ge 1}\displaystyle\prod_{i\ge 1}  
\frac{\left(\exp\bigl({\tfrac{s_i-t}2}
\bigr)\cdot Z^{OY}_{(1,s_i)\to(x,t)}\right)^{m_i^{(x)}}}{{m_i^{(x)}}!}
= 
\frac{1}{(2\pi\sqrt{-1})^m}\oint\cdots\oint
\prod_{1\le i<j\le m} \frac{v_j-v_i}{v_j-v_i-1} \\
\times
\prod_{k=1}^{\alpha}
\frac{\exp\left((t-s_k)\cdot\sum_{r>\fm[1,{k-1}]}^{\fm[1,k]}v_r\right)}{
	\fm_k!}
\cdot \fp_{\mu-1^m}({\rm colour}(\mu);v_1,\dots,v_m)
\prod_{i=1}^m 
{dv_i},
\end{multline}
with the integration contours are positively oriented and nested around $0$, with $v_j$-contour containing $(v_i\text{-contour})+1$ 
for all $i<j$. 
\end{prop}

\begin{proof} Let us take the limit $L\to\infty$ of \eqref{eq:main-strict-weak} with $\kappa=L$ and the coloured composition $\mu$ replaced by $\mu^{(L)}$ that has exactly the same parts, but the colours of those parts are $[s_1L],\dots,[s_\alpha L]$ instead of $c_1,\dots,c_\alpha$, respectively. 
	
Let us look at the left-hand side first. The weak joint convergence of the 
random variables
$$
\lim_{L\to\infty} 
\frac{Z^{([s_iL]+1)}_{Lt,x}(\kappa=L)}{L^{tL-[s_iL]}}=\exp\bigl({\tfrac
{s_i-t}2}
\bigr)\cdot Z^{OY}_{(1,s_i)\to (n,t)}
$$	 
for various $i$ and $x$ follows from the central limit theorem, as was 
mentioned above, and the convergence of moments of these random variables 
follows from the corresponding convergence of their integral representations, 
see 
\cite[Theorem 5.3]{CSS} for moments of the left-hand side, and 
\cite[Proposition 5.2.8]{BorodinC} for the moments of the right-hand 
side.\footnote{That convergence is, in fact, a special case of the one we are 
about to observe for single-coloured $\mu$.} Lemma 
\ref{lem:moments-conv} then shows that the left-hand side of 
\eqref{eq:main-strict-weak} divided by 
the power of $L$ with exponent 
$$
\sum_{i,x\ge 1} m_i^{(x)} (tL-[s_iL])=\sum_{k=1}^\alpha \fm_k (tL-[s_kL])
$$ 
converges to 
that of \eqref{eq:main-OY}. 

On the other hand, for the right-hand side the convergence of the $L$-dependent 
factors in the integrand is elementary:
$$
\lim_{L\to\infty} 
\frac{(L+v_r)^{tL-[s_kL]}}{L^{tL-[s_kL]}}=\exp\left((t-s_k)v_r\right),
$$
uniformly for bounded $v_r$'s, which leads to the right-hand side of 
\eqref{eq:main-OY}. 
\end{proof}

\subsection{Continuum Brownian polymer}\label{ssec:cont_Brown} One way to define partition
functions of the continuum Brownian polymer in (1+1)-dimensions is through 
solving the stochastic heat equation with multiplicative 2d white noise. More 
exactly, let $\widetilde{\mathcal{Z}}^{(y)}(t,x)$ be the unique solution of the following 
stochastic partial differential equation with the initial condition:
$$
\widetilde{\mathcal{Z}}^{(y)}_t=\tfrac12 \widetilde{\mathcal{Z}}_{xx}^{(y)}+\eta(t,x)\widetilde{\mathcal{Z}}^{(y)}, \qquad t>0,\quad 
x\in\mathbb{R}; \qquad \widetilde{\mathcal{Z}}^{(y)}(0,x)=\delta(x-y), 
$$
where $\eta=\eta(t,x)$ is the two-dimensional white noise, and set 
$$
\mathcal Z^{(y)}(t,x)=(2\pi t)^{1/2}\, e^{{(x-y)^2}/{2t}}\cdot \widetilde{\mathcal Z}^{(y)}(t,x).
$$ 
We refer to the 
survey \cite{Q} and references therein for an extensive literature on this 
equation and its close relation to continuum Brownian path integrals and the Kardar-Parisi-Zhang equation. 

The solutions $\mathcal{Z}^{(y)}(t,x)$ arise naturally as limits of the 
partition functions of the semi-discrete Brownian polymer from the previous 
section:
\begin{equation}
\label{eq:OY-to-SHE}
\lim_{L\to\infty} \frac{\exp\bigl({\tfrac{y-t\sqrt{L}}2}
\bigr)\cdot Z^{OY}_{(1,y)\to 
(tL-x\sqrt{L},t\sqrt{L})}}{\exp(tL)\cdot L^{(x\sqrt{L}-tL)/2}}=\mathcal{Z}^{
(y)}(t,x).
\end{equation}
This was essentially verified on the level of convergence of integral
representations for moments in \cite{BorodinC}, and a complete proof for convergence of finite-dimensional distributions and moments with varying $x$ was given in \cite{Nica} (in different scalings).
It is very likely that the methods of \cite{Nica} are sufficient to achieve the same result for varying $y$ as well; we will not address that here but rather focus on convergence of integral representations for joint moments instead. 
We need to start with an appropriate analog of Lemmas \ref{lemma:f-lim} and \ref{lemma:f-lim2}. 

\begin{lem}
\label{lemma:f-lim3}  Take a $\lambda$-coloured composition $\mu=0^{\bm{m}^{(0)}} 1^{\bm{m}^{(1)}} 2^{\bm{m}^{(2)}} \cdots$ of length $m\ge 1$, where $\lambda$ is a composition of weight $|\lambda|=m$, and assume that 
$$
\mu_i=tL-\kappa_i\sqrt{L}+o(\sqrt{L}) \quad \text{as }L\to\infty, \quad 1\le 
i\le m,
$$
for a fixed $m$-tuple of reals $\kappa=(\kappa_1,\dots,\kappa_m)$ and $t>0$. Then there exists a limit
\begin{equation}
\label{eq:near-crit-point}
\lim_{L\to\infty}\frac{e^{t\sqrt{L}(w_1+\dots+w_m)}\cdot
\fp_{\mu-1^m}(\sqrt{L}+w_1,\dots,\sqrt{L}+w_m)}{L^{-(\mu_1+\dots+\mu_m)/2
}}=e^{\tfrac t2(w_1^2+\dots+w_m^2)}\cdot\fe_\kappa(\lambda;w_1,\dots,w_m),
\end{equation}
uniformly for bounded $w_i$'s, where the function $\fe_\kappa(w_1,\dots,w_m)$ can be characterized as follows. 

 {\rm (i)} For a rainbow $\mu$, in the dominant sector $\kappa_1\ge\kappa_2\ge\dots\ge\kappa_m$ one has (omitting $\lambda$ from the notation)
\begin{equation}
\label{eq:f-lim-antidominant3}
\fe_\mu(w_1,\dots,w_m)=\exp(\kappa_1w_1+\dots+\kappa_m w_m),
\end{equation} 
and in case $\kappa_i>\kappa_{i+1}$ for some $1\le i\le m-1$ one uses the exchange relations \eqref{eq:exchange-limit} with $(\fe, w_*)$ instead of $(\ff, u_*)$.

{\rm (ii)} For a general colouring composition $\lambda$, let $\theta:\{1,\dots,m\}\to\{1,\dots,n\}$ be a monotone map with $\theta_*((1,\dots,1))=\lambda$. Then 
\begin{equation}
\label{eq:f-lim-merge3}
\sum_{{\rm rainbow}\,  \kappa'\,:\, \theta_*(\kappa') = \kappa}
\fe_{\kappa'}(w_1,\dots,w_m)
= \fe_{\kappa}(\lambda; w_1,\dots,w_m).
\end{equation}
\end{lem}
\begin{proof}
It suffices to check the convergence for rainbow anti-dominant $\mu$'s, as the other cases follow from that one by finite linear combinations (note that $w_i$'s are shifted but not scaled in the left-hand side of \eqref{eq:near-crit-point}). Then we need to prove that 
$$
\lim_{L\to\infty}\frac{e^{t\sqrt{L}(w_1+\dots+w_m)}
(\sqrt{L}+w_1)^{-\mu_1}\cdots(\sqrt{L}+w_m)^{-\mu_m}}{L^{-
(\mu_1+\dots+\mu_m)/2}}=e^{\tfrac t2(w_1^2+\dots+w_m^2)}\cdot 
e^{\kappa_1w_1+\dots+\kappa_m w_m}.
$$
As the relation splits into a product over $w_i$'s, it suffices to consider the case of a single variable. We have 
\begin{multline*}
-\mu_1\log(\sqrt{L}+w_1)=-\mu_1\left(\log\sqrt{L}+\left(\frac{w_1}{\sqrt{L
}}-\frac{w_1^2}{2L}+o(L^{-1})\right)\right)\\=
-\frac{\mu_1\log{(L)}}{2}-(tL-\kappa_1\sqrt{L}+o(\sqrt{L}))\left(\frac{w_1}{
\sqrt{L}}-\frac{w_1^2}{2L}+o(L^{-1})\right)\\
=\log{L^{-\mu_1/2}}-t\sqrt{L} w_1+\kappa_1w_1+\frac{tw_1^2}{2}
+o(1),
\end{multline*}
as required.
\end{proof}

We can now make a limiting statement for the moments of $\mathcal{Z}^{(y)}(t,x)$. 
\begin{prop}
\label{prop:main-SHE}
Let $\kappa=(\kappa_1,\dots,\kappa_m)$ be a coloured composition of length $m$ 
with real  coordinates, and let the colours $s_1<\dots<s_\alpha$ of the parts of 
$\kappa$ also take real values; denote 
$$
m_i^{(x)}=\#\{j:\kappa_j=x\ {\rm and\  has\  colour}\ s_i\},\qquad
\fm_i=\sum_x m_i^{(x)}, \qquad 1\le i\le \alpha, \ x\in\mathbb{R}.
$$
Then
\begin{multline}
\label{eq:main-SHE}
\E\, \prod_{(i,x)\,:\,m_i^{(x)}>0}
\frac{\left( 
\mathcal{Z}^{(s_i)}(t,x)\right)^{m_i^{(x)}}}{{m_i^{(x)}}!}
= 
\frac{1}{(2\pi\sqrt{-1})^m}\int\cdots\int
\prod_{1\le i<j\le m} \frac{w_j-w_i}{w_j-w_i-1} \\
\times
\prod_{k=1}^{\alpha}
\frac{\exp\left(-s_k\cdot\sum_{r>\fm[1,{k-1}]}^{\fm[1,k]}w_r\right)}{
	\fm_k!}
\cdot \fe_{\kappa}({\rm colour}(\kappa);w_1,\dots,w_m)
\prod_{i=1}^m e^{{tw_i^2}/{2}}{dw_i},
\end{multline}
where the integration is over upwardly oriented lines 
$w_i=a_i+\sqrt{-1}\cdot\mathbb{R}$ with 
$\Re a_j>\Re a_i +1$ for $j>i$. 
\end{prop} 
\begin{proof}[Sketch of the proof] We obtain \eqref{eq:main-SHE} as limit of 
\eqref{eq:main-OY}. The convergence of the left-hand sides was discussed below 
\eqref{eq:OY-to-SHE}. The convergence of the right-hand sides is a standard 
steepest descent argument with the main contribution coming from a finite 
neighborhood of the critical point $v=\sqrt{L}$, see the proof of 
\cite[Proposition 5.4.2]{BorodinC} for a similar situation. The change of 
variables $v_i=\sqrt{L}+w_i$, $1\le i\le m$, together with Lemma 
\ref{lemma:f-lim3}, leads to the convergence of the integrands. 
\end{proof}

\begin{rmk}\label{rem:shift-invariance} Assume that the string 
$\kappa=(\kappa_1,\dots,\kappa_m)$ can be split into three sequential 
(possibly empty) substrings $\kappa=(\kappa',\kappa'',\kappa''')$, with all the 
coordinates of $\kappa''$ being (weakly) smaller than those of $\kappa'$ and 
(weakly) larger than
those of $\kappa'''$, and with all the colours $s_i$ of the coordinates of 
$\kappa''$ being (strictly) larger than those of $\kappa'$ and 
(strictly) smaller than
those of $\kappa'''$. A special case of this situation is the dominant sector 
$\kappa_1\ge \dots\ge \kappa_m$ with rainbow compositions served by 
\eqref{eq:f-lim-antidominant3}. 

It is not hard to show from the definition of 
the $\fe_\kappa$-functions in Lemma \ref{lemma:f-lim3}, that under a simultaneous 
shift
of all the coordinates of $\kappa''$ by $\Delta$ that does not change the 
ordering conditions above, $\fe_\kappa(w_1,\dots,w_m)$ is multiplied by 
$\exp(\Delta(w_a+\dots,w_b))$, where $\kappa''=(\kappa_a,\dots,\kappa_b)$. 
Hence, if one simultaneously performs the shift of all the colours of the 
parts of $\kappa''$ by the same amount $\Delta$, then the right-hand side of 
\eqref{eq:main-SHE} is not going to change. Since this means that the moments
in the left-hand side do not change as well, it is natural to conjecture that
the joint distribution of the participating $\mathcal{Z}$'s also does not change
(the moments do not determine this distribution uniquely, though). When $\kappa''$
consists of one part, this conjecture was verified in \cite{BGW}, along
with its versions for higher models, up to coloured stochastic vertex models 
in general ``down-right'' domains. It would be very interesting to extend those results 
to $\kappa''$ consisting of more than a single part.
\end{rmk}

\bibliographystyle{alpha}
\bibliography{references}

\end{document}